\documentclass[11pt]{article}

\usepackage{amsmath}
\usepackage{esint}

   \def\XXint#1#2#3{{\setbox0=\hbox{$#1{#2#3}{\int}$}
        \vcenter{\hbox{$#2#3$}}\kern-.5\wd0}}

\usepackage{graphicx} 
\usepackage{mathtools}
\usepackage{amsthm}
\usepackage{amsfonts}
\usepackage{amssymb}
\usepackage{mathrsfs}
\usepackage{tikz}
\usepackage{subcaption}
\usepackage{placeins}

\usepackage{hyperref} 
\hypersetup{hidelinks}
\usepackage{microtype}
\usepackage{needspace}
\newtheorem{lemma}{Lemma}[section]
\newtheorem{definition}{Definition}[section]
\newtheorem{theorem}{Theorem}[section]
\newtheorem{assumption}[theorem]{Assumption}

\newtheorem{rmk}[theorem]{Remark}
\newtheorem{prop}[theorem]{Proposition}
\newtheorem{cor}[theorem]{Corollary}
\usepackage{tikz-cd}
\numberwithin{equation}{section}

\newcommand{\bJ}{\mathcal J}

\usepackage{tipa}

\usepackage[numbers]{natbib}
\usepackage{fullpage}
\usepackage{authblk}
\usepackage[titletoc,title]{appendix}

\allowdisplaybreaks
\begin{document}
\title{\textsc{Norm-resolvent asymptotics for the Neumann Laplacian on periodic resonant thin structures}}

\author[1]{Kirill Cherednichenko}
\author[1]{Chuanjie Wu}
\affil[1]{Department of Mathematical Sciences, University of Bath, Claverton Down, Bath,\qquad\qquad  BA2 7AY, United Kingdom}
\date{October 11, 2026}
\maketitle


\begin{abstract}
We consider the Neumann Laplacian on a three-dimensional periodic network
of thin rods with edge thickness of order $h$ and vertex diameter of order
$h^{2/3}$. We extend the fixed-graph resonant reduction to obtain a
finite-$h$ periodic graph approximation with generalised norm-resolvent
error $O(h)$, uniformly in the Floquet quasi-momentum. An exact gauge
transformation identifies the phase-corrected graph fibres with the
canonical lattice Floquet fibres, giving the same rate for the global
operators. Collapsing the contact points while retaining the ambient
edge tangents gives an approximation with error $O(h^{2/3})$.
The proof uses a Ryzhov boundary triple and a Schur--Feshbach reduction.
The estimates yield quantitative approximation of Bloch bands and
stability of open gaps on bounded energy intervals.
For the simple-cubic lattice, we refine the known resonant dispersion
calculation to obtain a complete band-gap decomposition and $O(h)$
control of the persistence of each fixed finite collection of gaps
in the thin structure.
\end{abstract}



\section*{Introduction}
\addcontentsline{toc}{section}{Introduction}

Metric graphs provide a natural effective description of wave propagation in
networks whose transverse dimensions are small compared with their
longitudinal scale.  The resulting one-dimensional operators, usually called
quantum-graph Hamiltonians, have a well-developed spectral theory; see, for
example, \cite{kuchment2004quantum,berkolaiko2013introduction}.  A more
delicate question concerns the conditions under which these operators arise
from multidimensional boundary-value problems on thin domains, and the sense
in which the approximation holds. For Neumann problems the limiting edge
expression is readily identified, whereas the vertex condition depends on the
relative amounts of mass and energy stored in the shrinking junction and the
adjacent channels.

Early results on thin-domain convergence include those of Rubinstein and
Schatzman \cite{rubinstein2001variational} and Kuchment and Zeng
\cite{kuchment2001convergence}, who established spectral convergence for
Neumann networks collapsing onto graphs.  Exner and Post
\cite{exner2005convergence} clarified the role of the vertex scaling in
graph-like manifolds. Their analysis distinguishes the Kirchhoff and decoupled
regimes from the borderline regime, in which the limiting operator acts on an
enlarged space and retains the vertex mass. In three dimensions an edge of
length of order one and thickness $h$ has volume of order $h^2$;
requiring the vertex volume to have the same order gives a vertex diameter of
order $h^{2/3}$. This is the resonant, or borderline, regime considered
here.

Exner and Post also treat periodic graph-like manifolds:
\cite[Theorem~9.2]{exner2005convergence} gives eigenvalue convergence
uniformly in the Floquet parameter and persistence of any fixed finite
collection of open gaps. Norm-resolvent convergence for quasi-one-dimensional
spaces, including non-compact settings under suitable uniformity assumptions,
was developed by Post \cite{post2006spectral}; see also
\cite{post2012spectral,post2021graphlike}. Grieser's survey
\cite{grieser2008thin} gives an overview of thin tubes, spectral geometry and
graph limits.

Our norm-resolvent analysis builds on the quantitative resonant reduction
of Cherednichenko, Ershova and Kiselev
\cite{cherednichenko2024neumann}.  For a thin domain collapsing to a fixed
metric graph they prove an order-sharp generalised norm-resolvent estimate; in
dimensions $d\geq3$ the error is of order $\varepsilon$, where
$\varepsilon$ is the edge thickness.  Their effective self-adjoint operator
acts in the enlarged space $L^2(G)\oplus\mathbb C^N$, where the finite-dimensional
component represents the limiting vertex mass.  After a unitary gauge, the
resulting matching law is closely related to a $\delta'$-type condition.
We retain this resonant mechanism and the local order of approximation.

To obtain the $O(h)$ norm-resolvent estimate in the periodic setting, we
control the constants over the whole Brillouin zone, retain the two incidence
phases of each geometric edge, and construct measurable fibre operators and
identification maps for the direct-integral reconstruction. The phase
analysis is needed because the magnetic rigid vertex mode has different
constant ground-mode coefficients on the distinct contact plates after
removal of the transverse magnetic phase. This gives the same
$O(h)$ rate in the global estimate and in the resulting approximation of
Bloch bands and gaps.

We treat the resonant reduction and the choice of graph gauge separately.
For each $\theta$ we first construct a phase-corrected graph operator
$\mathcal A^{\mathrm{ph}}_{\theta,h}$.  If the incidences of an oriented
edge $e$ are denoted by $(e,-)$ and $(e,+)$, its endpoint values are $\alpha^h_{e,-}c,$ $\alpha^h_{e,+}c,$ $\alpha^h_{e,\sigma}=e^{-i\theta\cdot x^h_{e,\sigma}},$ where $x^h_{e,\sigma}$ is the centre of the corresponding contact plate in
the local chart of the quotient vertex.  We prove, uniformly in $\theta$,
that the physical fibre resolvent is approximated by the lifted resolvent of
$\mathcal A^{\mathrm{ph}}_{\theta,h}$ with error $O(h)$.

Let
$n_e\in\mathbb Z^3$ be the lattice shift of the edge and let $l_e$ be its
finite-thickness centre-line length.  An affine edge gauge replaces the two
endpoint factors by the constant magnetic potential
\[
   \vartheta_{e,h}(\theta)=\frac{\theta\cdot n_e}{l_e}.
\]
Thus $\mathcal A^{\mathrm{ph}}_{\theta,h}$ is unitarily equivalent to the
standard Floquet fibre $\mathcal A^{\mathrm{lat}}_{\theta,h}$ of the
periodic metric graph: its total magnetic phase along the edge is precisely
$\theta\cdot n_e$.  The same gauge is used in the identification map and preserves the
$O(h)$ error in the global reconstruction.

A second comparison retains the ambient finite-thickness tangent in the
collapsed graph.  We denote the
resulting fibre by $\mathcal A^{\mathrm{geo}}_{\theta,h}$.  Its endpoint
values are common, but its magnetic derivative contains
$\theta\cdot t_e$, where $t_e$ is the physical edge tangent.  The
centre-line identity $n_e+x^h_{e,+}-x^h_{e,-}=l_et_e$ shows that
\[
   \frac{\theta\cdot n_e}{l_e}-\theta\cdot t_e
   =-\frac{\theta\cdot(x^h_{e,+}-x^h_{e,-})}{l_e}
   =O(h^{2/3})
\]
uniformly in $\theta$.  Hence the lattice and ambient-geometric graph
resolvents differ by $O(h^{2/3})$.  This is the error introduced by the additional geometric identification.
At $\theta=0$ all three graph fibres coincide and the physical comparison
retains the $O(h)$ scale.

Related work on periodic graph and waveguide models considers a different
order of limiting operations.  Kiselev and
Ryadovkin \cite{kiselev2024phase} study an $\varepsilon$-periodic tubular
structure already modelled by a magnetic metric graph and obtain graph-level
norm-resolvent homogenisation, including a magnetic phase transition.  Their
starting point is the graph operator, whereas the present paper first derives,
uniformly in quasi-momentum, the graph family from a three-dimensional
resonant thin structure.  In a different thin periodic regime, Delourme and
Fliss derive Dirac points, guided modes and edge states for a hexagonal
periodic graph-like domain by comparison with a honeycomb quantum graph
\cite{delourme2024guided}.  Thin Neumann waveguides with transversal windows
and related singular geometries give another class of dimension-reduction
problems; see, for example, \cite{cardone2019delta}.
Their scaling and effective mechanisms differ from the fixed-period,
vanishing-thickness geometry considered here.

Another approach derives vertex conditions from scattering and threshold
behaviour. Molchanov and Vainberg derive graph limits for networks
of thin fibres through junction scattering data and analyse the spectrum near
the transverse threshold
\cite{molchanov2007scattering,molchanov2008threshold}.
A systematic treatment of the corresponding spectral and resolvent
asymptotics is given in \cite{molchanov2010propagation}.
They also study periodic necklace networks, deriving dispersion relations
and analysing band gaps, reflection and transmission
\cite{molchanov2010periodic}.
In these two works the junction diameter is of the same order as the
fibre thickness, whereas here the larger junctions retain a non-vanishing
vertex mass in the limit.
For Dirichlet waveguides, non-trivial coupling is tied to threshold resonances;
representative results include \cite{post2005branched,dellantonio2010effective}.
Related mechanisms occur for Robin waveguides \cite{cacciapuoti2010robin}.
In these problems the effective vertex law reflects the spectral response
of the junction.

We use the Dirichlet-to-Neumann, or Weyl, function to analyse the transmission
problem. Boundary triples and Weyl functions reduce such a problem to an
operator equation on the interface; see
\cite{derkach1991generalized,ryzhov2007general}.  The asymptotic separation of
a low-dimensional resonant channel from a stiff complement underlies the
operator-theoretic approach to critical-contrast and resonant media in
\cite{cherednichenko2020effective,cherednichenko2023operatorfamilies,cherednichenko2025effectiveII}.
In the present problem the contrast is geometric: the vertex
Dirichlet-to-Neumann operator has
one zero Steklov mode, corresponding to the magnetic rigid mode, while its
remaining spectrum is separated from zero by a gap bounded below by $ch^{-1}$.

The reduction requires uniform bounds for the two Weyl cross-blocks
containing the retained channel. For straight cylindrical edges these follow
from transverse-mode invariance and the bounded endpoint
Dirichlet-to-Neumann matrix on the ground-mode sector. The full edge
Dirichlet-to-Neumann operator remains unbounded, so the Schur--Feshbach step
is carried out with the high-block domain made explicit.

The geometric and mixed-regularity assumptions are collected in
Assumption~\ref{ass:standing-geometry}. We work with one vertex orbit per
cell and finitely many straight cylindrical edge components. Their lengths
and directions may depend on $h$. Several vertex orbits would lead to a
matrix-valued retained component and are not considered here.

The main estimates are as follows.  For a compact set
$K\Subset\mathbb C\setminus\mathbb R$, we construct self-adjoint
lattice and ambient-geometric graph fibres and identification operators such
that, uniformly in $\theta\in Q'$ and $z\in K$,
\[
 \left\Vert(\mathcal A_\theta^h-z)^{-1}
 -\Phi^{\mathrm{lat}}_{\theta,h}
  (\mathcal A^{\mathrm{lat}}_{\theta,h}-z)^{-1}
  (\Phi^{\mathrm{lat}}_{\theta,h})^*\right\Vert
 \leq Ch,\quad\ 
 \left\Vert(\mathcal A_\theta^h-z)^{-1}
 -\Phi_{\theta,h}
  (\mathcal A^{\mathrm{geo}}_{\theta,h}-z)^{-1}
  \Phi_{\theta,h}^*\right\Vert
 \leq Ch^{2/3}.
\]
Both estimates extend to the global operators. The first also yields $O(h)$
approximation of fibre spectral sets and periodic band spectra on bounded energy
intervals, as well as stability of open gaps; the geometric comparison has
the corresponding $O(h^{2/3})$ control.

The final section specialises the resonant Cayley-graph calculation of
Exner and Post \cite[Section~9.7 and Theorem~9.6]{exner2005convergence} to
the simple-cubic lattice. Starting from their dispersion relation and infinite-gap result, we give the
complete band-gap decomposition,
including a scalar equation for each upper band endpoint and monotonicity of
the gaps with respect to the vertex mass. The uniform norm-resolvent estimate
then gives $O(h)$ control of the persistence of each fixed finite
collection of gaps in the thin structure.

The paper is organised as follows.  Section~1 specifies the periodic geometry
and the incidence convention.  Section~2 introduces the Floquet--Gelfand
transform and the fibre transmission problem.  Section~3 constructs the
Ryzhov triple and the Kre\u{\i}n resolvent formula.  Section~4 establishes the
mixed Dirichlet estimate, the uniform Poisson-lift bound and the
$h^{-1}$ Steklov gap.  Section~5 performs the Weyl reduction, constructs the
reduced thin operator, identifies the phase-corrected, lattice and
ambient-geometric graph fibres, proves the two-stage resolvent estimates, and
derives the band and gap consequences before reconstructing the global
operators. Section~6 gives the simple-cubic band-gap decomposition and
applies the quantitative gap-stability result.

\section{Model setup on a periodic thin structure}

\begin{figure}
    \centering
    \includegraphics[width=0.49\linewidth]{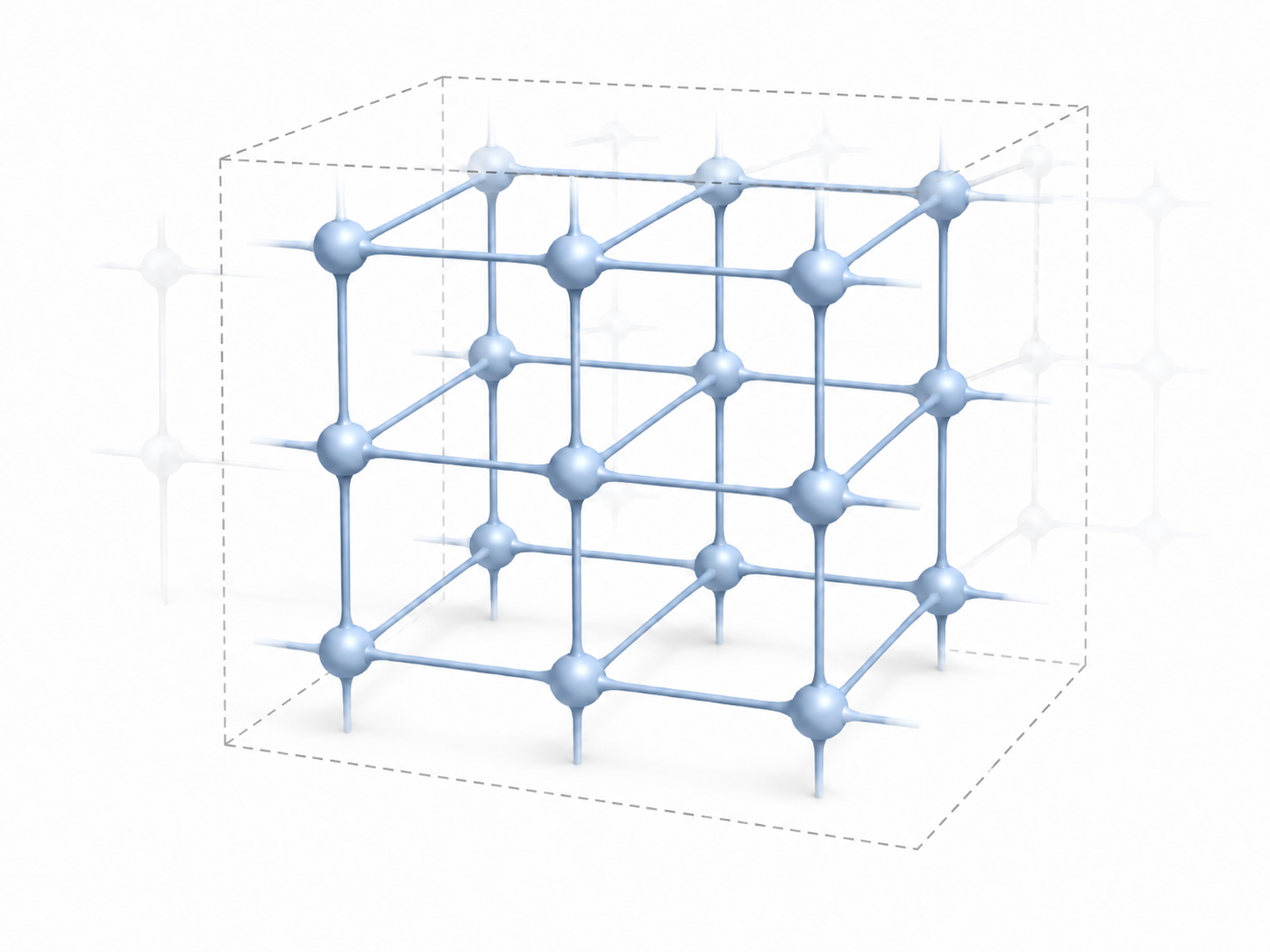}
    \includegraphics[width=0.39\linewidth]{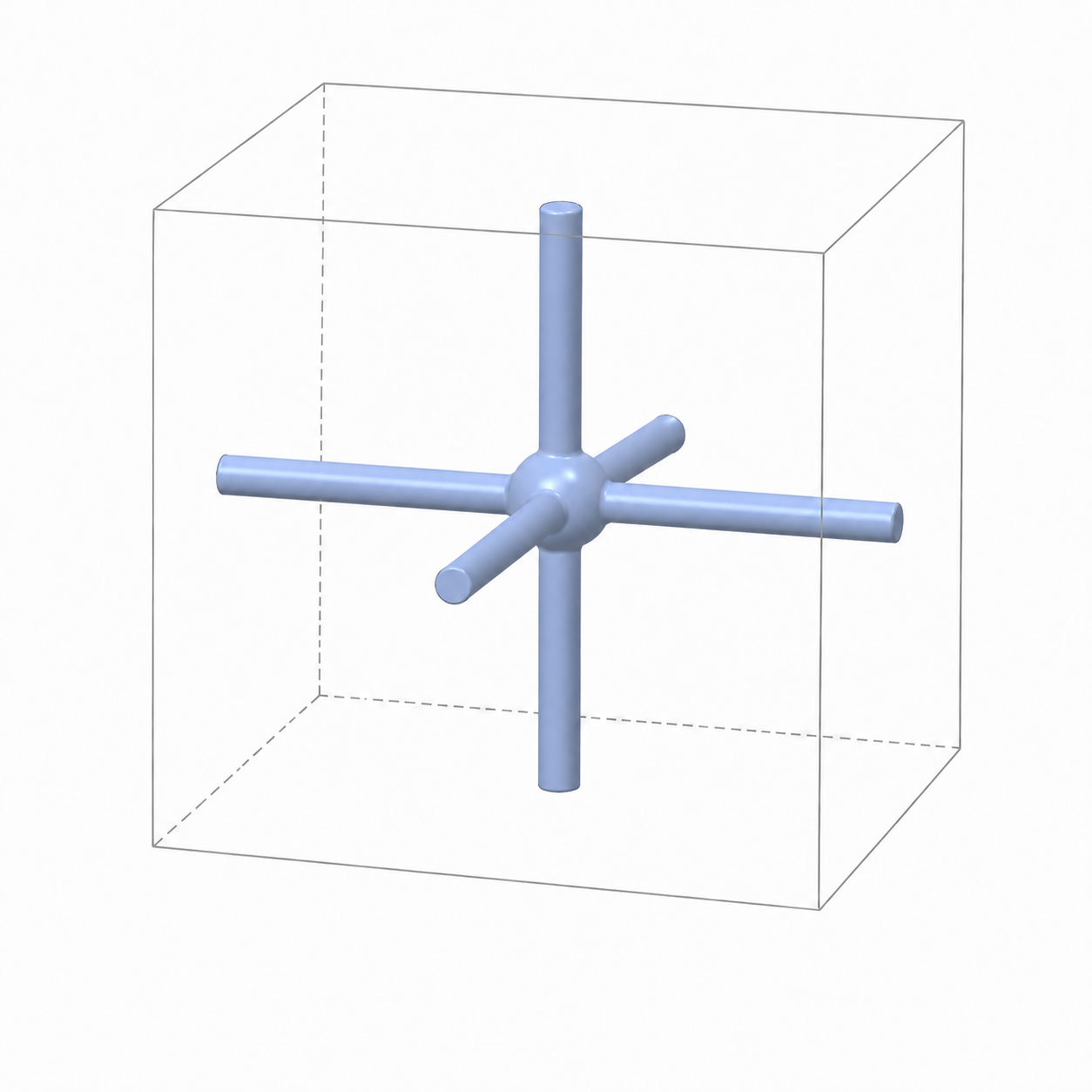}
    \caption{Three-dimensional periodic thin structure in the resonant scaling regime. 
    The edge thickness is of order $h$, whereas the diameter of each vertex neighbourhood is of order $h^{2/3}$. 
    The edge and vertex components are denoted by
    $X_e^h$ and $X_v^h$, respectively. 
    The left panel shows a portion of the periodic network, while the right panel shows
    $X_{\mathrm{cell}}^h$, the part of the thin structure contained in one unit cell.}
    \label{fig:periodic-thin-structure}
\end{figure}

We consider a three-dimensional periodic thin structure
$X^h\subset\mathbb R^3$, obtained by thickening a periodic graph. As
illustrated in Figure~\ref{fig:periodic-thin-structure}, the structure consists
of thin cylindrical edge regions $\{X_e^h\}_{e\in E}$ connected through small
three-dimensional vertex regions $\{X_v^h\}_{v\in V}$. We write
\[
    E=\{e_i\}_{i\in I},
    \qquad\quad
    V=\{v_j\}_{j\in J},
\qquad \quad 
    X^h
    =
    \operatorname{int}_{\mathbb R^3}
    \left(
    \bigcup_{e\in E}\overline{X_e^h}
    \cup
    \bigcup_{v\in V}\overline{X_v^h}
    \right).
\]
Subsequent vertex--edge decompositions are understood in this glued sense,
also in the periodic quotient.
Here $E$ and $V$ denote the sets of edges and vertices of the underlying
periodic graph. These index sets are infinite, with finitely many representatives in a
single periodicity cell. The edge regions have cross-sectional diameter of order
$h$, whereas the vertex regions have diameter of order $h^{2/3}$.

We assume throughout that the domains $X_e^h$, $X_v^h$,
and the full thin structure $X^h$ are Lipschitz, and that the rescaled
reference geometries have uniformly controlled Lipschitz character.  At each
edge--vertex connection we impose a full uniform Zaremba regularity
hypothesis.  If $D_{v,e}^h:=\partial X_v^h\cap\partial X_e^h$
denotes the contact plate between a vertex $v$ and an incident edge
$e$, then separating the vertex and edge parts at $D_{v,e}^h$
produces a local mixed problem with Dirichlet data on the contact plate and
Neumann data on the remaining physical boundary.  We assume that there are
$s_{\mathrm Z}\in(0,1/2)$ and $C_{\mathrm Z}>0$, independent of $h$,
such that, after rescaling every local chart to unit size, every function
$w\in H^1(U')$ with $\Delta w\in L^2(U')$, satisfying the
homogeneous local Dirichlet and Neumann conditions in the weak trace and
conormal sense, obeys
\begin{equation}
\label{eq:standing-zaremba-estimate}
 \Vert\chi w\Vert_{H^{3/2+s_{\mathrm Z}}(U)}
 \leq
 C_{\mathrm Z}
 \bigl(
   \Vert\Delta w\Vert_{L^2(U')}
   +\Vert w\Vert_{H^1(U')}
 \bigr)
\end{equation}
for cutoffs $\chi$ supported in $U'$ and equal to one on
$U\Subset U'$.  After the local magnetic gauge, the lower-order terms
produced by the cutoff are absorbed into the $H^1(U')$-term.  The estimate is
required in charts centred on smooth Dirichlet--Neumann junction curves and
also at every edge, corner or conical point at which the local junction
geometry changes.  In a smooth dihedral chart it follows, for example, when
the first positive mixed-boundary singular exponent is bounded below by
$1/2+s_{\mathrm Z}+\eta$ for some uniform $\eta>0$; in the presence
of corner or conical points the
corresponding operator-pencil bounds are included in the assumption.  We use
this regularity assumption to obtain the uniform conormal-trace and lifting
estimates in Subsection~4.2.

\begin{rmk}[Geometry of the rescaled vertex]\label{rmk:rescaled-vertex-geometry}
More precisely, for each vertex $v\in V$ we assume the existence of a bounded
reference domain $X_v^0\subset\mathbb R^3$, a translation vector
$b_v^h\in\mathbb R^3$, and a homeomorphism $\mathcal T_v^h:
    \overline{X_v^0}
    \to
    \overline{\mathcal T_v^h(X_v^0)}$ such that $X_v^h
    =
    b_v^h+h^{2/3}\mathcal T_v^h(X_v^0).$
The map $\mathcal T_v^h$ describes the deformation near the contact
plates. After vertex rescaling, the domain has size of order one and each
contact plate has diameter of order $\varepsilon_h=h^{1/3}$.

We assume that $\mathcal T_v^h$ is volume-preserving, so that
$|X_v^h|=h^2|X_v^0|$, and uniformly bi-Lipschitz: there exists a constant
$C>0$, independent of $h$, such that
\[
    C^{-1}|x-y|
    \le
    |\mathcal T_v^h(x)-\mathcal T_v^h(y)|
    \le
    C|x-y|,
    \qquad x,y\in \overline{X_v^0}.
\]
These deformations are required to preserve
\eqref{eq:standing-zaremba-estimate}, as specified in
Assumption~\ref{ass:standing-geometry} below.
\end{rmk}

Let $Q:=[-1/2,1/2)^3$ be the periodicity cell, and let $X_{\mathrm{cell}}^h\subset Q$ denote the
part of the thin structure contained in one cell. We assume that
$X_{\mathrm{cell}}^h$ consists of one vertex region together with finitely
many thin edge regions incident to it. Some of these edge regions meet the
boundary of $Q$. The full structure is obtained by periodic continuation:
\[
    X^h
    :=
    \bigcup_{n\in\mathbb Z^3}
    \bigl(X_{\mathrm{cell}}^h+n\bigr).
\]
Throughout the paper we assume that $X^h$ is connected.

We use the following convention for the edge geometry.  Let
$\mathscr E_{\mathrm{cell}}$ contain one chosen orientation of each
geometric edge orbit in the quotient cell.  For each
$e\in\mathscr E_{\mathrm{cell}}$, let
$n_e\in\mathbb Z^3\setminus\{0\}$ be the lattice vector joining its initial vertex to its
terminal translate.  The axis of the physical cylindrical edge joins the
contact-plate centres $x^h_{e,-}$ and $n_e+x^h_{e,+},$ where both $x^h_{e,-}$ and $x^h_{e,+}$ are represented in the same local
chart, whose origin is the quotient vertex, and satisfy $|x^h_{e,-}|+|x^h_{e,+}|\leq Ch^{2/3}.$
Thus, for a suitable edge length $l_e^h$ and unit tangent $t_e^h$,
\begin{equation}\label{eq:edge-centreline-geometry}
   n_e+x^h_{e,+}-x^h_{e,-}=l_e^h t_e^h,
   \qquad
   0<l_-\leq l_e^h\leq l_+<\infty.
\end{equation}
The superscript $h$ on $l_e^h$ and $t_e^h$ is suppressed below.  The two
incidences $(e,-)$ and $(e,+)$ are always kept distinct, even when they
represent the same quotient vertex.  We assume that each reference cross-section $\omega_e$ is connected and
satisfies $|\omega_e|=1$.  This normalisation avoids inessential
weights in the graph vertex condition; general fixed cross-sectional areas
lead to the corresponding weighted incidence maps.

\begin{assumption}[Standing geometric hypotheses]
\label{ass:standing-geometry}
The following conditions are in force throughout the paper, uniformly for all
sufficiently small $h$.
\begin{enumerate}
\item The periodic quotient has one vertex orbit and finitely many oriented
geometric edge orbits.  The structure is connected, and the vertex region is
strictly contained in the interior of the chosen periodicity cell.
\item The vertex representation in Remark~\ref{rmk:rescaled-vertex-geometry} is volume-preserving and
uniformly bi-Lipschitz.  The rescaled vertex domains are connected, have
diameter and volume comparable with one, and have Poincar\'e and trace
constants bounded independently of $h$.
\item The rescaled contact set is a disjoint union of a fixed finite number of
connected flat plates $D_j^h$, with
$
 \operatorname{diam}D_j^h\asymp\varepsilon_h
$ and $|D_j^h|\asymp\varepsilon_h^2$, where
$\varepsilon_h=h^{1/3}$.  Each plate admits a straight collar of
thickness comparable with $\varepsilon_h$, and the blown-up plates and
collars belong to a fixed bounded Lipschitz class.
\item The full local mixed-boundary estimate
\eqref{eq:standing-zaremba-estimate} holds in every relevant rescaled chart,
including charts at junction edges, corners and conical points.
\item Every edge component is an exact straight product cylinder
$(0,l_e^h)\times h\omega_e$, with
$0<l_-\leq l_e^h\leq l_+<\infty$.  The reference cross-sections
$\omega_e$ are connected bounded Lipschitz domains, form a fixed finite
family, satisfy $|\omega_e|=1$, and have their first positive Neumann
eigenvalues bounded away from zero.  The contact centres and edge axes obey
\eqref{eq:edge-centreline-geometry}.  In the periodic quotient, the edge part
is the pairwise disjoint union, up to endpoint interfaces, of these cylinders.
Distinct cylinders meet neither one another nor the vertex region except
along their prescribed full end sections $D^h_{e,-}=\{0\}\times h\omega_e,$ $D^h_{e,+}=\{l_e^h\}\times h\omega_e.$
These end sections are precisely the corresponding components of the
artificial interface $D_{\mathrm{cell}}$.
\end{enumerate}
\end{assumption}

On $X^h$ we study the resolvent problem for the Neumann Laplacian.  Fix a
compact set $K\Subset\mathbb C\setminus\mathbb R$.  Replacing
$K$ by $K\cup\overline K$, where
$\overline K:=\{\overline z:z\in K\}$, if necessary, we assume throughout
that $K=\overline K$.  Thus there are constants $\varepsilon>0$ and
$R<\infty$ such that $|\operatorname{Im}z|\geq\varepsilon,$ $|z|\leq R,$ $z\in K.$
Unless stated otherwise, constants denoted by $C$ may depend on $K$
and on the constants in Assumption~\ref{ass:standing-geometry}, but are
independent of $h$, $\theta$ and $z\in K$.
Given $f\in L^2(X^h)$ and $z\in K$, we consider
\begin{equation}
\label{eq:resolvent_problem}
    -\Delta u-zu=f
    \qquad
    \text{in }X^h,
\end{equation}
with homogeneous Neumann boundary condition on $\partial X^h$.

The problem is understood in the weak sense: given $f\in L^2(X^h)$, find
$u\in H^1(X^h)$ such that
\begin{equation}
\label{eq:weak_form}
    \int_{X^h}\nabla u\cdot\overline{\nabla\phi}\,dx
    -
    z\int_{X^h}u\,\overline{\phi}\,dx
    =
    \int_{X^h}f\,\overline{\phi}\,dx,
    \qquad
    \forall \phi\in H^1(X^h).
\end{equation}
This weak formulation is associated with the sesquilinear form
\[
    a^h(u,v)
    :=
    \int_{X^h}\nabla u\cdot\overline{\nabla v}\,dx,
    \qquad
    u,v\in H^1(X^h).
\]
The form $a^h$ is densely defined, closed, symmetric, and non-negative in
$L^2(X^h)$. Hence it defines a self-adjoint non-negative operator
$\mathcal A^h$, namely the Neumann Laplacian on $X^h$. Its domain is
\[
D(\mathcal A^h)
:=
\left\{
u\in H^1(X^h):
\exists\, w\in L^2(X^h)
\text{ such that }
a^h(u,v)=\langle w,v\rangle_{L^2(X^h)}
\quad \forall v\in H^1(X^h)
\right\},
\]
and we write $\mathcal A^h u=w$. In strong form,
$\mathcal A^h u=-\Delta u$ with Neumann boundary condition on
$\partial X^h$. Therefore \eqref{eq:resolvent_problem} can be written as $(\mathcal A^h-z)u=f$.

Since $\mathcal A^h$ is self-adjoint and non-negative, one has $\sigma(\mathcal A^h)\subset[0,\infty)$, and hence $K\subset\rho(\mathcal A^h)$. Consequently, for every $z\in K$, the resolvent $(\mathcal A^h-z)^{-1}:L^2(X^h)\to L^2(X^h)$
is well defined, and the weak problem \eqref{eq:weak_form} has the unique
solution $u=(\mathcal A^h-z)^{-1}f$.

In the following sections we study the asymptotic behaviour of this family
of resolvents as $h\to0$ and construct an effective lower-dimensional
model which approximates $(\mathcal A^h-z)^{-1}$ in operator norm,
uniformly for $z\in K$.

\section{Floquet--Gelfand transform and fibre decomposition}

The lattice translations $(\tau_m u)(x)=u(x+m)$, $m\in\mathbb Z^3$,
commute with $\mathcal A^h$. The Floquet--Gelfand transform diagonalises
this translation group and decomposes the problem into cell problems
parametrised by the quasi-momentum $\theta\in Q'=[-\pi,\pi)^3$.

\begin{definition}[Floquet--Gelfand transform]
Let $Q'=[-\pi,\pi)^3$ be the Brillouin zone. For $u\in C_c^\infty(X^h)$, the Floquet--Gelfand transform is given by
\begin{equation}\label{eq:gelfand}
     \widehat u(x,\theta)
    =
    (2\pi)^{-3/2}
    \sum_{n\in\mathbb Z^3}
    e^{-i\theta\cdot(n+x)}
    u(x+n),
    \qquad
    x\in X_{\mathrm{cell}}^h,\quad \theta\in Q' .
\end{equation}
For general $u\in L^2(X^h)$, the sum is understood in the
$L^2$-sense. Parseval's identity for the sequence
$\{u(x+n)\}_{n\in\mathbb Z^3}$, followed by integration over the cell,
gives the unitary extension
\[
    U:L^2(X^h)\to L^2\bigl(Q';L^2(X_{\mathrm{cell}}^h)\bigr),
    \qquad
    Uu=\widehat u .
\]
The inverse transform is given by
\begin{equation}
\label{eq:gelfand_inverse}
    u(x+n)
    =
    (2\pi)^{-3/2}
    \int_{Q'}
    e^{i\theta\cdot (x+n)}
    \widehat u(x,\theta)\,d\theta,
    \qquad
    x\in X_{\mathrm{cell}}^h,\quad n\in\mathbb Z^3 .
\end{equation}
For $f\in L^2(X^h)$, we similarly write $\widehat f(\cdot,\theta):=(Uf)(\cdot,\theta)$.
\end{definition}

The phase convention in \eqref{eq:gelfand} fixes the periodic gauge. Thus, for
each fixed quasi-momentum $\theta\in Q'$, the fibre function is periodic with
respect to the lattice $\widehat u(x+n,\theta)=\widehat u(x,\theta)$ for $n\in\mathbb Z^3$.
We therefore regard the fibre functions as functions on the periodicity
cell $X_{\mathrm{cell}}^h$, with opposite faces identified. With this choice of gauge, the quasi-momentum enters the differential expression rather than
quasiperiodic boundary conditions.
More precisely, the ordinary gradient is replaced by the magnetic gradient
$\nabla+i\theta$. Hence the fibre operator is given formally by $\mathcal A_\theta^h
    =
    -(\nabla+i\theta)\cdot(\nabla+i\theta)$, with homogeneous magnetic Neumann condition on the physical boundary of the thin structure inside the cell. Applying the Floquet--Gelfand transform to the resolvent equation $(\mathcal A^h-z)u=f$ therefore gives, for almost every $\theta\in Q'$,
\begin{equation}
\label{eq:fibre_problem2}
    (\mathcal A_\theta^h-z)\widehat u(\cdot,\theta)
    =
    \widehat f(\cdot,\theta).
\end{equation}

Equivalently, in weak form, for almost every $\theta\in Q'$, one seeks
$\widehat u(\cdot,\theta)\in H^1(X_{\mathrm{cell}}^h)$ such that
\begin{equation}
\label{eq:fibre-weak-form}
    \int_{X_{\mathrm{cell}}^h}
    (\nabla+i\theta)\widehat u
    \cdot
    \overline{(\nabla+i\theta)\phi}\,dx
    -
    z\int_{X_{\mathrm{cell}}^h}
    \widehat u\,\overline{\phi}\,dx
    =
    \int_{X_{\mathrm{cell}}^h}
    \widehat f\,\overline{\phi}\,dx
\end{equation}
for all $\phi\in H^1(X_{\mathrm{cell}}^h)$. Throughout the fibre
analysis, $H^1(X_{\mathrm{cell}}^h)$ denotes the Sobolev space on the
periodic quotient $X_{\mathrm{cell}}^h/\mkern-3mu\sim\,\subset\mathbb T^3$.
Thus corresponding traces are matched across opposite cell faces, and the
Neumann condition acts only on the physical boundary. The global resolvent
decomposes as
\begin{equation}
\label{eq:direct_integral}
    (\mathcal A^h-z)^{-1}
    =
    U^*
    \left(
    \int_{Q'}^\oplus
    (\mathcal A_\theta^h-z)^{-1}\,d\theta
    \right)
    U .
\end{equation}
The norm-resolvent analysis of $\mathcal A^h$ therefore reduces to
estimates uniform in both $h$ and $\theta$ for the fibre resolvents $(\mathcal A_\theta^h-z)^{-1}$ for all $\theta\in Q'$.

\medskip

We now rewrite the fibre problem as a transmission problem on one periodicity
cell. Decompose $X_{\mathrm{cell}}^h
    =
    X_{\mathrm{cell},v}^h
    \cup
    X_{\mathrm{cell},e}^h$, where $X_{\mathrm{cell},v}^h$ is the vertex region and $X_{\mathrm{cell},e}^h$ is the union of the incident edge channels inside the periodic quotient. These two parts meet along the artificial interface $D_{\mathrm{cell}}
    :=
    \partial X_{\mathrm{cell},v}^h
    \cap
    \partial X_{\mathrm{cell},e}^h$.
We write $\widehat u_v
    :=
    \widehat u|_{X_{\mathrm{cell},v}^h},
    \widehat u_e
    :=
    \widehat u|_{X_{\mathrm{cell},e}^h}$, and $\widehat f_v
    :=
    \widehat f|_{X_{\mathrm{cell},v}^h},
    \widehat f_e
    :=
    \widehat f|_{X_{\mathrm{cell},e}^h}$. The vertex lies strictly inside the cell by
Assumption~\ref{ass:standing-geometry}, so
$H^1(X_{\mathrm{cell},v}^h)$ is the ordinary vertex Sobolev space.
For the edge part, $H^1(X_{\mathrm{cell},e}^h)$ includes the periodic
trace identifications.

With this convention, the fibre space $H^1(X_{\mathrm{cell}}^h)$ can be
described as
\[
H^1(X_{\mathrm{cell}}^h)
=
\bigl\{
    (\widehat u_v,\widehat u_e)
    \in
    H^1(X_{\mathrm{cell},v}^h)
    \oplus
    H^1(X_{\mathrm{cell},e}^h)
    :
    \widehat u_v=\widehat u_e
    \text{ on }D_{\mathrm{cell}}
\bigr\}.
\]
Using this decomposition, the fibre weak formulation
\eqref{eq:fibre-weak-form} becomes: find $(\widehat u_v,\widehat u_e)
    \in
    H^1(X_{\mathrm{cell},v}^h)
    \oplus
    H^1(X_{\mathrm{cell},e}^h)$ with common trace on $D_{\mathrm{cell}}$, such that
\begin{equation}
\label{eq:transmission-weak}
\begin{aligned}
&
\int_{X_{\mathrm{cell},v}^h}
(\nabla+i\theta)\widehat u_v
\cdot
\overline{(\nabla+i\theta)\phi_v}\,dx
+
\int_{X_{\mathrm{cell},e}^h}
(\nabla+i\theta)\widehat u_e
\cdot
\overline{(\nabla+i\theta)\phi_e}\,dx
\\
&\quad
- z\int_{X_{\mathrm{cell},v}^h}
\widehat u_v\overline{\phi_v}\,dx
- z\int_{X_{\mathrm{cell},e}^h}
\widehat u_e\overline{\phi_e}\,dx
=
\int_{X_{\mathrm{cell},v}^h}
\widehat f_v\overline{\phi_v}\,dx
+
\int_{X_{\mathrm{cell},e}^h}
\widehat f_e\overline{\phi_e}\,dx
\end{aligned}
\end{equation}
for all test pairs $(\phi_v,\phi_e)
    \in
    H^1(X_{\mathrm{cell},v}^h)
    \oplus
    H^1(X_{\mathrm{cell},e}^h)$ whose traces agree on $D_{\mathrm{cell}}$.

Formally, if the solution is sufficiently smooth, integration by parts in
\eqref{eq:transmission-weak} gives the following transmission problem:
\begin{equation}
\label{eq:transmission}
\begin{cases}
-(\nabla+i\theta)\cdot(\nabla+i\theta)\widehat u_v
-z\widehat u_v
=
\widehat f_v,
& \text{in } X_{\mathrm{cell},v}^h, \\[0.5em]

-(\nabla+i\theta)\cdot(\nabla+i\theta)\widehat u_e
-z\widehat u_e
=
\widehat f_e,
& \text{in } X_{\mathrm{cell},e}^h, \\[0.5em]

\widehat u_v=\widehat u_e,
& \text{on } D_{\mathrm{cell}}, \\[0.5em]

(\nabla+i\theta)\widehat u_v\cdot \nu_v
+
(\nabla+i\theta)\widehat u_e\cdot \nu_e
=0,
& \text{on } D_{\mathrm{cell}}, \\[0.5em]

(\nabla+i\theta)\widehat u_{v(e)}\cdot n=0,
& \text{on }
\partial X_{\mathrm{cell},v(e)}^h
\setminus D_{\mathrm{cell}} .
\end{cases}
\end{equation}
Here $\nu_v$ and $\nu_e$ denote the outgoing normals to the vertex and edge
subdomains on $D_{\mathrm{cell}}$, while $n$ denotes the outward normal on
the physical boundary of the corresponding subdomain. The fourth condition in \eqref{eq:transmission} expresses flux balance across
the artificial interface. Since $D_{\mathrm{cell}}$ lies inside the thin structure, the outgoing
fluxes of the two subdomains balance. Hence \eqref{eq:transmission} is precisely the transmission
form of the fibre resolvent problem \eqref{eq:fibre_problem2}.

\section{Construction of the Ryzhov triple on the unit cell}
\label{subsec:ryzhov-triple-cell}

For fixed $\theta\in Q'$, we construct a Ryzhov triple for the fibre
transmission problem on the periodicity cell.

We set $\mathcal H
    :=
    L^2(X_{\mathrm{cell}}^h)
    =
    L^2(X^h_{\mathrm{cell},v})
    \oplus
    L^2(X^h_{\mathrm{cell},e})
    =:
    \mathcal H_v\oplus\mathcal H_e,$ $\mathcal E:=L^2(D_{\mathrm{cell}}),$ and denote by $P_v$ and $P_e$ the orthogonal projections of
$\mathcal H$ onto $\mathcal H_v$ and $\mathcal H_e$, respectively.
Here $L^2(X_{\mathrm{cell}, e}^h)$ denotes the $L^2$-space on the
periodic quotient of the edge part in the cell. We construct a Ryzhov triple $\bigl(
        \mathcal A_{0,\theta},
        \Pi_\theta,
        \Lambda_\theta
    \bigr)$ in $\mathcal H$ with boundary space $\mathcal E$. Henceforth, inner products are assumed linear in the first argument.

\paragraph{The decoupled Dirichlet operator.}

Define the sesquilinear form
\begin{equation}
\label{eq:a0-ve-theta}
    a_{0,v(e),\theta}(\hat u,\hat v)
    :=
    \int_{X^h_{\mathrm{cell},v(e)}}
    (\nabla+i\theta)\hat u\cdot
    \overline{(\nabla+i\theta)\hat v}\,dx
\end{equation}
with form domain
\[
    D(a_{0,v(e),\theta})
    :=
    H^1_{D_{\mathrm{cell}}}
    (X^h_{\mathrm{cell},v(e)})
    :=
    \left\{
        \hat u\in H^1(X^h_{\mathrm{cell},v(e)}):
        \hat u|_{D_{\mathrm{cell}}}=0
    \right\}.
\]
The associated self-adjoint operator in
$L^2(X^h_{\mathrm{cell},v(e)})$ is denoted by $\mathcal A_{0,v(e),\theta} $. It corresponds to the differential expression $-(\nabla+i\theta)\cdot(\nabla+i\theta)$ on $X^h_{\mathrm{cell},v(e)}$, with Dirichlet condition on
$D_{\mathrm{cell}}$ and Neumann condition $(\nabla+i\theta)\hat u\cdot \nu_{v(e)}=0$ on $\partial X^h_{\mathrm{cell},v(e)}\setminus D_{\mathrm{cell}}$.

More precisely, we have
\[
\begin{aligned}
D(\mathcal A_{0,v(e),\theta})
=
\bigl\{
    & \hat u\in
    H^1_{D_{\mathrm{cell}}}(X^h_{\mathrm{cell},v(e)}):
     \exists f\in L^2(X^h_{\mathrm{cell},v(e)})
    \text{ such that}                                      \\
    &\ a_{0,v(e),\theta}(\hat u,\hat v)
    =
    (f,\hat v)_{L^2(X^h_{\mathrm{cell},v(e)})} \quad 
    \forall
    \hat v\in
    H^1_{D_{\mathrm{cell}}}(X^h_{\mathrm{cell},v(e)})
\bigr\},
\end{aligned}
\]
and then $\mathcal A_{0,v(e),\theta}\hat u=f$. 

Since the interface $D_{\mathrm{cell}}$ has positive surface measure and
Dirichlet data are imposed there, the corresponding Poincaré inequality gives the existence of a constant $C^h_\theta>0$ such that
\[
    ||\hat u||_{L^2(X^h_{\mathrm{cell},v(e)})}^2
    \leq
    C_{\theta}^h
    a_{0,v(e),\theta}(\hat u,\hat u),
    \qquad
    \hat u\in
    H^1_{D_{\mathrm{cell}}}(X^h_{\mathrm{cell},v(e)}).
\]
Hence $0\in\rho(\mathcal A_{0,v(e),\theta})$.

We define the decoupled Dirichlet operator $\mathcal A_{0,\theta}
    :=
    \mathcal A_{0,v,\theta}
    \oplus
    \mathcal A_{0,e,\theta}$ in $\mathcal H_v\oplus\mathcal H_e.$
Then $\mathcal A_{0,\theta}$ is self-adjoint, non-negative, and boundedly
invertible.

\paragraph{The Ryzhov lift operator.}

We next define the Ryzhov lift operator from the artificial interface
$D_{\mathrm{cell}}$ into the vertex and edge regions. We first construct the
classical Poisson lift for sufficiently regular boundary data. For
the relatively open interface $D_{\mathrm{cell}}$, the notation
$H^{1/2}(D_{\mathrm{cell}})$ denotes the trace space obtained by restriction
from the boundary of either adjacent subdomain, equipped with its quotient
norm.  The two realisations are canonically identified on the common flat
contact plates.  For
$\phi\in H^{1/2}(D_{\mathrm{cell}})$, let $\widetilde \Pi_{v(e),\theta}\phi=\hat u_{v(e)}$, where
$\hat u_{v(e)}\in H^1(X^h_{\mathrm{cell},v(e)})$ is the unique weak solution of
\begin{equation}
\label{eq:classical-poisson-lift-ve}
\begin{cases}
-(\nabla+i\theta)\cdot(\nabla+i\theta)\hat u_{v(e)}=0,
& \text{in } X^h_{\mathrm{cell},v(e)},\\[2mm]
\hat u_{v(e)}=\phi,
& \text{on } D_{\mathrm{cell}},\\[2mm]
(\nabla+i\theta)\hat u_{v(e)}\cdot n=0,
& \text{on }
\partial X^h_{\mathrm{cell},v(e)}\setminus D_{\mathrm{cell}} .
\end{cases}
\end{equation}
Here the condition on $D_{\mathrm{cell}}$ is understood in the usual trace
sense.

Indeed, let $R_{v(e)}^h:
    H^{1/2}(D_{\mathrm{cell}})
    \to
    H^1(X^h_{\mathrm{cell},v(e)})$ be a bounded right inverse of the trace map on $D_{\mathrm{cell}}$. Writing
\[
    \hat u_{v(e)}
    =
    R_{v(e)}^h\phi+w_{v(e)},
    \qquad
    w_{v(e)}
    \in
    H^1_{D_{\mathrm{cell}}}(X^h_{\mathrm{cell},v(e)}),
\]
the weak formulation of \eqref{eq:classical-poisson-lift-ve} becomes
\[
    a_{0,v(e),\theta}(w_{v(e)},\eta)
    =
    -
    a_{0,v(e),\theta}(R_{v(e)}^h\phi,\eta),
    \qquad
    \forall
    \eta\in
    H^1_{D_{\mathrm{cell}}}(X^h_{\mathrm{cell},v(e)}).
\]
Since the Dirichlet part $D_{\mathrm{cell}}$ has positive surface measure,
the form $a_{0,v(e),\theta}$ is coercive on
$H^1_{D_{\mathrm{cell}}}(X^h_{\mathrm{cell},v(e)})$. Hence the
Lax--Milgram lemma gives a unique $w_{v(e)}$, and therefore a unique
$\hat u_{v(e)}$. Moreover,
\begin{equation}
\label{eq:H1-poisson-estimate-ve}
    ||\widetilde \Pi_{v(e),\theta}\phi||_{H^1(X^h_{\mathrm{cell},v(e)})}
    \leq
    C_{\theta}^h
    ||\phi||_{H^{1/2}(D_{\mathrm{cell}})} .
\end{equation}

The Ryzhov construction requires a lift which is bounded from
$L^2(D_{\mathrm{cell}})$ into $L^2(X^h_{\mathrm{cell},v(e)})$. We construct this
$L^2$-bounded extension by transposition. Let $f\in L^2(X^h_{\mathrm{cell},v(e)})$ and set $w_f
    :=
    (\mathcal A_{0,v(e),\theta})^{-1}f$. Then $w_f$ solves
\[
\begin{cases}
-(\nabla+i\theta)\cdot(\nabla+i\theta)w_f=f,
& \text{in } X^h_{\mathrm{cell},v(e)},\\[2mm]
w_f=0,
& \text{on } D_{\mathrm{cell}},\\[2mm]
(\nabla+i\theta)w_f\cdot n=0,
& \text{on }
\partial X^h_{\mathrm{cell},v(e)}\setminus D_{\mathrm{cell}} .
\end{cases}
\]
By the assumed mixed-boundary regularity for the Zaremba problem, together
with the uniform non-degeneracy of the Dirichlet--Neumann junctions, the
conormal trace of $w_f$ on $D_{\mathrm{cell}}$ is an $L^2$-function and
satisfies
\begin{equation}
\label{eq:adjoint-conormal-estimate-ve}
    \left\Vert
    \left.
    (\nabla+i\theta)w_f\cdot \nu_{v(e)}
    \right|_{D_{\mathrm{cell}}}
    \right\Vert_{L^2(D_{\mathrm{cell}})}
    \leq
    C_{\theta}^h
    ||f||_{L^2(X^h_{\mathrm{cell},v(e)})}.
\end{equation}
Equivalently, the operator $B_{v(e),\theta}:
    L^2(X^h_{\mathrm{cell},v(e)})
    \to
    L^2(D_{\mathrm{cell}})$ defined by
\[
    B_{v(e),\theta}f
    :=
    -
    \left.
    \bigl(
    (\nabla+i\theta)
    (\mathcal A_{0,v(e),\theta})^{-1}f
    \cdot \nu_{v(e)}
    \bigr)
    \right|_{D_{\mathrm{cell}}}
\]
is bounded.

For $\phi\in H^{1/2}(D_{\mathrm{cell}})$ and
$f\in L^2(X^h_{\mathrm{cell},v(e)})$, Green's identity gives
\begin{equation}
\label{eq:green-transposition-poisson}
    \bigl(
        \widetilde \Pi_{v(e),\theta}\phi,
        f
    \bigr)_{L^2(X^h_{\mathrm{cell},v(e)})}
    =
    \bigl(
        \phi,
        B_{v(e),\theta}f
    \bigr)_{L^2(D_{\mathrm{cell}})} .
\end{equation}
Therefore, using \eqref{eq:adjoint-conormal-estimate-ve}, we obtain
\[
\begin{aligned}
    \left|
    \bigl(
        \widetilde \Pi_{v(e),\theta}\phi,
        f
    \bigr)_{L^2(X^h_{\mathrm{cell},v(e)})}
    \right| \leq
    ||\phi||_{L^2(D_{\mathrm{cell}})}
    ||B_{v(e),\theta}f||_{L^2(D_{\mathrm{cell}})}
    \leq
    C_{\theta}^h
    ||\phi||_{L^2(D_{\mathrm{cell}})}
    ||f||_{L^2(X^h_{\mathrm{cell},v(e)})}.
\end{aligned}
\]
Taking the supremum over all
$f\in L^2(X^h_{\mathrm{cell},v(e)})$ with
$||f||_{L^2(X^h_{\mathrm{cell},v(e)})}=1$, we obtain
\begin{equation}
\label{eq:l2-poisson-estimate-ve}
    ||\widetilde \Pi_{v(e),\theta}\phi||_{L^2(X^h_{\mathrm{cell},v(e)})}
    \leq
    C_{\theta}^h
    ||\phi||_{L^2(D_{\mathrm{cell}})},
    \qquad
    \phi\in H^{1/2}(D_{\mathrm{cell}}).
\end{equation}

Since $H^{1/2}(D_{\mathrm{cell}})$ is dense in
$L^2(D_{\mathrm{cell}})$, estimate \eqref{eq:l2-poisson-estimate-ve}
implies that
\[
    \widetilde \Pi_{v(e),\theta}: L^2(D_{\mathrm{cell}}) \supset  H^{1/2}(D_{\mathrm{cell}})
    \to
    L^2(X^h_{\mathrm{cell},v(e)})
\]
extends uniquely to a bounded operator $\Pi_{v(e),\theta}:
    L^2(D_{\mathrm{cell}})
    \to
    L^2(X^h_{\mathrm{cell},v(e)})$.
Equivalently, $\Pi_{v(e),\theta}=B_{v(e),\theta}^*$. Thus, for $g\in L^2(D_{\mathrm{cell}})$, the element
$\Pi_{v(e),\theta}g\in L^2(X^h_{\mathrm{cell},v(e)})$ is characterised by
\begin{equation}
\label{eq:transposition-definition-Pi}
    \bigl(
        \Pi_{v(e),\theta}g,
        f
    \bigr)_{L^2(X^h_{\mathrm{cell},v(e)})}
    =
    \bigl(
        g,
        B_{v(e),\theta}f
    \bigr)_{L^2(D_{\mathrm{cell}})}
    \qquad
    \forall f\in L^2(X^h_{\mathrm{cell},v(e)}).
\end{equation}
If $g\in H^{1/2}(D_{\mathrm{cell}})$, this definition agrees with the
classical weak solution of \eqref{eq:classical-poisson-lift-ve}. For general $g\in L^2(D_{\mathrm{cell}})$, the element
$\Pi_{v(e),\theta}g$ need not belong to
$H^1(X^h_{\mathrm{cell},v(e)})$, and need not have an
ordinary Sobolev trace on $D_{\mathrm{cell}}$. We therefore interpret
$\Pi_{v(e),\theta}g$ by transposition.

Next, we define the coupled Ryzhov lift by $\Pi_\theta g
    :=
    \Pi_{v,\theta}g\oplus \Pi_{e,\theta}g$ for all $g\in L^2(D_{\mathrm{cell}})$. Thus $\Pi_\theta:
    \mathcal E=L^2(D_{\mathrm{cell}})
    \to
    \mathcal H$ is bounded.

We verify the structural properties of the lift, beginning with injectivity. Since the
transposition Poisson operator is defined by $\Pi_{v(e),\theta}=B_{v(e),\theta}^*$, we have
\[
    \ker \Pi_{v(e),\theta}
    =
    \ker B_{v(e),\theta}^*
    =
    \bigl(\overline{\operatorname{Ran}B_{v(e),\theta}}\bigr)^\perp.
\]
It is therefore enough to show that
$\overline{\operatorname{Ran}B_{v(e),\theta}}=L^2(D_{\mathrm{cell}})$.

For density, it suffices to show that every smooth function compactly
supported in the interior of the contact plate belongs to
$\operatorname{Ran}B_{v(e),\theta}$. Let
$\psi\in C_c^\infty(D_{\mathrm{cell}}^\circ)$. Choose a collar
neighbourhood of $D_{\mathrm{cell}}^\circ$ inside
$X_{\mathrm{cell},v(e)}^h$, with local coordinates $x=y-t\,\nu_{v(e)}(y)$, where
$y\in D_{\mathrm{cell}}^\circ$ and $0<t<\varepsilon$, $\nu_{v(e)}$ denotes the outward unit normal to $X_{\mathrm{cell},v(e)}^h$ on $D_{\mathrm{cell}}$. Let $\chi\in C^\infty([0,\varepsilon])$ have support in $[0,\varepsilon)$, satisfy $\chi(0)=1$, and define
$w(y,t):=t\chi(t)\psi(y)$ in the collar, extending $w$ by zero outside a slightly larger collar. Then $w\in D(\mathcal A_{0,v(e),\theta})$, $w|_{D_{\mathrm{cell}}}=0$, and $-\left.
    (\nabla+i\theta)w\cdot \nu_{v(e)}
    \right|_{D_{\mathrm{cell}}}
    =
    \psi$. 

Set $f:=\mathcal A_{0,v(e),\theta}w\in L^2(X_{\mathrm{cell},v(e)}^h)$. Since $w=(\mathcal A_{0,v(e),\theta})^{-1}f$, the definition of
$B_{v(e),\theta}$ gives
\[
    B_{v(e),\theta}f
    =
    -
    \left.
    (\nabla+i\theta)
    (\mathcal A_{0,v(e),\theta})^{-1}f
    \cdot \nu_{v(e)}
    \right|_{D_{\mathrm{cell}}}
    =
    -
    \left.
    (\nabla+i\theta)w\cdot \nu_{v(e)}
    \right|_{D_{\mathrm{cell}}}
    =
    \psi .
\]
Therefore $C_c^\infty(D_{\mathrm{cell}}^\circ)
    \subset
    \operatorname{Ran}B_{v(e),\theta}$. Since $\partial D_{\mathrm{cell}}$ has surface measure zero, $C_c^\infty(D_{\mathrm{cell}}^\circ)$ is dense in
$L^2(D_{\mathrm{cell}})$. Hence $\overline{\operatorname{Ran}B_{v(e),\theta}}
    =
    L^2(D_{\mathrm{cell}})$. Consequently,
\[
    \ker \Pi_{v(e),\theta}
    =
    \ker B_{v(e),\theta}^*
    =
    \bigl(\overline{\operatorname{Ran}B_{v(e),\theta}}\bigr)^\perp
    =
    \{0\}.
\]
Since the coupled lift is $\Pi_\theta g
    =
    \Pi_{v,\theta}g\oplus \Pi_{e,\theta}g$, it follows in particular that $\ker \Pi_\theta=\{0\}$.

For the second structural condition, we prove that
\[
    D(\mathcal A_{0,v(e),\theta})
    \cap
    \operatorname{Ran}\Pi_{v(e),\theta}
    =
    \{0\}.
\]
Indeed, suppose that $\hat u_{v(e)}
    =
    \Pi_{v(e),\theta}g$
belongs to $D(\mathcal A_{0,v(e),\theta})$, for some
$g\in L^2(D_{\mathrm{cell}})$.  The transposition solution is homogeneous
in the interior.  To see this, let
$\varphi\in C_c^\infty(X^h_{\mathrm{cell},v(e)})$ and put
$f:=\mathcal A_{0,v(e),\theta}\varphi$.  Since $\varphi$ is supported
away from $D_{\mathrm{cell}}$, one has
$B_{v(e),\theta}f=0$.  The transposition identity
\eqref{eq:transposition-definition-Pi} therefore gives
\[
    \bigl(
        \hat u_{v(e)},
        \mathcal A_{0,v(e),\theta}\varphi
    \bigr)_{L^2(X^h_{\mathrm{cell},v(e)})}
    =0.
\]
Thus $-(\nabla+i\theta)\mathbin{\cdot}(\nabla+i\theta)\hat u_{v(e)}=0$
in the distributional sense in $X^h_{\mathrm{cell},v(e)}$.  Since
$\hat u_{v(e)}\in D(\mathcal A_{0,v(e),\theta})$, it follows that
$\mathcal A_{0,v(e),\theta}\hat u_{v(e)}=0$.  But
$0\in\rho(\mathcal A_{0,v(e),\theta})$, and hence
$\hat u_{v(e)}=0$.  Injectivity of $\Pi_{v(e),\theta}$ then gives
$g=0$.

Summing over the vertex and edge components gives $D(\mathcal A_{0,\theta})
    \cap
    \operatorname{Ran}\Pi_\theta
    =
    \{0\}$. Therefore $\Pi_\theta$ satisfies the structural conditions required in the
Ryzhov construction.

\paragraph{The Dirichlet-to-Neumann operator.}

We first define the fundamental Dirichlet-to-Neumann map in the natural
duality scale. For $v(e)\in\{v,e\}$, define the form
\[
    \lambda_{v(e),\theta}:
    H^{1/2}(D_{\mathrm{cell}})
    \times
    H^{1/2}(D_{\mathrm{cell}})
    \to
    \mathbb C
\]
by
\begin{equation}
\label{eq:lambda-ve-theta}
    \lambda_{v(e),\theta}(g,\psi)
    :=
    \int_{X^h_{\mathrm{cell},v(e)}}
    (\nabla+i\theta)
    \widetilde \Pi_{v(e),\theta} g
    \cdot
    \overline{
    (\nabla+i\theta)
    \widetilde \Pi_{v(e),\theta} \psi
    }
    \,dx .
\end{equation}
Here $H^{-1/2}(D_{\mathrm{cell}})$ denotes the anti-dual of the restriction
trace space $H^{1/2}(D_{\mathrm{cell}})$, with $L^2(D_{\mathrm{cell}})$
as pivot space.

In the Zaremba setting the conormal trace of
$\widetilde \Pi_{v(e),\theta} g$ need not belong to
$L^2(D_{\mathrm{cell}})$. We therefore interpret it as
\[
    \lambda_{v(e),\theta}(g,\psi)
    =
    \Bigl\langle
        \left.
        (\nabla+i\theta)
        \widetilde \Pi_{v(e),\theta} g
        \cdot \nu_{v(e)}
        \right|_{D_{\mathrm{cell}}},
        \psi
    \Bigr\rangle_{H^{-1/2}(D_{\mathrm{cell}}),H^{1/2}(D_{\mathrm{cell}})} .
\]
The norm equivalence proved in Theorem~\ref{thm:lambda-coercivity-compact-resolvent}
below shows, in particular, that these forms are densely defined, symmetric,
non-negative and closed in $L^2(D_{\mathrm{cell}})$.

Define the fundamental positive Dirichlet-to-Neumann map $\mathfrak T_{v(e),\theta}:
    H^{1/2}(D_{\mathrm{cell}})
    \to
    H^{-1/2}(D_{\mathrm{cell}})$ by
\begin{equation}
\label{eq:fundamental-DtN}
    \left\langle
        \mathfrak T_{v(e), \theta} g,\psi
    \right\rangle_{H^{-1/2},H^{1/2}}
    =
    \lambda_{v(e),\theta}(g,\psi),
    \qquad
    g,\psi\in H^{1/2}(D_{\mathrm{cell}}).
\end{equation}
Equivalently,
\[
    \mathfrak T_{v(e),\theta} g
    =
    \left.
    (\nabla+i\theta)
    \widetilde \Pi_{v(e),\theta} g
    \cdot \nu_{v(e)}
    \right|_{D_{\mathrm{cell}}},
\]
where the equality is understood in
$H^{-1/2}(D_{\mathrm{cell}})$.

For the Ryzhov triple, we use the $L^2$-realisation of this
fundamental Dirichlet-to-Neumann map. Define the non-positive $L^2$-realisation
$\Lambda_{v(e),\theta}$ by
\[
D(\Lambda_{v(e),\theta})
    =
    \bigl\{
        g\in H^{1/2}(D_{\mathrm{cell}}):
        -\mathfrak T_{v(e),\theta} g\in L^2(D_{\mathrm{cell}})
    \bigr\}.
\]
Equivalently,
\[
\begin{aligned}
D(\Lambda_{v(e),\theta})
=
\bigl\{
    g\in H^{1/2}(D_{\mathrm{cell}}):
    &\ \exists l\in L^2(D_{\mathrm{cell}})
    \text{ such that}                                     \\
    &\
    \lambda_{v(e),\theta}(g,\psi)
    =
    (l,\psi)_{L^2(D_{\mathrm{cell}})}
    \quad
    \forall \psi\in H^{1/2}(D_{\mathrm{cell}})
\bigr\},
\end{aligned}
\]
and then $\Lambda_{v(e),\theta}g=-l$.  By Kato's representation theorem,
$\Lambda_{v(e),\theta}$ is self-adjoint and non-positive in
$L^2(D_{\mathrm{cell}})$.  For
$g\in D(\Lambda_{v(e),\theta})$ and
$\psi\in H^{1/2}(D_{\mathrm{cell}})$, one has
\[
 -(\Lambda_{v(e),\theta}g,\psi)_{L^2(D_{\mathrm{cell}})}
 =\lambda_{v(e),\theta}(g,\psi).
\]
Equivalently, whenever the conormal traces are represented by
$L^2(D_{\mathrm{cell}})$-functions,
\[
    \Lambda_{v(e),\theta} g
    =
    -
        (\nabla+i\theta)
        \widetilde \Pi_{v(e),\theta} g
        \cdot \nu_{v(e)}|_{D_{\mathrm{cell}}}
   .
\]

The following theorem establishes compactness of the resolvent of
$\Lambda_{v(e),\theta}$.
\begin{theorem}
\label{thm:lambda-coercivity-compact-resolvent}
Assume the standing geometric hypotheses in
Assumption~\ref{ass:standing-geometry}, in particular the full local
mixed-boundary estimate \eqref{eq:standing-zaremba-estimate} at smooth
junction curves and at every edge, corner or conical point. Then, for
each fixed $\theta\in Q'$, there exist constants $c_\theta^h>0$ and
$C_\theta^h>0$ such that
\begin{equation}
\label{eq:lambda-coercive}
    c_\theta^h
    ||g||_{H^{1/2}(D_{\mathrm{cell}})}^2
    \leq
    \lambda_{v(e),\theta}(g,g)
    +
    ||g||_{L^2(D_{\mathrm{cell}})}^2
    \leq
    C_\theta^h
    ||g||_{H^{1/2}(D_{\mathrm{cell}})}^2 ,
    \qquad
    g\in H^{1/2}(D_{\mathrm{cell}}).
\end{equation}
Consequently, the shifted form $\lambda_{v(e),\theta}(\cdot,\cdot)
    +
    (\cdot,\cdot)_{L^2(D_{\mathrm{cell}})}$ is coercive on $H^{1/2}(D_{\mathrm{cell}})$. Moreover,
\[
    (-\Lambda_{v(e),\theta}+I)^{-1}:
    L^2(D_{\mathrm{cell}})
    \to
    L^2(D_{\mathrm{cell}})
\]
is compact. Hence $\Lambda_{v(e),\theta}$ has compact resolvent and purely
discrete spectrum in $L^2(D_{\mathrm{cell}})$.
\end{theorem}

\begin{proof}
Let $\hat u=\widetilde\Pi_{v(e),\theta}g$. Since $\hat u$ minimises
the magnetic energy among functions with trace $g$, the bounded right
inverse of the trace gives
\[
 \lambda_{v(e),\theta}(g,g)
 \leq \Vert(\nabla+i\theta)R^h_{v(e)}g\Vert_{L^2}^2
 \leq C_\theta^h\Vert g\Vert_{H^{1/2}(D_{\mathrm{cell}})}^2.
\]
This proves the upper bound in \eqref{eq:lambda-coercive}. Conversely, the
trace theorem and the magnetic trace--Poincar\'e inequality give
\[
 \Vert g\Vert_{H^{1/2}(D_{\mathrm{cell}})}^2
 \leq C^h\Vert\hat u\Vert_{H^1}^2
 \leq C_\theta^h\bigl(
    \lambda_{v(e),\theta}(g,g)
    +\Vert g\Vert_{L^2(D_{\mathrm{cell}})}^2\bigr),
\]
which is the lower bound.

If $q=(-\Lambda_{v(e),\theta}+I)^{-1}g$, its form equation tested
against $q$, together with \eqref{eq:lambda-coercive}, yields
\[
 \Vert q\Vert_{H^{1/2}(D_{\mathrm{cell}})}
 \leq C_\theta^h\Vert g\Vert_{L^2(D_{\mathrm{cell}})}.
\]
Thus the shifted resolvent maps into $H^{1/2}(D_{\mathrm{cell}})$, which
is compactly embedded in $L^2(D_{\mathrm{cell}})$. This proves the
assertion.
\end{proof}

We define the coupled Dirichlet-to-Neumann operator through the sum of the
component forms. Set
\[
    \lambda_\theta(g,\psi)
    :=
    \lambda_{v,\theta}(g,\psi)
    +
    \lambda_{e,\theta}(g,\psi),
    \qquad
    g,\psi\in H^{1/2}(D_{\mathrm{cell}}).
\]
Thus the form domain of the coupled Dirichlet-to-Neumann form is $D(\lambda_\theta)=H^{1/2}(D_{\mathrm{cell}})$. Since each of the forms $\lambda_{v,\theta}$ and $\lambda_{e,\theta}$ is densely defined, closed, symmetric, and non-negative in $L^2(D_{\mathrm{cell}})$, their sum
$\lambda_\theta$ has the same properties. In particular, by the first
representation theorem, there exists a unique self-adjoint non-positive
operator $\Lambda_\theta$ in $L^2(D_{\mathrm{cell}})$ associated with the form $\lambda_\theta$, in the sense that
\[
    -(\Lambda_\theta g,\psi)_{L^2(D_{\mathrm{cell}})}
    =
    \lambda_\theta(g,\psi),
    \qquad
    g\in D(\Lambda_\theta),\quad
    \psi\in H^{1/2}(D_{\mathrm{cell}}).
\]
Equivalently, its operator domain is
\[
\begin{aligned}
D(\Lambda_\theta)
=
\bigl\{
 g\in H^{1/2}(D_{\mathrm{cell}}):\ &\exists\,l\in L^2(D_{\mathrm{cell}})
 \text{ such that}\\
 &\lambda_\theta(g,\psi)
 =(l,\psi)_{L^2(D_{\mathrm{cell}})}
 \quad\forall\psi\in H^{1/2}(D_{\mathrm{cell}})
\bigr\}.
\end{aligned}
\]
Then $\Lambda_\theta g=-l$. In this sense, $\Lambda_\theta$ is the $L^2(D_{\mathrm{cell}})$-realisation of the negative total Dirichlet-to-Neumann map. Formally, whenever the conormal traces are represented by $L^2(D_{\mathrm{cell}})$-functions, one has
\[
    \Lambda_\theta g
    =
    -
    \Bigl(
        (\nabla+i\theta)\widetilde\Pi_{v,\theta}g\cdot \nu_v
        \big|_{D_{\mathrm{cell}}}
        +
        (\nabla+i\theta)\widetilde\Pi_{e,\theta}g\cdot \nu_e
        \big|_{D_{\mathrm{cell}}}
    \Bigr).
\]
By Theorem~\ref{thm:lambda-coercivity-compact-resolvent} and the non-negativity
of the component forms, the coupled form satisfies
\[
    c_\theta^h||q||_{H^{1/2}(D_{\mathrm{cell}})}^2
    \leq
    \lambda_\theta(q,q)+||q||_{L^2(D_{\mathrm{cell}})}^2
    \leq
    C_\theta^h||q||_{H^{1/2}(D_{\mathrm{cell}})}^2 .
\]
Hence $(-\Lambda_\theta+I)^{-1}$ maps $L^2(D_{\mathrm{cell}})$ boundedly
into $H^{1/2}(D_{\mathrm{cell}})$. Since the embedding
$H^{1/2}(D_{\mathrm{cell}})\hookrightarrow L^2(D_{\mathrm{cell}})$ is compact,
it follows that $(-\Lambda_\theta+I)^{-1}:L^2(D_{\mathrm{cell}})
    \to L^2(D_{\mathrm{cell}})$ is compact.

\paragraph{The maximal operator and the boundary maps.}

The abstract maximal operator associated with the Ryzhov triple is defined by $D(\mathcal A_{\max,\theta})
    :=
    D(\mathcal A_{0,\theta})
    \dotplus
    \operatorname{Ran}\Pi_\theta.$
Every element $\hat u\in D(\mathcal A_{\max,\theta})$ has a unique
decomposition
\[
    \hat u
    =
    (\mathcal A_{0,\theta})^{-1}f
    +
    \Pi_\theta\phi,
    \qquad
    f\in\mathcal H,
    \quad
    \phi\in\mathcal E .
\]
We set
\begin{equation}
\label{eq:Amax-definition}
    \mathcal A_{\max,\theta}
    \left(
        (\mathcal A_{0,\theta})^{-1}f
        +
        \Pi_\theta\phi
    \right)
    :=
f .
\end{equation}
Thus, formally, $\mathcal A_{\max,\theta}\hat u
    =
    -(\nabla+i\theta)\cdot (\nabla+i\theta)\hat u$
on $X^h_{\mathrm{cell},v}$ and $X^h_{\mathrm{cell},e}$. In the Lipschitz
Zaremba setting we use the definition \eqref{eq:Amax-definition}.

The first boundary map 
\begin{equation}
\label{eq:Gamma0}
 \Gamma_{0,\theta}:
    D(\mathcal A_{\max,\theta})
    \to
    \mathcal E, \qquad \textup{is defined by} \qquad 
    \Gamma_{0,\theta}
    \left(
        (\mathcal A_{0,\theta})^{-1}f
        +
        \Pi_\theta\phi
    \right)
    :=
    \phi .
\end{equation}
Thus $\Gamma_{0,\theta}$ records the common trace on the interface
$D_{\mathrm{cell}}$.

The second boundary map is defined on $D(\Gamma_{1,\theta})
    :=
    D(\mathcal A_{0,\theta})
    \dotplus
    \Pi_\theta D(\Lambda_\theta)$
by
\begin{equation}
\label{eq:Gamma1}
    \Gamma_{1,\theta}
    \left(
        (\mathcal A_{0,\theta})^{-1}f
        +
        \Pi_\theta\phi
    \right)
    :=
    (\Pi_\theta)^*f
    +
    \Lambda_\theta\phi ,
    \qquad
    f\in\mathcal H,
    \quad
    \phi\in D(\Lambda_\theta).
\end{equation}
The transposition interpretation of $(\Pi_\theta)^*$ and the negative
Dirichlet-to-Neumann convention give the following conormal representation. If $\hat u=\hat u_v\oplus \hat u_e
    \in D(\Gamma_{1,\theta})$ is sufficiently regular, then
\begin{equation}
\label{eq:Gamma1-flux}
    \Gamma_{1,\theta}\hat u
    =
    -
    \left(
        \left.
        (\nabla+i\theta)\hat u_v\cdot \nu_v
        \right|_{D_{\mathrm{cell}}}
        +
        \left.
        (\nabla+i\theta)\hat u_e\cdot \nu_e
        \right|_{D_{\mathrm{cell}}}
    \right).
\end{equation}
For later use, we also introduce the component boundary maps.  For
$j\in\{v,e\}$, write
\[
 D(\mathcal A_{\max,j,\theta})
 :=D(\mathcal A_{0,j,\theta})\dotplus\operatorname{Ran}\Pi_{j,\theta}.
\]
On this space define $\Gamma_{0,j,\theta}$ by $\Gamma_{0,j,\theta}
 \bigl(\mathcal A_{0,j,\theta}^{-1}f_j+\Pi_{j,\theta}\phi\bigr)
 :=\phi.$
On
$D(\mathcal A_{0,j,\theta})\dotplus
\Pi_{j,\theta}D(\Lambda_{j,\theta})$, define
\[
 \Gamma_{1,j,\theta}
 \bigl(\mathcal A_{0,j,\theta}^{-1}f_j+\Pi_{j,\theta}\phi\bigr)
 :=\Pi_{j,\theta}^*f_j+\Lambda_{j,\theta}\phi.
\]
On the common-trace domain the two component values of $\Gamma_0$
agree with $\Gamma_{0,\theta}$. The identity
$\Gamma_{1,\theta}=\Gamma_{1,v,\theta}+\Gamma_{1,e,\theta}$ holds as an
$\mathcal E$-valued equality for common traces in
$D(\Lambda_{v,\theta})\cap D(\Lambda_{e,\theta})$.

The maps $\Gamma_{0,\theta}$ and $\Gamma_{1,\theta}$ satisfy the abstract Green identity
\begin{equation}
\label{eq:abstract-green-identity}
\begin{aligned}
    \bigl(
        \mathcal A_{\max,\theta}\hat u,\hat v
    \bigr)_{\mathcal H}
    -
    \bigl(
        \hat u,\mathcal A_{\max,\theta}\hat v
    \bigr)_{\mathcal H}
    =
    \bigl(
        \Gamma_{1,\theta}\hat u,
        \Gamma_{0,\theta}\hat v
    \bigr)_{\mathcal E}
    -
    \bigl(
        \Gamma_{0,\theta}\hat u,
        \Gamma_{1,\theta}\hat v
    \bigr)_{\mathcal E},
    \qquad
    \hat u,\hat v\in D(\Gamma_{1,\theta}).
\end{aligned}
\end{equation}
Thus $\bigl(
        \mathcal A_{0,\theta},
        \Pi_\theta,
        \Lambda_\theta
    \bigr)$ is a Ryzhov triple, and $\bigl(
        \mathcal H,
        \mathcal E,
        \Gamma_{0,\theta},
        \Gamma_{1,\theta}
    \bigr)$ is the corresponding boundary structure.

\paragraph{The solution operator and the Weyl function.}

For $z\in\rho(\mathcal A_{0,\theta})$, define the solution operator
\begin{equation}
\label{eq:solution-operator}
    S_{\theta,z}
    :=
    \left(
        I-z(\mathcal A_{0,\theta})^{-1}
    \right)^{-1}
    \Pi_\theta .
\end{equation}
Equivalently, $S_{\theta,z}
    =
    \Pi_\theta
    +
    z
    (\mathcal A_{0,\theta}-z)^{-1}
    \Pi_\theta$. It satisfies $(\mathcal A_{\max,\theta}-z)S_{\theta,z}\phi=0$ and $\Gamma_{0,\theta}S_{\theta,z}\phi=\phi$. Hence $S_{\theta,z}\phi$ is the solution of the homogeneous interior
problem with prescribed common trace $\phi$ on the artificial interface. More
precisely, writing $S_{\theta,z}\phi
    =
    \hat u_v\oplus \hat u_e$,
one has
\begin{equation}\label{BVPsolutionop}
\begin{cases}
-(\nabla+i\theta)\cdot (\nabla+i\theta)\hat u_{v(e)}
-z\hat u_{v(e)}=0,
& \text{in } X_{\mathrm{cell},v(e)}^h,\\[2mm]
\hat u_v=\phi=\hat u_e,
& \text{on } D_{\mathrm{cell}},\\[2mm]
(\nabla+i\theta)\hat u_{v(e)}\cdot \nu_{v(e)}=0,
& \text{on }
\partial X_{\mathrm{cell},v(e)}^h\setminus D_{\mathrm{cell}}.
\end{cases}
\end{equation}

The Weyl function is defined by
\begin{equation}
\label{eq:Weyl-function}
    M_\theta(z)\phi
    :=
    \Gamma_{1,\theta}S_{\theta,z}\phi,
    \qquad
    \phi\in D(\Lambda_\theta).
\end{equation}
Using \eqref{eq:Gamma1} and \eqref{eq:solution-operator}, one obtains
\begin{equation}
\label{eq:Weyl-explicit}
    M_\theta(z)
    =
    \Lambda_\theta
    +
    z
    (\Pi_\theta)^*
    \left(
        I-z(\mathcal A_{0,\theta})^{-1}
    \right)^{-1}
    \Pi_\theta .
\end{equation}
Equivalently, $M_\theta(z)
    =
    M_{v,\theta}(z)
    +
    M_{e,\theta}(z)$, where, in the natural operator or form sense,
\[
    M_{v(e),\theta}(z)
    =
    \Lambda_{v(e),\theta}
    +
    z
    (\Pi_{v(e),\theta})^*
    \left(
        I-z(\mathcal A_{0,v(e),\theta})^{-1}
    \right)^{-1}
    \Pi_{v(e),\theta} .
\]
For sufficiently regular boundary data, $M_\theta(z)$ is the negative of
the total conormal derivative of the solution of the boundary-value problem \eqref{BVPsolutionop}.

\paragraph{The almost solvable extension and the transmission operator.}

We now identify the fibre transmission problem as an almost solvable extension
of $\mathcal A_{\max,\theta}$. Define
\[
    \mathcal A_{0,\mathbb I}^{\theta}
    :=
    \mathcal A_{\max,\theta}
    \upharpoonright
    \ker \Gamma_{1,\theta} ,
\qquad \textup{with} \qquad 
    D(\mathcal A_{0,\mathbb I}^{\theta})
    :=
    \left\{
        \hat u\in D(\Gamma_{1, \theta}):
        \Gamma_{1,\theta}\hat u=0
    \right\}.
\]
Since elements of $D(\mathcal A_{\max,\theta})$ are represented using a
common boundary datum $\phi\in L^2(D_{\mathrm{cell}})$, the continuity
condition $\hat u_v=\hat u_e$ on $D_{\mathrm{cell}}$ is encoded by $\Gamma_{0,\theta}$. For sufficiently regular functions, \eqref{eq:Gamma1-flux} identifies $\Gamma_{1,\theta}\hat u=0$ with the flux balance condition
\[
    (\nabla+i\theta)\hat u_v\cdot \nu_v
    +
    (\nabla+i\theta)\hat u_e\cdot \nu_e
    =
    0
    \qquad
    \text{on }D_{\mathrm{cell}}.
\]
To identify the operator domains, let
$u\in D(\mathcal A_\theta^h)$, set $f=\mathcal A_\theta^h u$, and
write $g\in H^{1/2}(D_{\mathrm{cell}})$ for its common interface trace.
Testing the weak formulation with functions of zero interface trace gives $u=\mathcal A_{0,\theta}^{-1}f+\Pi_\theta g.$
Testing next with $\Pi_\theta\psi$, for arbitrary
$\psi\in H^{1/2}(D_{\mathrm{cell}})$, and using the harmonicity of this
lift, we obtain $\lambda_\theta(g,\psi)=(\Pi_\theta^*f,\psi)_{\mathcal E}.$
By the definition of $\Lambda_\theta$, it follows that
$g\in D(\Lambda_\theta)$ and
$\Lambda_\theta g=-\Pi_\theta^*f$, so that $\Gamma_{1,\theta}u=0$.
Conversely, if $u=\mathcal A_{0,\theta}^{-1}f+\Pi_\theta g$ with
$g\in D(\Lambda_\theta)$ and $\Gamma_{1,\theta}u=0$, decomposing each
test function into a function of zero interface trace and a harmonic lift
gives the weak formulation for $\mathcal A_\theta^h u=f$.
Thus $\mathcal A_{0,\mathbb I}^{\theta}=\mathcal A_\theta^h$, with equality
of domains and actions.  We use the extension notation in the boundary-triple
analysis below.  In particular,
$(\mathcal A_{0,\mathbb I}^{\theta}-z)\hat u=\hat f$ is equivalent to the
fibre transmission problem \eqref{eq:transmission}.

By the Kre\u{\i}n resolvent formula, for $z\in \rho(\mathcal A_{0,\theta})
    \cap
    \rho(\mathcal A_{0,\mathbb I}^{\theta})$, one has
\begin{equation}
\label{eq:Krein-formula}
\begin{aligned}
    R_{0,\mathbb I}^{\theta}(z)
    &:=
    (\mathcal A_{0,\mathbb I}^{\theta}-z)^{-1}
    =
    (\mathcal A_{0,\theta}-z)^{-1}
    -
    S_{\theta,z}
    \bigl(M_\theta(z)\bigr)^{-1}
    \bigl(S_{\theta,\overline z}\bigr)^* .
\end{aligned}
\end{equation}
Since $\mathcal A_{0,\mathbb I}^{\theta}$ is self-adjoint and
non-negative, the formula applies in particular for $z\in K$. This represents the fibre resolvent of the Neumann Laplacian on the periodic
thin structure in terms of the Weyl function $M_\theta(z)$.

\section{Auxiliary estimates of the fibrewise Ryzhov triple on the vertex}

We establish the following three vertex estimates for the Weyl reduction:
\[
 \Vert\mathcal A_{0,v,\theta}^{-1}\Vert\leq Ch,
 \qquad \Vert\Pi_{v,\theta}\Vert\leq C,
 \qquad \mu_{v,\theta,2}\geq ch^{-1}.
\]
Here $\mu_{v,\theta,2}$ is the first positive eigenvalue of
$-\Lambda_{v,\theta}$. The preceding eigenvalue is zero and corresponds
to the rigid vertex mode retained in the effective model.
Under the vertex rescaling $x=b_v^h+r_hy$, $r_h=h^{2/3}$, the bulk
Rayleigh quotient acquires the factor $r_h^{-2}=h^{-4/3}$.
The estimates below quantify the effect of the shrinking contact plates
on the mixed problem and on the non-rigid boundary modes.

\subsection{Estimate for the decoupled Dirichlet operator}
\label{subsec:vertex-dirichlet-estimate}

The vertex lies strictly inside the periodicity cell, so multiplication by
$e^{i\theta\cdot x}$ defines unitary operators $U_\theta$ in
$L^2(X_{\mathrm{cell},v}^h)$ and $V_\theta$ in
$L^2(D_{\mathrm{cell}})$. The identity $(\nabla+i\theta)u=e^{-i\theta\cdot x}\nabla(U_\theta u)$
preserves the homogeneous mixed boundary conditions and gives
\begin{equation}\label{eq:vertex-gauge-identities}
 U_\theta\mathcal A_{0,v,\theta}U_\theta^{-1}=\mathcal A_{0,v,0},
 \qquad U_\theta\Pi_{v,\theta}V_\theta^{-1}=\Pi_{v,0},\qquad
 V_\theta\Lambda_{v,\theta}V_\theta^{-1}=\Lambda_{v,0}.
\end{equation}
Thus all three vertex estimates reduce to $\theta=0$.

Write
\[
 X_{\mathrm{cell},v}^h=b_v^h+r_hY_v^h,
 \quad D_{\mathrm{cell}}=b_v^h+r_hD^h,
 \quad r_h=h^{2/3},\quad \varepsilon_h=h/r_h=h^{1/3}.
\]
By Assumption~\ref{ass:standing-geometry}, $Y_v^h$ belongs to a uniformly
bi-Lipschitz family of domains of size one and
$|D^h|\asymp\varepsilon_h^2$. For $w\in H^1(Y_v^h)$, let
$c=|Y_v^h|^{-1}\int_{Y_v^h}w\,dy$. The trace embedding
$H^1(Y_v^h)\to L^4(\partial Y_v^h)$ and the Poincar\'e inequality have
uniform constants on this family. H\"older's inequality therefore gives
\begin{equation}\label{eq:vertex-window-trace}
 \Vert w-c\Vert_{L^2(D^h)}^2
 \leq |D^h|^{1/2}\Vert w-c\Vert_{L^4(\partial Y_v^h)}^2\leq C\varepsilon_h\Vert w-c\Vert_{H^1(Y_v^h)}^2
 \leq C\varepsilon_h\Vert\nabla w\Vert_{L^2(Y_v^h)}^2.
\end{equation}

If $w|_{D^h}=0$, the left-hand side equals $|c|^2|D^h|$, whence
$|c|^2\leq C\varepsilon_h^{-1}\Vert\nabla w\Vert_{L^2(Y_v^h)}^2$.
Another application of the Poincar\'e inequality yields
\[
 \Vert w\Vert_{L^2(Y_v^h)}^2
 \leq C\varepsilon_h^{-1}\Vert\nabla w\Vert_{L^2(Y_v^h)}^2,
 \qquad w\in H^1_{D^h}(Y_v^h).
\]
Hence the first mixed eigenvalue on the rescaled vertex is bounded below
by $c\varepsilon_h$, the small-window scale in
\cite{denzler1998bounds}. Scaling back and using
\eqref{eq:vertex-gauge-identities}, the first eigenvalue
$\lambda^h_{1,v,\theta}$ of $\mathcal A_{0,v,\theta}$ satisfies
$\lambda^h_{1,v,\theta}\geq cr_h^{-2}\varepsilon_h$. Consequently,
\begin{equation}\label{A0estimate}
 \Vert\mathcal A_{0,v,\theta}^{-1}\Vert_{\mathcal H_v\to\mathcal H_v}
 =\bigl(\lambda^h_{1,v,\theta}\bigr)^{-1}
 \leq Cr_h^2\varepsilon_h^{-1}=Ch,
 \qquad\theta\in Q'.
\end{equation}

\subsection{Estimate for the vertex lift operator}
\label{subsubsec:vertex-lift-estimate}

We use the gauge and rescaled geometry of
Subsection~\ref{subsec:vertex-dirichlet-estimate} to estimate
$\Pi_{v,\theta}$.

On $Y_v^h$, let $\mathcal A^0_{v,h}$ denote the mixed
Dirichlet--Neumann Laplacian associated with the form
\[
    a^0_{v,h}(u,\varphi)
    :=
    \int_{Y_v^h}
    \nabla u\cdot\overline{\nabla\varphi}\,dy,
    \qquad
    D(a^0_{v,h})
    =
    H^1_{D^h}(Y_v^h).
\]
Thus $\mathcal A^0_{v,h}u=-\Delta u$ in $Y_v^h$, with Dirichlet condition
on $D^h$ and homogeneous Neumann condition on
$\partial Y_v^h\setminus D^h$. The unitary rescaling $(T_hu)(y):=r_h^{3/2}u(b_v^h+r_hy)$
relates it to the physical operator by $T_h\mathcal A_{0,v,0}T_h^{-1}=r_h^{-2}\mathcal A^0_{v,h}.$

For smooth $g$ on $D^h$, let $\Pi^0_{v,0}g$ be the mixed harmonic lift
defined by
\[
\begin{cases}
    -\Delta \Pi^0_{v,0}g=0,
    & \text{in }Y_v^h,\\[0.3em]
    \Pi^0_{v,0}g=g,
    & \text{on }D^h,\\[0.3em]
    \partial_n \Pi^0_{v,0}g=0,
    & \text{on }\partial Y_v^h\setminus D^h .
\end{cases}
\]
As for the Ryzhov lift $\Pi_{v,\theta}$ in
\eqref{eq:transposition-definition-Pi}, we extend the rescaled lift
$\Pi^0_{v,0}$ to $L^2(D^h)$ by transposition: for every $f\in L^2(Y_v^h)$,
\begin{equation}
\label{eq:Pi-adjoint-conormal-rescaled}
    (\Pi^0_{v,0})^*f
    =
    -\left.
    \partial_{n}
    (\mathcal A^0_{v,h})^{-1}f
    \right|_{D^h}.
\end{equation}
Here $n$ points outwards from the rescaled vertex region, in accordance
with the normal convention for the Ryzhov triple.

\begin{theorem}[Uniform estimate for the vertex lift]
\label{thm:vertex-lift-estimate}
Assume the standing geometry in
Assumption~\ref{ass:standing-geometry}.  In particular, the rescaled contact
plates and collars satisfy the uniform size assumptions there, and the full
local mixed-boundary estimate
\eqref{eq:standing-zaremba-estimate} holds in every blown-up contact chart.
Then there exist constants $C>0$ and $h_0>0$, independent of
$\theta\in Q'$, such that, for $0<h<h_0$,
\begin{equation}
\label{eq:Pi-v-theta-uniform-bound}
    ||\Pi_{v,\theta}||_{L^2(D_{\mathrm{cell}})\to
    L^2(X_{\mathrm{cell},v}^h)}
    \le C .
\end{equation}
\end{theorem}

\begin{proof}
By \eqref{eq:vertex-gauge-identities}, it suffices to take $\theta=0$.

We first work on the rescaled vertex $Y_v^h$. Let
$f\in L^2(Y_v^h)$, and set $w:=(\mathcal A^0_{v,h})^{-1}f$.
Then
\[
    -\Delta w=f
    \quad\text{in }Y_v^h,
    \qquad
    w=0\quad\text{on }D^h,
    \qquad
    \partial_n w=0
    \quad\text{on }\partial Y_v^h\setminus D^h .
\]

We first prove the conormal estimate
\begin{equation}
\label{eq:conormal-rescaled-main}
    \left\Vert
    \partial_n(\mathcal A^0_{v,h})^{-1}f|_{D^h}
    \right\Vert_{L^2(D^h)}
    \le
    C\varepsilon_h^{-1}
    ||f||_{L^2(Y_v^h)} .
\end{equation}

Fix one contact component $D_j^h$, and consider its collar $U_j^h$. In the
blown-up variables $y=y_j^h+\varepsilon_h\widehat y$ and $\widehat w(\widehat y):=w(y_j^h+\varepsilon_h\widehat y)$, let $\widehat U_j$ denote the blown-up collar and $\widehat D_j$ the corresponding Dirichlet interface.

Let $\widehat U_j'$ be a smaller collar with the same contact face
$\widehat D_j$. Applying the standing local estimate
\eqref{eq:standing-zaremba-estimate} to $\widehat w$ in the
contact charts, and combining the estimates over a uniformly finite cover
of this smaller collar, gives, with $s=s_{\mathrm Z}>0$,
\begin{equation}
\label{eq:local-zaremba-H-regularity}
    ||\widehat w||_{H^{3/2+s}(\widehat U_j')}
    \le
    C\left(
        ||\Delta_{\widehat y}\widehat w||_{L^2(\widehat U_j)}
        +
        ||\widehat w||_{H^1(\widehat U_j)}
    \right),
\end{equation}
with $C$ independent of $h$ and $j$ by
Assumption~\ref{ass:standing-geometry}.  For related mixed regularity in
smooth dihedral charts, see \cite[Theorem~8.1]{dauge1992neumann}; the full
uniform estimate used here, including charts centred at corners or conical
points, is part of the standing assumption.

The trace theorem on the flat face $\widehat D_j$ gives $||\partial_{\widehat n}\widehat w||_{L^2(\widehat D_j)}
    \le
    C||\widehat w||_{H^{3/2+s}(\widehat U_j')}$. Combining this with \eqref{eq:local-zaremba-H-regularity}, we obtain $||\partial_{\widehat n}\widehat w||_{L^2(\widehat D_j)}
    \le
    C(||\Delta_{\widehat y}\widehat w||_{L^2(\widehat U_j)}
        +
        ||\widehat w||_{H^1(\widehat U_j)})$. Since $\widehat w=0$ on the Dirichlet interface $\widehat D_j$, the Poincaré
inequality in the blown-up collar yields $||\widehat w||_{L^2(\widehat U_j)}
    \le
    C||\nabla_{\widehat y}\widehat w||_{L^2(\widehat U_j)}$. Hence
\begin{equation}
\label{eq:local-zaremba-reference-derived}
    ||\partial_{\widehat n}\widehat w||_{L^2(\widehat D_j)}
    \le
    C\left(
        ||\nabla_{\widehat y}\widehat w||_{L^2(\widehat U_j)}
        +
        ||\Delta_{\widehat y}\widehat w||_{L^2(\widehat U_j)}
    \right).
\end{equation}

We return in \eqref{eq:local-zaremba-reference-derived} to the rescaled
vertex coordinates. Since $\nabla_{\widehat y}\widehat w
    =
    \varepsilon_h\nabla_y w$, $\Delta_{\widehat y}\widehat w
    =
    \varepsilon_h^2\Delta_y w$ and $dy=\varepsilon_h^3d\widehat y$, we have
\[
    ||\nabla_{\widehat y}\widehat w||_{L^2(\widehat U_j)}
    =
    \varepsilon_h^{-1/2}
    ||\nabla_y w||_{L^2(U_j^h)},
    \qquad 
    ||\Delta_{\widehat y}\widehat w||_{L^2(\widehat U_j)}
    =
    \varepsilon_h^{1/2}
    ||\Delta_y w||_{L^2(U_j^h)}.
\]
Moreover, $\partial_{\widehat n}\widehat w
    =
    \varepsilon_h\partial_n w$ and $d\widehat S=\varepsilon_h^{-2}dS$, and therefore $||\partial_{\widehat n}\widehat w||_{L^2(\widehat D_j)}
    =
    ||\partial_n w||_{L^2(D_j^h)}$. Thus
\[
    ||\partial_n w||_{L^2(D_j^h)}
    \le
    C\left(
        \varepsilon_h^{-1/2}
        ||\nabla w||_{L^2(U_j^h)}
        +
        \varepsilon_h^{1/2}
        ||\Delta w||_{L^2(U_j^h)}
    \right).
\]
Since the number of contact components is finite and independent of $h$,
and the collars $U_j^h$ have uniformly bounded overlap, summing over
$j=1,\dots,m$ gives
\begin{equation}
\label{eq:local-conormal-scaled}
    ||\partial_n w||_{L^2(D^h)}
    \le
    C\left(
        \varepsilon_h^{-1/2}
        ||\nabla w||_{L^2(Y_v^h)}
        +
        \varepsilon_h^{1/2}
        ||\Delta w||_{L^2(Y_v^h)}
    \right).
\end{equation}

We estimate the two terms on the right-hand side separately. Since
$w=(\mathcal A^0_{v,h})^{-1}f$, we have $||\Delta w||_{L^2(Y_v^h)}
    = ||f||_{L^2(Y_v^h)}$. For the gradient term, the rescaled mixed-eigenvalue bound from
Subsection~\ref{subsec:vertex-dirichlet-estimate} gives
\begin{equation}
\label{eq:A-inverse-rescaled}
    ||(\mathcal A^0_{v,h})^{-1}||_{L^2(Y_v^h)\to L^2(Y_v^h)}
    \le C\varepsilon_h^{-1}.
\end{equation}
The energy identity gives
\begin{equation}\label{eq:gradient-bound-rescaled}
    ||\nabla w||_{L^2(Y_v^h)}^2
    =
    (f,w)_{L^2(Y_v^h)}
    \le
    ||f||_{L^2(Y_v^h)}
    ||w||_{L^2(Y_v^h)}
    \le
    C\varepsilon_h^{-1}
    ||f||_{L^2(Y_v^h)}^2.
\end{equation}
Substituting \eqref{eq:gradient-bound-rescaled} and
$-\Delta w =f$ into \eqref{eq:local-conormal-scaled}, we obtain
\[
    ||\partial_n w||_{L^2(D^h)}
    \le
    C\left(
        \varepsilon_h^{-1}
        +
        \varepsilon_h^{1/2}
    \right)
    ||f||_{L^2(Y_v^h)}
    \le
    C\varepsilon_h^{-1}
    ||f||_{L^2(Y_v^h)} .
\]
This proves \eqref{eq:conormal-rescaled-main}.

By the adjoint identity \eqref{eq:Pi-adjoint-conormal-rescaled}, $||(\Pi^0_{v,0})^*f||_{L^2(D^h)}
    \le
    C\varepsilon_h^{-1}
    ||f||_{L^2(Y_v^h)}$. Therefore
\begin{equation}
\label{eq:rescaled-Pi-bound}
    ||\Pi^0_{v,0}||_{L^2(D^h)\to L^2(Y_v^h)}
    = ||(\Pi^0_{v,0})^*||_{L^2(Y_v^h)\to L^2(D^h)}
    \le
    C\varepsilon_h^{-1}.
\end{equation}

It remains to scale this estimate back to the physical vertex
$X_{\mathrm{cell},v}^h$. Let $g\in L^2(D_{\mathrm{cell}})$, and define
the rescaled boundary datum $\widehat g(y):=g(b_v^h+r_hy)$ for $y\in D^h$. If $u=\Pi_{v,0}g$ on $X_{\mathrm{cell},v}^h$, then $u(b_v^h+r_hy)=\Pi^0_{v,0}\widehat g(y)$. The $L^2$-norms scale as $||u||_{L^2(X_{\mathrm{cell},v}^h)}
    =
    r_h^{3/2}
    ||\Pi^0_{v,0}\widehat g||_{L^2(Y_v^h)}$ and $||g||_{L^2(D_{\mathrm{cell}})}
    =
    r_h
    ||\widehat g||_{L^2(D^h)}$. By \eqref{eq:rescaled-Pi-bound},
\[
\begin{aligned}
    ||\Pi_{v,0}g||_{L^2(X_{\mathrm{cell},v}^h)}
    &\le
    r_h^{3/2}
    C\varepsilon_h^{-1}
    ||\widehat g||_{L^2(D^h)}=
    C r_h^{1/2}\varepsilon_h^{-1}
    ||g||_{L^2(D_{\mathrm{cell}})}.
\end{aligned}
\]
Since $r_h=h^{2/3}, \varepsilon_h=h^{1/3}$, we have $r_h^{1/2}\varepsilon_h^{-1}=1$. Therefore $||\Pi_{v,0}||_{L^2(D_{\mathrm{cell}})\to
    L^2(X_{\mathrm{cell},v}^h)}
    \le C$. Finally, by the gauge equivalence at the beginning of the proof,
\[
    ||\Pi_{v,\theta}||_{L^2(D_{\mathrm{cell}})\to
    L^2(X_{\mathrm{cell},v}^h)}
    =
    ||\Pi_{v,0}||_{L^2(D_{\mathrm{cell}})\to
    L^2(X_{\mathrm{cell},v}^h)}
    \le C,
    \qquad
    \theta\in Q'.
\]
This proves \eqref{eq:Pi-v-theta-uniform-bound}.
\end{proof}

\subsection{Estimate for the Dirichlet-to-Neumann operator}

We next consider the vertex Dirichlet-to-Neumann operator $\Lambda_{v,\theta}$.
With the sign convention used in the Ryzhov triple, $\Lambda_{v,\theta}$ is
non-positive. Its Steklov eigenvalues are therefore defined by
\[
    -\Lambda_{v,\theta}\psi_{v,\theta,k}
    =
    \mu_{v,\theta,k}\psi_{v,\theta,k},
    \qquad
    0\le \mu_{v,\theta,1}\le \mu_{v,\theta,2}\le \cdots .
\]
Equivalently, $\mu_{v,\theta,k}$ are the eigenvalues of the mixed Steklov
problem
\begin{equation}
\label{eq:steklov-v}
\begin{cases}
-(\nabla+i\theta)\cdot(\nabla+i\theta)u=0,
& \text{in } X_{\mathrm{cell},v}^h,\\[0.4em]
(\nabla+i\theta)u\cdot n=\mu u,
& \text{on } D_{\mathrm{cell}},\\[0.4em]
(\nabla+i\theta)u\cdot n=0,
& \text{on } \partial X_{\mathrm{cell},v}^h\setminus D_{\mathrm{cell}}.
\end{cases}
\end{equation}

The lowest absolute Steklov eigenvalue is equal to zero for every
$\theta\in Q'$. Indeed, since $X_{\mathrm{cell},v}^h$ is an internal vertex
region in the periodic quotient, it does not carry any quasiperiodic matching
condition on the boundary of the periodicity cell. Hence the function $u_{v,\theta,1}(x)=e^{-i\theta\cdot x}$ satisfies $(\nabla+i\theta)u_{v,\theta,1}=0$ in $X_{\mathrm{cell},v}^h$, and therefore solves \eqref{eq:steklov-v} with $\mu=0$. Since
$X_{\mathrm{cell},v}^h$ is connected, this zero eigenspace is
one-dimensional.

We denote by $\psi_{v,\theta}
    :=
    |D_{\mathrm{cell}}|^{-1/2}
    e^{-i\theta\cdot x}\big|_{D_{\mathrm{cell}}}$ the normalised boundary trace of this lowest mode, and define the orthogonal projection
\begin{equation}\label{projdef0}
     \mathcal P_\theta g
    :=
    (g,\psi_{v,\theta})_{L^2(D_{\mathrm{cell}})}
    \psi_{v,\theta},
    \qquad
    g\in L^2(D_{\mathrm{cell}}).
\end{equation}
Thus $\operatorname{Ran}\mathcal P_\theta
    =
    \ker \Lambda_{v,\theta}$.  We put
\[
 \mathcal P_\theta^\perp:=I-\mathcal P_\theta,
 \qquad
 \hat{\mathcal E}_\theta:=\mathcal P_\theta\mathcal E,
 \qquad
 \breve{\mathcal E}_\theta:=\mathcal P_\theta^\perp\mathcal E.
\]
The zero Steklov mode is the resonant vertex degree of freedom retained in
the effective model. In reducing the Weyl function, we therefore split off $\hat{\mathcal E}_\theta$ and estimate the inverse of the Dirichlet-to-Neumann operator only on
$\breve{\mathcal E}_\theta$.  The relevant quantity is the first non-zero
Steklov eigenvalue.
    
\begin{theorem}\label{thm:DtN-gap}
Assume the uniform rescaled-vertex and contact geometry of
Theorem~\ref{thm:vertex-lift-estimate} and
Lemma~\ref{lem:boundary-poincare-trace-shrinking-interface} below.  In
particular,
let the vertex diameter satisfy $r_h\asymp h^{2/3}$, and let the contact
patches have physical diameter of order $h$, so that their diameter after
vertex rescaling is $\varepsilon_h=h/r_h\asymp h^{1/3}$.  Then there
exist constants $c>0$ and $h_0>0$, independent of
$\theta\in Q'$, such that, for $0<h<h_0$, one has $\mu_{v,\theta,2}\geq c h^{-1}.$
Equivalently,
\[
\bigl\Vert(\mathcal P_\theta^\perp
\Lambda_{v,\theta}\mathcal P_\theta^\perp)^{-1}\bigr\Vert_
{\breve{\mathcal E}_\theta\to\breve{\mathcal E}_\theta}
\leq Ch,
\]
where $\mathcal P_\theta$ is the orthogonal projection onto
$\ker\Lambda_{v,\theta}$.
\end{theorem}

\begin{proof}
\smallskip
\noindent
\emph{Step 1. Gauge reduction.}

By \eqref{eq:vertex-gauge-identities},
$\mu_{v,\theta,k}=\mu_{v,0,k}$ for all $k\geq1$.
It is therefore enough to estimate the second eigenvalue for $\theta=0$.

\smallskip
\noindent
\emph{Step 2. Variational characterisation of the first non-zero Steklov eigenvalue.}
\nopagebreak[4]

For $\theta=0$, the first Steklov eigenvalue equals $0$, and the corresponding eigenfunctions are the constants. Consequently, the second eigenvalue is characterised by the Rayleigh quotient on the subspace orthogonal to constants on the interface:
\begin{equation}\label{eq:minmax-steklov-rewrite}
\mu_{v,0,2}
=
\inf_{\substack{u\in H^1(X^h_{\textup{cell},v})\\
\Vert u|_{D_{\textup{cell}}}\Vert_{L^2}>0,\ \int_{D_{\textup{cell}}} u\,dS=0}}
\frac{\displaystyle \int_{X^h_{\textup{cell},v}} |\nabla u|^2\,dx}
{\displaystyle \int_{D_{\textup{cell}}} | u|^2\,dS}.
\end{equation}
Indeed, once the trace on $D_{\textup{cell}}$ is prescribed, the harmonic extension minimises the Dirichlet energy in the bulk. Thus \eqref{eq:minmax-steklov-rewrite} follows from the min--max principle
for the mixed Steklov problem.

\smallskip
\noindent
\emph{Step 3. Rescaling the vertex and tracking the interface.}

Use the rescaled domains $Y_v^h,D^h$ and the parameters
$r_h,\varepsilon_h$ introduced in
Subsection~\ref{subsec:vertex-dirichlet-estimate}.
Under the change of variables $x=b_v^h+r_hy$,
$u(b_v^h+r_hy)=w(y)$, the Dirichlet energy and the boundary $L^2$ norm
transform as
\[
\int_{X^h_{\textup{cell},v}} |\nabla u|^2\,dx
=
r_h \int_{Y^h_v} |\nabla_y w|^2\,dy,
\qquad 
\int_{D_{\textup{cell}}} | u|^2\,dS_x
=
r_h^{2} \int_{D^h} | w|^2\,dS_y.
\]
Therefore the min--max formula becomes
\begin{equation}\label{eq:scaled-minmax-rewrite}
\mu_{v,0,2}
=
r_h^{-1}
\inf_{\substack{w\in H^1(Y^h_v)\\
\Vert w|_{D^h}\Vert_{L^2}>0,\ \int_{D^h} w\,dS=0}}
\frac{\displaystyle \int_{Y^h_v} |\nabla w|^2\,dy}
{\displaystyle \int_{D^h} | w|^2\,dS}.
\end{equation}
\smallskip
\noindent
\emph{Step 4. Boundary Poincar\'e--trace estimate on the shrinking interface.}

The following consequence of \eqref{eq:vertex-window-trace} controls the
denominator in \eqref{eq:scaled-minmax-rewrite}.

\begin{lemma}[Boundary Poincaré--trace estimate on the shrinking interface]
\label{lem:boundary-poincare-trace-shrinking-interface}
Let $Y_v^h\subset\mathbb R^3$ and
$D^h=\bigcup_{j=1}^mD_j\subset\partial Y_v^h$ be the rescaled vertex and
contact set supplied by Assumption~\ref{ass:standing-geometry}, and let
$\varepsilon_h=h^{1/3}$.  Thus the number of connected plates is fixed,
their diameters and surface measures are comparable with
$\varepsilon_h$ and $\varepsilon_h^2$, respectively, and they admit the
uniform flat collars specified in that assumption.

Then there exists a constant $C>0$, independent of $h$, such that
\begin{equation}
\label{eq:boundary-poincare-rewrite}
    \left\Vert
        w|_{D^h}
    \right\Vert_{L^2(D^h)}^2
    \le
    C\varepsilon_h
    \Vert\nabla w\Vert_{L^2(Y_v^h)}^2,
    \qquad
    w\in H^1(Y_v^h),
    \qquad
    \int_{D^h} w\,dS=0 .
\end{equation}
Here $w|_{D^h}$ is understood in the trace sense.
\end{lemma}

\begin{proof}
Let $c$ be the volume average of $w$ on $Y_v^h$. Since
$\int_{D^h}w\,dS=0$,
\[
 \Vert w-c\Vert_{L^2(D^h)}^2
 =\Vert w\Vert_{L^2(D^h)}^2+|c|^2|D^h|.
\]
Combining this identity with \eqref{eq:vertex-window-trace} proves
\eqref{eq:boundary-poincare-rewrite}. The argument applies to the whole
contact set and does not require its connectedness.
\end{proof}

We return to the quotient in \eqref{eq:scaled-minmax-rewrite}. For every
admissible $w$ satisfying $\int_{D^h}w\,dS=0$, Lemma~\ref{lem:boundary-poincare-trace-shrinking-interface} implies
\[
    \frac{
        \displaystyle
        \int_{Y_v^h}|\nabla w|^2\,dy
    }{
        \displaystyle
        \int_{D^h}|w|^2\,dS
    }
    \ge
    c\varepsilon_h^{-1}.
\]
Substituting this estimate into \eqref{eq:scaled-minmax-rewrite}, we obtain
$\mu_{v,0,2}
    \ge
    c\,r_h^{-1}\varepsilon_h^{-1}$. Since $r_h=h^{2/3}$ and $\varepsilon_h=h^{1/3}$, this gives $\mu_{v,0,2}
    \ge
    c h^{-1}$. The gauge equivalence gives the same bound uniformly in
$\theta$. The spectral theorem on the orthogonal complement of the zero
mode now yields
\[
\bigl\Vert(\mathcal P_\theta^\perp \Lambda_{v,\theta}\mathcal P_\theta^\perp)^{-1}\bigr\Vert_{\breve{\mathcal E}_\theta\to\breve{\mathcal E}_\theta}
=
\mu_{v,\theta,2}^{-1}
\le Ch,
\]
as claimed.
\end{proof}

\section{Reduction of the Weyl function and effective fibre operator}

We now pass from the fibrewise Ryzhov representation to the effective periodic
graph model.  For each quasi-momentum $\theta\in Q'$, Kre\u{\i}n's formula
gives
\[
   (\mathcal A^\theta_{0,\mathbb I}-z)^{-1}
   =
   (\mathcal A_{0,\theta}-z)^{-1}
   -S_{\theta,z}M_\theta(z)^{-1}S_{\theta,\bar z}^*.
\]
We reduce $M_\theta(z)^{-1}$ to the rigid vertex channel, uniformly in $h$, $\theta$ and $z\in K$. We first establish
the edge estimates and control the Weyl blocks, then realise the compressed
resolvent on $\mathcal H_e\oplus\mathbb C$. Restriction to the transverse ground modes then gives the graph fibres. The final subsections derive the spectral
consequences and reconstruct the global operators.

\subsection{Uniform estimates for the edge Ryzhov triple}
\label{step:edge-ryzhov-estimates}

Each edge component of the periodic quotient is represented by the
unfolded cylinder
\[
 X^h_{\mathrm{cell},e}\simeq(0,l_e)\times h\omega_e,
 \qquad 0<l_-\leq l_e\leq l_+<\infty.
\]
Its two end sections are the edge--vertex contact plates. A local gauge
removes the transverse magnetic part. The longitudinal gauge used below multiplies endpoint values by unit phases,
which must be retained in the coupled fibre problem.

\begin{theorem}[Uniform inverse bound for the decoupled edge operator]
\label{lem:edge-A0-inverse}
On $X^h_{\mathrm{cell},e}
    \simeq
    (0,l_e)\times h\omega_e$, there exist constants $C>0$ and
$h_0>0$, independent of $e$ and $\theta\in Q'$, such that, for
$0<h<h_0$,
\[
    \Bigl\Vert
        \left(\mathcal A_{0,e,\theta}\right)^{-1}
    \Bigr\Vert_{L^2(X_{\mathrm{cell},e}^h)\to L^2(X_{\mathrm{cell},e}^h)}
    \leq C .
\]
\end{theorem}

\begin{proof}
After removing the transverse magnetic part, set
$v(s,y)=e^{i(t_e\cdot\theta)s}u(s,y)$. The zero endpoint traces are
preserved and
$\partial_s v=e^{i(t_e\cdot\theta)s}(\partial_s+it_e\cdot\theta)u$.
The one-dimensional Poincar\'e inequality, integrated over $h\omega_e$,
therefore gives
\[
 \Vert u\Vert_{L^2(X^h_{\mathrm{cell},e})}^2
 \leq \frac{l_+^2}{\pi^2}
       \Vert(\partial_s+it_e\cdot\theta)u\Vert_{L^2}^2
 \leq C a_{0,e,\theta}(u,u).
\]
This is the required uniform lower bound for $\mathcal A_{0,e,\theta}$.
\end{proof}

We next estimate the edge Poisson lift.

\begin{theorem}[Uniform bound for the edge Poisson lift]
\label{lem:edge-Pi-bound}
Let
$X^h_{\mathrm{cell},e}\simeq(0,l_e)\times h\omega_e$ be a straight
cylinder, where $0<l_-\leq l_e\leq l_+<\infty$.  Its two end sections
\[
 D_{e,-}^h:=\{0\}\times h\omega_e,
 \qquad
 D_{e,+}^h:=\{l_e\}\times h\omega_e
\]
are the edge--vertex contact plates of the unfolded edge; put
$D_e^h:=D_{e,-}^h\cup D_{e,+}^h$.  There exist constants $C>0$ and
$h_0>0$, independent of $e$ and $\theta\in Q'$, such that, for
$0<h<h_0$, the Poisson lift on this cylinder satisfies
\[
    ||\Pi_{e,\theta} g||_{L^2(X^h_{\mathrm{cell},e})}
    \leq
    C
    ||g||_{L^2(D_e^h)},
    \qquad
    g\in L^2(D_e^h).
\]
Consequently, after summing over the finitely many edge components in the
cell, the full edge lift satisfies
$
 ||\Pi_{e,\theta}g||_{\mathcal H_e}
 \leq C||g||_{L^2(D_{\mathrm{cell}})}
$.
\end{theorem}

\begin{proof}
Take smooth endpoint data $g$ and write $\hat u_e=\Pi_{e,\theta}g$.
In cylinder coordinates, put $\theta_e=t_e\cdot\theta$ and
$\theta_{\perp,e}=\theta-\theta_et_e$. The transverse gauge
$\widetilde u_e=e^{i\theta_{\perp,e}\cdot y}\hat u_e$ transforms the
lateral magnetic Neumann condition into the ordinary Neumann condition.
Expand the gauged solution and endpoint data in the orthonormal Neumann
eigenbasis $\{\chi_k^h\}_{k\geq0}$ of $h\omega_e$:
\[
 \widetilde u_e(s,y)=\sum_{k\geq0}u_k(s)\chi_k^h(y),
 \quad u_k(0)=a_k,\quad u_k(l_e)=b_k,
 \quad \mu_0^h=0,\quad \mu_k^h=h^{-2}\mu_k\ (k\geq1).
\]
This gauge is unitary, so Parseval's identity gives
$\sum_k(|a_k|^2+|b_k|^2)=\Vert g\Vert_{L^2(D_e^h)}^2$.
The longitudinal gauge $v_k(s)=e^{i\theta_es}u_k(s)$ yields
\[
 -v_k''+\mu_k^h v_k=0,
 \qquad v_k(0)=a_k,\quad v_k(l_e)=e^{i\theta_el_e}b_k.
\]
For $k=0$, the solution is affine, and hence
\[
 \Vert v_0\Vert_{L^2(0,l_e)}^2
 \leq C(|a_0|^2+|b_0|^2).
\]
For $k\geq1$, put $q_k=\sqrt{\mu_k^h}$. Then
\[
 v_k(s)=a_k\frac{\sinh(q_k(l_e-s))}{\sinh(q_kl_e)}
       +e^{i\theta_el_e}b_k\frac{\sinh(q_ks)}{\sinh(q_kl_e)}.
\]
Since $q_kl_e\geq l_-\sqrt{\mu_1}/h$, the two coefficients are
bounded by $Ce^{-q_ks}$ and $Ce^{-q_k(l_e-s)}$, respectively, for
all sufficiently small $h$. Therefore
\[
 \Vert v_k\Vert_{L^2(0,l_e)}^2
 \leq Cq_k^{-1}(|a_k|^2+|b_k|^2)
 \leq C(|a_k|^2+|b_k|^2).
\]
Summing over $k$ and applying Parseval's identity proves the estimate
on one cylinder. Density extends it to $L^2$ endpoint data, and summing
over the fixed finite set of cylinders gives the full edge-lift bound.
\end{proof}

\subsection{Low--high boundary decomposition and invertibility of the compressed Weyl block}

Recall the orthogonal projection $\mathcal P_\theta$ defined in
\eqref{projdef0}. It is the spectral projection of the vertex Dirichlet-to-Neumann operator $\Lambda_{v,\theta}$ onto its one-dimensional kernel $\ker\Lambda_{v,\theta}$. Since the reference vertex is represented by $v=0$, this eigenspace is generated by the normalised boundary trace $\psi_{v,\theta}
    =
    |D_{\mathrm{cell}}|^{-1/2}
    e^{-i\theta\cdot x}\big|_{D_{\mathrm{cell}}}$. 
    
With the notation introduced in Section~4.3, the boundary space decomposes as
\[
    \mathcal E=L^2(D_{\mathrm{cell}})
    =\hat{\mathcal E}_\theta\oplus\breve{\mathcal E}_\theta.
\]
In particular, one has $\hat{\mathcal E}_\theta
    =
    \operatorname{span}\{\psi_{v,\theta}\},$ $\mathcal P_\theta\Lambda_{v,\theta}\mathcal P_\theta=0.$
With respect to this decomposition, the Weyl function has the block
representation
\[
    M_\theta(z)
    =
    \begin{pmatrix}
        \mathbb A_\theta(z) & \mathbb B_\theta(z)\\[0.4em]
        \mathbb E_\theta(z) & \mathbb D_\theta(z)
    \end{pmatrix},
\]
where
\[
    \mathbb A_\theta(z)
    :=
    \mathcal P_\theta M_\theta(z)\mathcal P_\theta,
    \quad
    \mathbb B_\theta(z)
    :=
    \mathcal P_\theta M_\theta(z)\mathcal P_\theta^\perp,\quad \mathbb E_\theta(z)
    :=
    \mathcal P_\theta^\perp M_\theta(z)\mathcal P_\theta,
    \quad
    \mathbb D_\theta(z)
    :=
    \mathcal P_\theta^\perp M_\theta(z)\mathcal P_\theta^\perp .
\]

\begin{theorem}
\label{thm:D-inv}
Assume that the vertex Dirichlet-to-Neumann operator satisfies the uniform
Steklov gap estimate
\begin{equation}
\label{eq:vertex-steklov-gap-used}
    \bigl(
        -\Lambda_{v,\theta} q,q
    \bigr)_{L^2(D_{\mathrm{cell}})}
    \geq
    c h^{-1}
    ||q||_{L^2(D_{\mathrm{cell}})}^2,
    \qquad
    q\in
    D(\Lambda_{v,\theta})\cap\breve{\mathcal E}_\theta ,
\end{equation}
with $c>0$ independent of $h$ and $\theta$. Then there exist $C>0$ and $h_0>0$, independent of
$\theta\in Q'$ and $z\in K$, such that, for $0<h<h_0$,
\[
    \Bigl\Vert
        \bigl(
            \mathcal P_\theta^\perp
            M_\theta(z)
            \mathcal P_\theta^\perp
        \bigr)^{-1}
    \Bigr\Vert_{\breve{\mathcal E}_\theta\to\breve{\mathcal E}_\theta}
    \leq
    Ch .
\]
\end{theorem}

\begin{proof}
The Ryzhov representation of the Weyl function gives
\[
    M_\theta(z)
    =
    \Lambda_\theta
    +
    z
    \Pi_\theta^*
    \bigl(
        I-z\mathcal A_{0,\theta}^{-1}
    \bigr)^{-1}
    \Pi_\theta ,
    \qquad
    \Lambda_\theta=\Lambda_{v,\theta}+\Lambda_{e,\theta}
\]
with the last identity understood in the form sense.
Set $\mathcal B_{\theta,h}(z)
    :=
    z
    \Pi_\theta^*
    \bigl(
        I-z\mathcal A_{0,\theta}^{-1}
    \bigr)^{-1}
    \Pi_\theta$. By the uniform bounds for the vertex lift and by
Theorem~\ref{lem:edge-Pi-bound}, we have $||\Pi_\theta||_{\mathcal E\to\mathcal H}\le C$.

Moreover, since $\mathcal A_{0,\theta}$ is self-adjoint and non-negative,
$\bigl(
        I-z\mathcal A_{0,\theta}^{-1}
    \bigr)^{-1}
    =
    \mathcal A_{0,\theta}
    (\mathcal A_{0,\theta}-z)^{-1}$. Therefore, by the spectral theorem and the definition of $K$,
\[
    \Bigl\Vert
        \bigl(
            I-z\mathcal A_{0,\theta}^{-1}
        \bigr)^{-1}
    \Bigr\Vert_{\mathcal H\to\mathcal H}
    \le C_K,
    \qquad
    z\in K,
\]
with $C_K$ independent of $h$ and $\theta$. Consequently,
\begin{equation}
\label{eq:Btheta-bound}
    ||\mathcal B_{\theta,h}(z)||_{\mathcal E\to\mathcal E}
    \le C_K,
    \qquad
    z\in K,\quad \theta\in Q'.
\end{equation}

We define the compressed block through its associated form.  By spectral
calculus, the Steklov gap \eqref{eq:vertex-steklov-gap-used}, initially
stated on the operator domain, extends to the form domain of
$-\Lambda_{v,\theta}$.  Let
$\mathcal V_\theta^\perp
    :=
    H^{1/2}(D_{\mathrm{cell}})
    \cap
    \breve{\mathcal E}_\theta$. On $\mathcal V_\theta^\perp$, we define the sesquilinear form
\[
\begin{aligned}
    \mathfrak m_{\theta,z}(q,\psi)
    &:=
    \lambda_{v,\theta}(q,\psi)
    +
    \lambda_{e,\theta}(q,\psi)
    -
    \bigl(
        \mathcal B_{\theta,h}(z)q,\psi
    \bigr)_{\mathcal E},
    \qquad
    q,\psi\in\mathcal V_\theta^\perp .
\end{aligned}
\]
This is a bounded perturbation of the closed non-negative form
$\lambda_{v,\theta}+\lambda_{e,\theta}$. Hence
$\mathfrak m_{\theta,z}$ is closed and sectorial on
$\breve{\mathcal E}_\theta$. Let $T_{\theta,z}$ be the operator associated
with this form. By the form definition of the compressed Weyl block,
the operator $T_{\theta,z}$ satisfies
$T_{\theta,z}
    =
    -
    \mathbb D_\theta(z)$
on the compressed space $\breve{\mathcal E}_\theta$.

We next prove coercivity. For
$q\in\mathcal V_\theta^\perp$, the vertex Steklov gap gives
$\lambda_{v,\theta}(q,q)\ge c h^{-1}||q||_{\mathcal E}^2$. The edge Dirichlet-to-Neumann form is non-negative: $\lambda_{e,\theta}(q,q)\ge 0$. Using \eqref{eq:Btheta-bound}, we obtain
\[
    \operatorname{Re}\mathfrak m_{\theta,z}(q,q)=
    \lambda_{v,\theta}(q,q)
    +
    \lambda_{e,\theta}(q,q)
    -
    \operatorname{Re}
    \bigl(
        \mathcal B_{\theta,h}(z)q,q
    \bigr)_{\mathcal E}\ge
    \lambda_{v,\theta}(q,q)
    +
    \lambda_{e,\theta}(q,q)
    -
    C_K||q||_{\mathcal E}^2 .
\]
By the gap estimate,
\[
    C_K||q||_{\mathcal E}^2
    \le
    \frac{C_Kh}{c}
    \lambda_{v,\theta}(q,q).
\]
Therefore, after decreasing $h_0$ if necessary, for all
$0<h<h_0$ we have $\frac{C_Kh}{c}\le \frac12$, and hence
\begin{equation}
\label{eq:t-coercive-V}
    \operatorname{Re}\mathfrak m_{\theta,z}(q,q)
    \ge
    \frac12\lambda_{v,\theta}(q,q)
    +
    \lambda_{e,\theta}(q,q).
\end{equation}
In particular, using again the Steklov gap, the right-hand side of \eqref{eq:t-coercive-V} controls the natural form
norm on $\mathcal V_\theta^\perp$. Thus
$\mathfrak m_{\theta,z}$ is coercive on $\mathcal V_\theta^\perp$,
uniformly in $\theta\in Q'$ and $z\in K$.

Let $f\in\breve{\mathcal E}_\theta$. By the Lax--Milgram lemma applied to
the coercive form $\mathfrak m_{\theta,z}$, there exists a unique
$q\in\mathcal V_\theta^\perp$ such that $\mathfrak m_{\theta,z}(q,\psi)
    =
    (f,\psi)_{\mathcal E}$ for all $\psi\in\mathcal V_\theta^\perp$. By the definition of the associated operator,
$q\in D(T_{\theta,z})$ and $T_{\theta,z}q=f$. Hence $T_{\theta,z}$ is bijective on $\breve{\mathcal E}_\theta$. Since
$T_{\theta,z}=-\mathbb D_\theta(z)$, it follows that
$\mathbb D_\theta(z)$ is also bijective on $\breve{\mathcal E}_\theta$.

It remains to obtain the inverse estimate. Taking $\psi=q$ in the weak
equation and using the lower bound before \eqref{eq:t-coercive-V} gives
\[
    \frac{c}{2h}||q||_{\mathcal E}^2
    \le
    \operatorname{Re}\mathfrak m_{\theta,z}(q,q)
    =
    \operatorname{Re}(f,q)_{\mathcal E}
    \le
    ||f||_{\mathcal E}||q||_{\mathcal E}.
\]
Therefore $||q||_{\mathcal E}
    \le
    Ch||f||_{\mathcal E}$. Since $q=T_{\theta,z}^{-1}f=-\mathbb D_\theta(z)^{-1}f$, we conclude that $
        ||\mathbb D_\theta(z)^{-1}||_{\breve{\mathcal E}_\theta\to\breve{\mathcal E}_\theta}
    \le
    Ch$. Equivalently,
\[
    \Bigl\Vert
        \bigl(
            \mathcal P_\theta^\perp
            M_\theta(z)
            \mathcal P_\theta^\perp
        \bigr)^{-1}
    \Bigr\Vert_{\breve{\mathcal E}_\theta\to\breve{\mathcal E}_\theta}
    \le
    Ch ,
\]
uniformly for $z\in K$, $\theta\in Q'$, and $0<h<h_0$. The proof is
complete.

\end{proof}

\subsubsection{Uniform control of the off-diagonal Weyl blocks}
\label{sec:offdiag-weyl}

The edge Dirichlet-to-Neumann operator is unbounded on
$\mathcal E=L^2(D_{\mathrm{cell}})$. The Schur reduction requires uniform
bounds only for the two cross-blocks involving the retained vertex channel.
In the straight cylindrical geometry these blocks involve only the transverse
ground-mode sector.

Let $\mathscr I$ be the finite set of edge--vertex incidences in the periodic quotient. Thus an unfolded edge contributes one incidence for each endpoint contact plate. For $\iota=(e,\sigma)\in\mathscr I$, $\sigma\in\{-,+\}$, denote the corresponding contact plate by $D_\iota^h\simeq h\omega_e$. After the transverse magnetic gauge used in Section~5.1, let $\phi_{\iota,\theta}^h$ be the normalised transverse ground mode on $D_\iota^h$, including the constant incidence-dependent Floquet phase. We set $\mathcal E_0:=\operatorname{span}\{\phi_{\iota,\theta}^h:\,\iota\in\mathscr I\},$ $P_0:\mathcal E\to\mathcal E_0.$ The dimension of $\mathcal E_0$ is independent of $h$. Since the retained vertex trace is proportional to $e^{-i\theta\cdot x}$ on every contact plate,
\begin{equation}
\label{eq:P0P}
\hat{\mathcal E}_\theta\subset\mathcal E_0,
\qquad
\mathcal P_\theta=P_0\mathcal P_\theta=\mathcal P_\theta P_0.
\end{equation}

\begin{prop}[Edge cross-block estimate]
\label{prop:edge-cross-blocks}
Under the straight-cylinder assumptions of Section~5.1, the edge Dirichlet-to-Neumann cross forms extend uniquely to bounded operators
\[
\mathcal P_\theta\Lambda_{e,\theta}\mathcal P_\theta^\perp:
\breve{\mathcal E}_\theta\to\hat{\mathcal E}_\theta,
\qquad
\mathcal P_\theta^\perp\Lambda_{e,\theta}\mathcal P_\theta:
\hat{\mathcal E}_\theta\to\breve{\mathcal E}_\theta,
\]
and
\begin{equation}
\label{eq:edge-cross-block-bounds}
\bigl\Vert\mathcal P_\theta\Lambda_{e,\theta}\mathcal P_\theta^\perp\bigr\Vert+
\bigl\Vert\mathcal P_\theta^\perp\Lambda_{e,\theta}\mathcal P_\theta\bigr\Vert
\le C,
\end{equation}
uniformly in $h$ and $\theta$.
\end{prop}

\begin{proof}
Fix an unfolded edge
$X_{\mathrm{cell},e}^h=(0,l_e)\times h\omega_e$, and write
$\theta_e:=t_e\cdot\theta$.  After removal of the transverse magnetic
phase, expand the boundary data and the harmonic extension in the Neumann
eigenbasis $\{\chi_k^h\}_{k\ge0}$ of $h\omega_e$, where
\[
-\Delta_y\chi_k^h=\mu_k^h\chi_k^h,
\qquad
\mu_0^h=0,
\qquad
\mu_k^h=h^{-2}\mu_k\quad(k\ge1).
\]
If $a_k,b_k$ are the endpoint coefficients of the $k$-th transverse mode, its longitudinal coefficient satisfies
\[
-\left(\partial_s+i\theta_e\right)^2u_k+\mu_k^hu_k=0,
\qquad u_k(0)=a_k,\quad u_k(l_e)=b_k.
\]
Thus different transverse indices do not couple. Equivalently, in the form sense,
\begin{equation}
\label{eq:edge-mode-orthogonality}
\lambda_{e,\theta}(g_0,g_\perp)=0,
\qquad
 g_0\in\mathcal E_0,
\quad g_\perp\in(I-P_0)\mathcal E\cap H^{1/2}(D_{\mathrm{cell}}).
\end{equation}

It remains to estimate the restriction of the edge Dirichlet-to-Neumann form to $\mathcal E_0$. For the transverse ground mode, after the longitudinal gauge $v(s)=e^{i\theta_es}u_0(s)$, one has $-v''=0$ with
\[
v(0)=a_0,
\qquad
v(l_e)=e^{i\theta_el_e}b_0.
\]
Hence
\[
v(s)=\left(1-\frac{s}{l_e}\right)a_0+
\frac{s}{l_e}e^{i\theta_el_e}b_0,
\]
and therefore
\[
|v'(0)|+|v'(l_e)|
\le \frac{2}{l_e}(|a_0|+|b_0|)
\le \frac{2}{l_-}(|a_0|+|b_0|).
\]
Returning to the magnetic variables only multiplies endpoint quantities by phases of modulus one. Thus the ground-mode endpoint Dirichlet-to-Neumann matrix is uniformly bounded. Summing over the finitely many edges in the cell gives a bounded operator $\Lambda_{e,\theta}^{(0)}=P_0\Lambda_{e,\theta}P_0$ on $\mathcal E_0$, in the sense of the restricted form, with
\begin{equation}
\label{eq:edge-ground-dtn-bound}
\Vert\Lambda_{e,\theta}^{(0)}\Vert_{\mathcal E_0\to\mathcal E_0}\le C.
\end{equation}

Let $q\in\breve{\mathcal E}_\theta\cap H^{1/2}(D_{\mathrm{cell}})$ and $p\in\hat{\mathcal E}_\theta$. Decompose $q=P_0q+(I-P_0)q$. By \eqref{eq:P0P} and \eqref{eq:edge-mode-orthogonality}, one has
\[
-\lambda_{e,\theta}(q,p)
=(\Lambda_{e,\theta}^{(0)}P_0q,p)_{\mathcal E},
\]
and hence $|\lambda_{e,\theta}(q,p)|
\le C\Vert q\Vert_{\mathcal E}\Vert p\Vert_{\mathcal E}.$
Density gives the first bounded cross-block. The second follows in the same way, or from symmetry of the edge Dirichlet-to-Neumann form. This proves \eqref{eq:edge-cross-block-bounds}.
\end{proof}

\begin{prop}[Off-diagonal Weyl estimate]
\label{prop:weyl-cross-blocks}
The off-diagonal blocks of the Weyl function, understood as their bounded extensions, satisfy
\begin{equation}
\label{eq:off-diagonal-bounds}
\Vert\mathbb B_\theta(z)\Vert_{\breve{\mathcal E}_\theta\to\hat{\mathcal E}_\theta}
+
\Vert\mathbb E_\theta(z)\Vert_{\hat{\mathcal E}_\theta\to\breve{\mathcal E}_\theta}
\le C
\end{equation}
uniformly for $\theta\in Q'$ and $z\in K$.
\end{prop}

\begin{proof}
Because $\mathcal P_\theta$ is a spectral projection of $\Lambda_{v,\theta}$, one has $\mathcal P_\theta\Lambda_{v,\theta}\mathcal P_\theta^\perp=0,$ $\mathcal P_\theta^\perp\Lambda_{v,\theta}\mathcal P_\theta=0.$
The edge cross-blocks are bounded by Proposition~\ref{prop:edge-cross-blocks}. The remaining part of the Weyl function is the bounded operator $\mathcal B_{\theta,h}(z)$ from \eqref{eq:Btheta-bound}. Therefore, $\mathbb B_\theta(z)=
\mathcal P_\theta\Lambda_{e,\theta}\mathcal P_\theta^\perp
+
\mathcal P_\theta\mathcal B_{\theta,h}(z)\mathcal P_\theta^\perp,$
with the analogous formula for $\mathbb E_\theta(z)$, which proves the assertion.
\end{proof}

\subsection{Schur complement reduction of the Weyl-function inverse}
\label{step:schur-complement-reduction}

We now reduce the inverse of the Weyl function to the retained vertex channel. The Dirichlet-to-Neumann part of $M_\theta(z)$ is unbounded on $\mathcal E$. We therefore interpret the block notation through the closed Weyl form; Proposition~\ref{prop:weyl-cross-blocks} shows that the cross terms admit bounded operator representatives.

We first estimate the low block. Since $\mathcal P_\theta\Lambda_{v,\theta}\mathcal P_\theta=0$,
\[
\mathbb A_\theta(z)=
\mathcal P_\theta\Lambda_{e,\theta}\mathcal P_\theta
+
z\mathcal P_\theta\Pi_\theta^*(I-z\mathcal A_{0,\theta}^{-1})^{-1}\Pi_\theta\mathcal P_\theta.
\]
The standard Nevanlinna identity
\[
M_\theta(z)-M_\theta(\zeta)^*
=(z-\overline\zeta)S_{\theta,\zeta}^*S_{\theta,z}
\]
gives, for $q\in\hat{\mathcal E}_\theta$,
\[
\operatorname{Im}(\mathbb A_\theta(z)q,q)_{\mathcal E}
=(\operatorname{Im}z)\Vert S_{\theta,z}q\Vert_{\mathcal H}^2.
\]
The retained channel is uniformly non-degenerate. Indeed, let $m$ be the
number of incidence plates in $D_{\mathrm{cell}}$. By
Assumption~\ref{ass:standing-geometry},
\[
 |D_{\mathrm{cell}}|=mh^2,
 \qquad
 |X^h_{\mathrm{cell},v}|=h^2|X_v^0|.
\]
The vertex lift of the normalised trace $\psi_{v,\theta}$ is
$|D_{\mathrm{cell}}|^{-1/2}e^{-i\theta\cdot x}$, and hence
\[
 \Vert\Pi_{v,\theta}\psi_{v,\theta}\Vert_{\mathcal H_v}^2
 =\frac{|X^h_{\mathrm{cell},v}|}{|D_{\mathrm{cell}}|}
 =\frac{|X_v^0|}{m}.
\]
Since $\hat{\mathcal E}_\theta=\operatorname{span}\{\psi_{v,\theta}\}$, it follows that $\Vert \Pi_{v,\theta}q\Vert_{\mathcal H_v}\ge c_0\Vert q\Vert_{\mathcal E},$ $q\in\hat{\mathcal E}_\theta,$
with $c_0=(|X_v^0|/m)^{1/2}$, independently of $h$ and $\theta$. Moreover,
\[
P_vS_{\theta,z}q=(I-z\mathcal A_{0,v,\theta}^{-1})^{-1}\Pi_{v,\theta}q
=\Pi_{v,\theta}q+O(h)q
\]
by \eqref{A0estimate}. Hence $\Vert S_{\theta,z}q\Vert\ge c\Vert q\Vert$ for $h$ sufficiently small. Since $|\operatorname{Im}z|$ is bounded away from zero on $K$, and $\hat{\mathcal E}_\theta$ is one-dimensional,
\begin{equation}
\label{eq:A-block-inverse-bound}
\Vert\mathbb A_\theta(z)^{-1}\Vert_{\hat{\mathcal E}_\theta\to\hat{\mathcal E}_\theta}\le C.
\end{equation}

We next make the block representation precise.

\begin{lemma}[Operator associated with a block form]
\label{lem:block-form-operator}
Let $\mathscr H=\mathscr H_0\oplus\mathscr H_1$, where $\mathscr H_0$ is finite-dimensional. Let $\mathscr V_1\subset\mathscr H_1$ be dense, and let $\mathfrak d$ be a closed sectorial form on $\mathscr V_1$ with associated operator $D$. Suppose $A\in\mathcal B(\mathscr H_0)$, $B\in\mathcal B(\mathscr H_1,\mathscr H_0)$, and $C\in\mathcal B(\mathscr H_0,\mathscr H_1)$. If on $\mathscr V=\mathscr H_0\oplus\mathscr V_1$ one sets
\[
\mathfrak m(x+y,x'+y')
=(Ax,x')+(By,x')+(Cx,y')+\mathfrak d(y,y'),
\]
then the operator $M$ associated with $\mathfrak m$ satisfies
\[
D(M)=\mathscr H_0\oplus D(D),
\qquad
M=\begin{pmatrix}A&B\\ C&D\end{pmatrix}
\quad\text{on this domain}.
\]
\end{lemma}

\begin{proof}
Let $x\in\mathscr H_0$, $y\in\mathscr V_1$. By the definition of the operator associated with a closed sectorial form, $x+y\in D(M)$ if and only if there is $f_0+f_1\in\mathscr H_0\oplus\mathscr H_1$ such that
\[
\mathfrak m(x+y,x'+y')=(f_0,x')+(f_1,y')
\]
for all $x'\in\mathscr H_0$, $y'\in\mathscr V_1$. Taking $y'=0$ gives $f_0=Ax+By$. Taking $x'=0$ gives
\[
\mathfrak d(y,y')=(f_1-Cx,y')\qquad\forall y'\in\mathscr V_1,
\]
which is equivalent to $y\in D(D)$ and $Dy=f_1-Cx$. The assertion follows.
\end{proof}

The sign convention for the Dirichlet-to-Neumann map is non-positive.  We
therefore apply Lemma~\ref{lem:block-form-operator} to
$-M_\theta(z)$, which is associated with the closed sectorial form
\[
\mathfrak m_{\theta,z}(q,\varphi)
=\lambda_{v,\theta}(q,\varphi)+\lambda_{e,\theta}(q,\varphi)
-(\mathcal B_{\theta,h}(z)q,\varphi)_{\mathcal E},
\qquad
D(\mathfrak m_{\theta,z})=H^{1/2}(D_{\mathrm{cell}}).
\]
Since $\hat{\mathcal E}_\theta$ is generated by the smooth rigid trace, one has $H^{1/2}(D_{\mathrm{cell}})
=\hat{\mathcal E}_\theta\oplus
\bigl(H^{1/2}(D_{\mathrm{cell}})\cap\breve{\mathcal E}_\theta\bigr).$
The restriction of $\mathfrak m_{\theta,z}$ to
$H^{1/2}(D_{\mathrm{cell}})\cap\breve{\mathcal E}_\theta$
in both arguments has associated operator $-\mathbb D_\theta(z)$, while
Proposition~\ref{prop:weyl-cross-blocks} supplies bounded representatives of
the cross forms. Lemma~\ref{lem:block-form-operator}, followed by multiplication
by $-1$, therefore yields
\begin{equation}
\label{eq:M-domain-block}
D(M_\theta(z))
=\hat{\mathcal E}_\theta\oplus D(\mathbb D_\theta(z)),
\end{equation}
and, on this domain,
\begin{equation}
\label{eq:M-block-rigorous}
M_\theta(z)=
\begin{pmatrix}
\mathbb A_\theta(z)&\mathbb B_\theta(z)\\[0.3em]
\mathbb E_\theta(z)&\mathbb D_\theta(z)
\end{pmatrix}.
\end{equation}

We use the following Schur complement formula with an unbounded high block.

\begin{lemma}[Schur complement with a closed high block]
\label{lem:unbounded-schur}
Let $M$ have the form in Lemma~\ref{lem:block-form-operator}. Suppose $A$ and $D$ are bijective with bounded inverses. Put $S:=D-CA^{-1}B,$ $D(S)=D(D).$
If $\Vert D^{-1}CA^{-1}B\Vert<1$, then $S$ and $M$ are bijective,
\[
S^{-1}=(I-D^{-1}CA^{-1}B)^{-1}D^{-1},
\]
and
\[
M^{-1}=
\begin{pmatrix}
A^{-1}+A^{-1}BS^{-1}CA^{-1}&-A^{-1}BS^{-1}\\[0.4em]
-S^{-1}CA^{-1}&S^{-1}
\end{pmatrix}.
\]
\end{lemma}

\begin{proof}
Set $K_0=CA^{-1}B\in\mathcal B(\mathscr H_1)$. On $D(D)$, we have  $S=D(I-D^{-1}K_0).$
The factor in parentheses is invertible by the Neumann series. For $f\in\mathscr H_1$, put
$y=(I-D^{-1}K_0)^{-1}D^{-1}f$. Then
\[
y=D^{-1}(f+K_0y)\in D(D),
\]
and consequently $Sy=f$. If $Sy=0$, then
$(I-D^{-1}K_0)y=0$, so $y=0$. This proves the asserted formula for
$S^{-1}$. Solving $Ax+By=f_0$ and $Cx+Dy=f_1$ by eliminating $x$ then gives the displayed inverse of $M$. All compositions are well defined because $S^{-1}$ maps into $D(D)$.
\end{proof}

We now apply the lemma. By Theorem~\ref{thm:D-inv}, Proposition~\ref{prop:weyl-cross-blocks}, and \eqref{eq:A-block-inverse-bound},
\[
\left\Vert
\mathbb D_\theta(z)^{-1}
\mathbb E_\theta(z)
\mathbb A_\theta(z)^{-1}
\mathbb B_\theta(z)
\right\Vert_{\breve{\mathcal E}_\theta\to\breve{\mathcal E}_\theta}
\le Ch.
\]
Hence this norm is at most $1/2$ for small $h$. Define
\[
\mathbb S_\theta(z)
:=\mathbb D_\theta(z)-
\mathbb E_\theta(z)\mathbb A_\theta(z)^{-1}\mathbb B_\theta(z),
\qquad
D(\mathbb S_\theta(z))=D(\mathbb D_\theta(z)).
\]
Then
\begin{equation}
\label{eq:Schur-inverse-bound}
\mathbb S_\theta(z)^{-1}
=\left(I-\mathbb D_\theta(z)^{-1}\mathbb E_\theta(z)\mathbb A_\theta(z)^{-1}\mathbb B_\theta(z)\right)^{-1}\mathbb D_\theta(z)^{-1},
\qquad
\Vert\mathbb S_\theta(z)^{-1}\Vert\le Ch.
\end{equation}
Lemma~\ref{lem:unbounded-schur} and the preceding bounds now imply
\[
M_\theta(z)^{-1}=
\begin{pmatrix}
\mathbb A_\theta(z)^{-1}&0\\[0.3em]0&0
\end{pmatrix}+O_{\mathcal E\to\mathcal E}(h),
\]
uniformly in $\theta\in Q'$, $z\in K$. Equivalently,
\begin{equation}
\label{eq:weyl-schur-final}
\left\Vert
M_\theta(z)^{-1}
-\mathcal P_\theta
(\mathcal P_\theta M_\theta(z)\mathcal P_\theta)^{-1}
\mathcal P_\theta
\right\Vert_{\mathcal E\to\mathcal E}
\le Ch.
\end{equation}

Substitution into Kre\u{\i}n's formula yields the first resolvent reduction.

\begin{theorem}
\label{thm:first-reduction}
There exist constants $C>0$, $h_0>0$, independent of $\theta\in Q'$ and $z\in K$, such that for $0<h<h_0$,
\begin{equation}
\label{eqfirstapp}
\left\Vert (\mathcal A^\theta_{0,\mathbb I}-z)^{-1}
-(\mathcal A_{\mathcal P_\theta^\perp,\mathcal P_\theta}-z)^{-1}
\right\Vert_{\mathcal H\to\mathcal H}
\le Ch.
\end{equation}
Here $\mathcal A_{\mathcal P_\theta^\perp,\mathcal P_\theta}$ is the almost solvable extension determined by
$\mathcal P_\theta^\perp\Gamma_{0,\theta}u=0$, $\mathcal P_\theta\Gamma_{1,\theta}u=0$.  This extension is self-adjoint and non-negative.
\end{theorem}

\begin{proof}
We first justify self-adjointness.  Consider the energy form
\[
 \mathfrak b_{\theta,h}[u,v]
 :=\sum_{j\in\{v,e\}}\int_{X_{\mathrm{cell},j}^h}
 (\nabla+i\theta)u_j\cdot\overline{(\nabla+i\theta)v_j}\,dx
\]
on pairs in
$H^1(X_{\mathrm{cell},v}^h)\oplus H^1(X_{\mathrm{cell},e}^h)$ with a
common interface trace in $\hat{\mathcal E}_\theta$.  This domain is dense
in $\mathcal H$ and closed in the product $H^1$-space.  The form norm
is equivalent to the product $H^1$-norm, so the form is closed and
non-negative.  Let $u$ belong to the domain of its associated operator,
with image $f$ and common trace $q\in\hat{\mathcal E}_\theta$.
Testing with zero interface traces gives
$u=\mathcal A_{0,\theta}^{-1}f+\Pi_\theta q$.  Since
$\hat{\mathcal E}_\theta\subset D(\Lambda_\theta)$ by
\eqref{eq:M-domain-block}, testing with the retained harmonic lifts gives
$\mathcal P_\theta(\Pi_\theta^*f+\Lambda_\theta q)=0$.
Conversely, these identities give the form equation by decomposing each
test pair into a zero-trace pair and a retained harmonic lift.  Thus the
associated self-adjoint operator is precisely
$\mathcal A_{\mathcal P_\theta^\perp,\mathcal P_\theta}$.

Equation~\eqref{eq:weyl-schur-final} and the uniform boundedness of $S_{\theta,z}$ and $S_{\theta,\bar z}$ give, by the Kre\u{\i}n formula,
\[
(\mathcal A^\theta_{0,\mathbb I}-z)^{-1}
=(\mathcal A_{0,\theta}-z)^{-1}
-S_{\theta,z}\mathcal P_\theta
(\mathcal P_\theta M_\theta(z)\mathcal P_\theta)^{-1}
\mathcal P_\theta S_{\theta,\bar z}^*
+O_{\mathcal H\to\mathcal H}(h).
\]
The leading term is the Kre\u{\i}n formula for the extension whose boundary
trace is restricted to $\hat{\mathcal E}_\theta$: if
\[
\hat S_{\theta,z}:=S_{\theta,z}|_{\hat{\mathcal E}_\theta},
\qquad
\hat M_\theta(z):=\mathcal P_\theta M_\theta(z)|_{\hat{\mathcal E}_\theta},
\]
then
\[
(\mathcal A_{\mathcal P_\theta^\perp,\mathcal P_\theta}-z)^{-1}
=(\mathcal A_{0,\theta}-z)^{-1}
-\hat S_{\theta,z}\hat M_\theta(z)^{-1}\hat S_{\theta,\bar z}^*.
\]
Comparison proves \eqref{eqfirstapp}.
\end{proof}

\subsection{Reduction of the low-energy Weyl function}

We now use the representation
\begin{equation}\label{eq:M-representation}
M_\theta(z)
=
\Lambda_\theta
+
z\Pi_\theta^*(I-z\mathcal A_{0,\theta}^{-1})^{-1}\Pi_\theta,
\end{equation}
valid for $z\in \rho(\mathcal A_{0,\theta})$.  For
$j\in\{v,e\}$, set $\hat\Pi_{j,\theta}:=\Pi_{j,\theta}\mathcal P_\theta,$ $\hat\Lambda_{j,\theta}:=
 \mathcal P_\theta\Lambda_{j,\theta}\mathcal P_\theta,$
and define
\[
 \hat\Pi_\theta:=\Pi_\theta\mathcal P_\theta
 =\hat\Pi_{v,\theta}\oplus\hat\Pi_{e,\theta},
 \qquad
 \hat S_{j,z,\theta}
 :=(I-z\mathcal A_{0,j,\theta}^{-1})^{-1}
   \hat\Pi_{j,\theta}.
\]
Thus
$\hat S_{\theta,z}=\hat S_{v,z,\theta}\oplus
\hat S_{e,z,\theta}$, and restriction of
\eqref{eq:M-representation} to $\hat{\mathcal E}_\theta$ gives
\[
 \hat M_\theta(z)
 :=\mathcal P_\theta M_\theta(z)\mathcal P_\theta
 =\hat\Lambda_{v,\theta}+\hat\Lambda_{e,\theta}
  +z\hat\Pi_\theta^*(I-z\mathcal A_{0,\theta}^{-1})^{-1}
    \hat\Pi_\theta.
\]
Since $\hat\Lambda_{v,\theta}=0$, we use
\eqref{A0estimate} to replace the vertex resolvent contribution by its
leading term and define
\[
\hat Q_\theta(z)
:=
z\hat\Pi_{v,\theta}^*\hat\Pi_{v,\theta}
+
\hat M_{e,\theta}(z),
\qquad \textup{where} \qquad 
\hat M_{e,\theta}(z)
=
\hat\Lambda_{e,\theta}
+
z\hat\Pi_{e,\theta}^*(I-z\mathcal A_{0,e,\theta}^{-1})^{-1}\hat\Pi_{e,\theta}.
\]
The operator $\hat Q_\theta(z)$ acts on the one-dimensional boundary space
$\hat{\mathcal{E}}_\theta$. Its vertex term is
$z|X_v^0|/m$ times the identity, where $m$ is the number of contact
plates. The graph-normalised generator is fixed in Subsection~5.7;
evaluated on that generator, the corresponding sesquilinear form has vertex
contribution $z|X_v^0|$.
 
We use the separation of $K$ from the real axis to prove uniform
invertibility of $\hat Q_\theta(z)$ on
$\hat{\mathcal E}_\theta$.  The Nevanlinna identity for the edge
Weyl function gives
\[
 \operatorname{Im}
 \bigl(\hat M_{e,\theta}(z)q,q\bigr)_{\mathcal E}
 = (\operatorname{Im}z)
   \Vert\hat S_{e,z,\theta}q\Vert_{\mathcal H_e}^2.
\]
Consequently, for every $q\in\hat{\mathcal E}_\theta$,
\[
    \operatorname{Im}
    \bigl(
        \hat Q_\theta(z)q,q
    \bigr)_{\mathcal E}
    =
    \operatorname{Im}
    \bigl(
        \hat M_{e,\theta}(z)q,q
    \bigr)_{\mathcal E}
    +
    \operatorname{Im}z\,
    ||\hat\Pi_{v,\theta}q||_{\mathcal H_v}^2     =
    (\operatorname{Im}z)
    \left(
        ||\hat S_{e,z,\theta}q||_{\mathcal H_e}^2
        +
        ||\hat\Pi_{v,\theta}q||_{\mathcal H_v}^2
    \right).
\]
The lower bound
$\Vert\hat\Pi_{v,\theta}q\Vert\geq c\Vert q\Vert$ established before
\eqref{eq:A-block-inverse-bound} therefore gives
$\bigl|(\hat Q_\theta(z)q,q)_{\mathcal E}\bigr|
\geq c_K\Vert q\Vert_{\mathcal E}^2$.

Since $\hat{\mathcal E}_\theta$ is one-dimensional, this implies that
$\hat Q_\theta(z)$ is invertible and $||\hat Q_\theta(z)^{-1}||_{\hat{\mathcal E}_\theta\to\hat{\mathcal E}_\theta}
    \leq C_K$
uniformly for $z\in K$, $\theta\in Q'$, and $h$ sufficiently small.

Set
\[
 R_{\theta,h}(z):=\hat M_\theta(z)-\hat Q_\theta(z)
 =z\hat\Pi_{v,\theta}^*
  \bigl[(I-z\mathcal A_{0,v,\theta}^{-1})^{-1}-I\bigr]
  \hat\Pi_{v,\theta}.
\]
By \eqref{A0estimate} and the uniform boundedness of
$\hat\Pi_{v,\theta}$, one has $\Vert R_{\theta,h}(z)\Vert_{
 \hat{\mathcal E}_\theta\to\hat{\mathcal E}_\theta}
 \leq Ch$
uniformly for $\theta\in Q'$ and $z\in K$. Hence
\[
    \hat M_\theta(z)
    =
    \hat Q_\theta(z)
    \left(
        I+\hat Q_\theta(z)^{-1} R_{\theta,h}(z)
    \right).
\]
Since $\hat Q_\theta(z)^{-1}$ is uniformly bounded, $\Vert\hat Q_\theta(z)^{-1}R_{\theta,h}(z)\Vert\leq Ch.$
Therefore, for all sufficiently small $h$, the operator
$I+\hat Q_\theta(z)^{-1}R_{\theta,h}(z)$ is invertible by a Neumann
series, with inverse uniformly bounded. Consequently,
\[
\begin{aligned}
    \hat M_\theta(z)^{-1}
    =
    \bigl(
        I+\hat Q_\theta(z)^{-1} R_{\theta,h}(z)
    \bigr)^{-1}
    \hat Q_\theta(z)^{-1} =
    \hat Q_\theta(z)^{-1}
    +
    O(h),
\end{aligned}
\]
in the operator norm on $\hat{\mathcal E}_\theta$, uniformly for
$\theta\in Q'$ and $z\in K$.

\subsection{Approximation of the thin-fibre resolvent}

The fibre space was introduced in the order
$\mathcal H=\mathcal H_v\oplus\mathcal H_e$, whereas the effective block
formula is more naturally written in the order
$\mathcal H_e\oplus\mathcal H_v$.  We therefore use the fixed flip unitary
\[
   \mathfrak F:\mathcal H_v\oplus\mathcal H_e
   \longrightarrow\mathcal H_e\oplus\mathcal H_v,
   \qquad \mathfrak F(f_v,f_e)=(f_e,f_v).
\]
Following the preceding Weyl reduction, define the operator-valued function
$\widetilde{\mathcal R}_{\theta,h}(z)$ on
$\mathcal H_e\oplus\mathcal H_v$ by
\begin{equation}\label{R_eff}
    \widetilde{\mathcal R}_{\theta,h}(z)
:=
\begin{pmatrix}
(\mathcal A_{0,e,\theta}-z)^{-1}
-
\hat S_{e,z,\theta}\hat Q_\theta(z)^{-1}\hat S_{e,\bar z,\theta}^*
&
-\hat S_{e,z,\theta}\hat Q_\theta(z)^{-1}\hat\Pi_{v,\theta}^*
\\[0.8em]
-\hat\Pi_{v,\theta}\hat Q_\theta(z)^{-1}\hat S_{e,\bar z,\theta}^*
&
-\hat\Pi_{v,\theta}\hat Q_\theta(z)^{-1}\hat\Pi_{v,\theta}^*
\end{pmatrix}.
\end{equation}
On the original fibre space $\mathcal H$ we set
\begin{equation}\label{eq:R-physical-order}
   \mathcal R_{\theta,h}(z)
   :=\mathfrak F^*\widetilde{\mathcal R}_{\theta,h}(z)\mathfrak F.
\end{equation}
Combining Theorem~\ref{thm:first-reduction} with the estimate $\hat M_\theta(z)^{-1}=\hat Q_\theta(z)^{-1}+O(h)$, we obtain the following approximation.

\begin{theorem}\label{thm:R-approx}
There exist a constant $C>0$, depending only on $K$ and the
constants in Assumption~\ref{ass:standing-geometry}, and
$h_0>0$ such that, for $0<h<h_0$,
\[
||(\mathcal A^\theta_{0,\mathbb I}-z)^{-1}-\mathcal R_{\theta,h}(z)||_{\mathcal H\to\mathcal H}
\le Ch,
\qquad
\theta\in Q',\quad z\in K.
\]
\end{theorem}

\begin{proof}
Theorem~\ref{thm:first-reduction} reduces the assertion to comparison with
\[
 (\mathcal A_{\mathcal P_\theta^\perp,\mathcal P_\theta}-z)^{-1}
 =(\mathcal A_{0,\theta}-z)^{-1}
  -\hat S_{\theta,z}\hat M_\theta(z)^{-1}\hat S_{\theta,\bar z}^*.
\]
By \eqref{A0estimate} and the preceding Weyl reduction, uniformly in
$\theta\in Q'$ and $z\in K$,
\[
 (\mathcal A_{0,v,\theta}-z)^{-1}=O(h),
 \quad \hat S_{v,z,\theta}=\hat\Pi_{v,\theta}+O(h),
 \quad \hat M_\theta(z)^{-1}=\hat Q_\theta(z)^{-1}+O(h).
\]
The same estimates hold at $\bar z$. All remaining factors are uniformly
bounded. Substituting these three identities into the four blocks of the
displayed resolvent, and applying the flip $\mathfrak F$, gives
$\widetilde{\mathcal R}_{\theta,h}(z)$ from \eqref{R_eff} with an
$O(h)$ remainder in each block. Since there are finitely many blocks,
\[
 \Vert(\mathcal A_{\mathcal P_\theta^\perp,\mathcal P_\theta}-z)^{-1}
       -\mathcal R_{\theta,h}(z)\Vert_{\mathcal H\to\mathcal H}
 \leq Ch.
\]
Combining this with Theorem~\ref{thm:first-reduction} proves the assertion.
\end{proof}

\subsection{The reduced effective fibre operator}

We now realise the reduced formula as the lifted resolvent of a
self-adjoint operator on $\mathcal H_e\oplus\mathbb C$, using the
enlarged-space construction in \cite{cherednichenko2025effectiveII}.
The scalar coordinate represents the rigid vertex mode. We retain the
notation $\hat{\mathcal E}_\theta=\operatorname{span}\{\psi_{v,\theta}\}$
for its boundary trace space.

We first introduce the corresponding reduced edge operator. Let
\[
\mathcal A_{e,\theta,\max}u
:=
-(\nabla+i\theta)\cdot(\nabla+i\theta)u
\]
be the maximal edge operator on $\mathcal H_e=L^2(X^h_{\mathrm{cell},e})$,
with domain
\[
D(\mathcal A_{e,\theta,\max})
:=
D(\mathcal A_{0,e,\theta})\dotplus \operatorname{Ran}\Pi_{e,\theta}.
\]
Retaining only the low-energy boundary mode
$\hat{\mathcal E}_\theta$, we restrict this operator to the subspace
generated by the corresponding edge lifts and define
\[
D(\hat{\mathcal A}_{e,\theta})
:=
D(\mathcal A_{0,e,\theta})\dotplus \Pi_{e,\theta}\hat{\mathcal E}_\theta
\subset D(\mathcal A_{e,\theta,\max}),
\qquad 
\hat{\mathcal A}_{e,\theta}u
:=
\mathcal A_{e,\theta,\max}u,
\qquad
u\in D(\hat{\mathcal A}_{e,\theta}).
\]
Thus $\hat{\mathcal A}_{e,\theta}$ is the edge maximal operator with boundary
data restricted to the one-dimensional effective space $\hat{\mathcal E}_\theta$.

Correspondingly, for $j\in\{v,e\}$ and $k\in\{0,1\}$, set $\hat\Gamma_{k,j,\theta}:=
 \mathcal P_\theta\Gamma_{k,j,\theta}$
on the relevant reduced domains.  In particular, if
$u=u_0+\Pi_{e,\theta}(a\psi_{v,\theta})$, then
$\hat\Gamma_{0,e,\theta}u=a\psi_{v,\theta}
\in\hat{\mathcal E}_\theta$.

Let $P^{\mathrm{res}}_{v,\theta}$ be the orthogonal projection in
$\mathcal H_v$ onto $\operatorname{Ran}(\hat\Pi_{v,\theta})$.  Choose a
partial isometry
\[
 J_\theta:\mathcal H_v\to\mathbb C,
 \qquad
 J_\theta^*J_\theta=P^{\mathrm{res}}_{v,\theta},
 \qquad
 J_\theta J_\theta^*=I_{\mathbb C},
\]
whose restriction to $\operatorname{Ran}(\hat\Pi_{v,\theta})$ is unitary,
and define $\widetilde\Pi_{v,\theta}:=J_\theta\hat\Pi_{v,\theta}:
\hat{\mathcal E}_\theta\to\mathbb C.$
Since $\hat{\mathcal E}_\theta$ is one-dimensional and
$\hat\Pi_{v,\theta}$ is injective, the operator
$\widetilde\Pi_{v,\theta}$ is boundedly invertible.

We first define the reduced energy form.  Let
\[
\begin{aligned}
 \mathcal V_{\theta,h}^{\mathrm{eff}}
 :=\bigl\{(u,\beta)\in
 H^1(X_{\mathrm{cell},e}^h)\oplus\mathbb C:
 &\ u|_{D_{\mathrm{cell}}}\in\hat{\mathcal E}_\theta,\\
 &\ \beta=\widetilde\Pi_{v,\theta}
        (u|_{D_{\mathrm{cell}}})\bigr\},
\end{aligned}
\]
where the edge Sobolev space is understood on the periodic quotient, as in
Section~2.  On this domain set
\begin{equation}\label{eq:reduced-effective-form}
 \mathfrak a_{\theta,h}^{\mathrm{eff}}
 [(u,\beta),(v,\gamma)]
 :=
 \int_{X_{\mathrm{cell},e}^h}
 (\nabla+i\theta)u\cdot
 \overline{(\nabla+i\theta)v}\,dx.
\end{equation}
The trace map is continuous, so
$\mathcal V_{\theta,h}^{\mathrm{eff}}$ is closed in
$H^1(X_{\mathrm{cell},e}^h)\oplus\mathbb C$.  It is dense in
$\mathcal H_e\oplus\mathbb C$: for a prescribed scalar component one may
choose an edge lift with the corresponding retained trace and approximate the
remaining edge datum by functions with zero interface trace.  Since $\theta$
ranges over a bounded set, the norm
\[
 \left(
   \mathfrak a_{\theta,h}^{\mathrm{eff}}
   [(u,\beta),(u,\beta)]
   +\Vert u\Vert_{\mathcal H_e}^2+|\beta|^2
 \right)^{1/2}
\]
is equivalent to the product $H^1\oplus\mathbb C$-norm on this closed
subspace.  Hence $\mathfrak a_{\theta,h}^{\mathrm{eff}}$ is a densely
defined, closed, non-negative form.

We denote its associated self-adjoint operator in
$\mathcal H_e\oplus\mathbb C$ by
$\mathcal A_{\theta,h}^{\mathrm{eff}}$.  The reduced Green identity shows
that its operator domain is
\[
D(\mathcal A_{\theta,h}^{\mathrm{eff}})
=
\bigl\{
(u,\beta)\in D(\hat{\mathcal A}_{e,\theta})\oplus\mathbb C
:
\beta=\widetilde\Pi_{v,\theta}\hat\Gamma_{0,e,\theta}u
\bigr\}.
\]
Since $\Lambda_{v,\theta}$ vanishes on $\hat{\mathcal E}_\theta$, the reduced vertex Dirichlet-to-Neumann term is zero.  The associated operator therefore acts as
\begin{equation}\label{eq:A-eff-h-simplified}
\mathcal A_{\theta,h}^{\mathrm{eff}}
\begin{pmatrix}
u\\[0.3em]
\beta
\end{pmatrix}
=
\begin{pmatrix}
\hat{\mathcal A}_{e,\theta}u\\[0.5em]
-(\widetilde\Pi_{v,\theta}^*)^{-1}\hat\Gamma_{1,e,\theta}u
\end{pmatrix},
\qquad
\beta=\widetilde\Pi_{v,\theta}\hat\Gamma_{0,e,\theta}u.
\end{equation}

Indeed, if $(u,\beta)$ belongs to the displayed domain and
$(v,\gamma)\in\mathcal V_{\theta,h}^{\mathrm{eff}}$, the reduced Green
identity and
$\gamma=\widetilde\Pi_{v,\theta}(v|_{D_{\mathrm{cell}}})$ give
\[
 \mathfrak a_{\theta,h}^{\mathrm{eff}}
 [(u,\beta),(v,\gamma)]
 =
 (\hat{\mathcal A}_{e,\theta}u,v)_{\mathcal H_e}
 -(\hat\Gamma_{1,e,\theta}u,
   v|_{D_{\mathrm{cell}}})_{\hat{\mathcal E}_\theta}=
 (\hat{\mathcal A}_{e,\theta}u,v)_{\mathcal H_e}
 +\bigl(- (\widetilde\Pi_{v,\theta}^*)^{-1}
     \hat\Gamma_{1,e,\theta}u,\gamma\bigr)_{\mathbb C}.
\]
Conversely, let $(u,\beta)$ belong to the domain of the operator
associated with \eqref{eq:reduced-effective-form}, with image $(f_e,g_0)$,
and set $q=u|_{D_{\mathrm{cell}}}\in\hat{\mathcal E}_\theta$.
Testing with $(v,0)$, where $v$ has zero interface trace, gives
\[
 u-\Pi_{e,\theta}q\in D(\mathcal A_{0,e,\theta}),
 \qquad
 \mathcal A_{0,e,\theta}(u-\Pi_{e,\theta}q)=f_e.
\]
Hence $u\in D(\hat{\mathcal A}_{e,\theta})$, and the preceding identity
for general test pairs determines
$g_0=-(\widetilde\Pi_{v,\theta}^*)^{-1}\hat\Gamma_{1,e,\theta}u$.
This proves the asserted domain and action in
\eqref{eq:A-eff-h-simplified}.  The first component is the edge differential
expression restricted to the retained boundary channel, while the second is
the corresponding scalar vertex flux.

To compare the reduced resolvent with \eqref{R_eff}, define the isometry
\[
    \bJ_\theta
    :=
    I_{\mathcal H_e}\oplus J_\theta^*
    :
    \mathcal H_e\oplus\mathbb C
    \longrightarrow
    \mathcal H_e\oplus\mathcal H_v .
\]
Here $J_\theta^*:\mathbb C\to\mathcal H_v$ embeds the scalar
effective vertex coordinate into the one-dimensional resonant vertex
subspace $\operatorname{Ran}(\hat\Pi_{v,\theta})\subset\mathcal H_v$, while
$J_\theta$ vanishes on its orthogonal complement.  The adjoint of
$\bJ_\theta$ is
\[
    \bJ_\theta^*
    :=
    I_{\mathcal H_e}\oplus J_\theta:  \mathcal H_e\oplus\mathcal H_v  \to \mathcal H_e\oplus\mathbb C.
\]

\begin{theorem}
\label{thm:eff-fibre}
For every $h>0$ for which the standing geometric assumptions hold, and
every $\theta\in Q'$, the operator
$\mathcal A_{\theta,h}^{\mathrm{eff}}$ is self-adjoint in the reduced
Hilbert space $\mathcal H_e\oplus\mathbb C$. Moreover,
$K\subset\rho(\mathcal A_{\theta,h}^{\mathrm{eff}})$, and for every
$z\in K$ its resolvent satisfies
\[
    \bJ_\theta
    \bigl(\mathcal A_{\theta,h}^{\mathrm{eff}}-z\bigr)^{-1}
    \bJ_\theta^*
    =
    \widetilde{\mathcal R}_{\theta,h}(z)
\]
as an operator on $\mathcal H_e\oplus\mathcal H_v$, where
$\widetilde{\mathcal R}_{\theta,h}(z)$ is the operator-valued function
defined in \eqref{R_eff}.  Equivalently,
\[
  \mathfrak F^*\bJ_\theta
  \bigl(\mathcal A_{\theta,h}^{\mathrm{eff}}-z\bigr)^{-1}
  \bJ_\theta^*\mathfrak F
  =\mathcal R_{\theta,h}(z)
\]
on the original fibre space $\mathcal H$.
\end{theorem}

\begin{proof}
Self-adjointness and non-negativity follow from the closed form
\eqref{eq:reduced-effective-form}.  Since $K\subset\mathbb C\setminus\mathbb R$,
one has $K\subset\rho(\mathcal A_{\theta,h}^{\mathrm{eff}})$.  It remains
to identify the resolvent.  Fix $z\in K$ and $(f_e,g)\in\mathcal H_e\oplus\mathbb C$. Define $u$ and $\beta$ by
\[
    u
    =
    (\mathcal A_{0,e,\theta}-z)^{-1}f_e
    -
    \hat S_{e,z,\theta}\hat Q_\theta(z)^{-1}
    \left(
        \hat S_{e,\bar z,\theta}^*f_e
        +
        \widetilde\Pi_{v,\theta}^*g
    \right),
\]
and
\[
    \beta
    = \widetilde\Pi_{v,\theta}\hat\Gamma_{0,e,\theta}u= 
    -
    \widetilde\Pi_{v,\theta}\hat Q_\theta(z)^{-1}
    \left(
        \hat S_{e,\bar z,\theta}^*f_e
        +
        \widetilde\Pi_{v,\theta}^*g
    \right).
\]
Equivalently, relative to the decomposition
$\mathcal H_e\oplus\mathbb C$, one has
\[
\bigl(\mathcal A_{\theta,h}^{\mathrm{eff}}-z\bigr)^{-1}
=
\begin{pmatrix}
(\mathcal A_{0,e,\theta}-z)^{-1}
-
\hat S_{e,z,\theta}\hat Q_\theta(z)^{-1}\hat S_{e,\bar z,\theta}^*
&
-\hat S_{e,z,\theta}\hat Q_\theta(z)^{-1}\widetilde\Pi_{v,\theta}^*
\\[0.8em]
-\widetilde\Pi_{v,\theta}\hat Q_\theta(z)^{-1}\hat S_{e,\bar z,\theta}^*
&
-\widetilde\Pi_{v,\theta}\hat Q_\theta(z)^{-1}\widetilde\Pi_{v,\theta}^*
\end{pmatrix}.
\]
The invertibility of $\hat Q_\theta(z)$ shows that these formulae define $u\in D(\hat{\mathcal A}_{e,\theta})$ and $\beta=\widetilde\Pi_{v,\theta}\hat\Gamma_{0,e,\theta}u$. Hence $(u,\beta)\in D(\mathcal A_{\theta,h}^{\mathrm{eff}})$. A direct substitution gives
\[
(\mathcal A_{\theta,h}^{\mathrm{eff}}-z)(u,\beta)=(f_e,g).
\]
The same calculation with zero right-hand side gives injectivity, and therefore
verifies the displayed inverse formula.

Finally, the identities
$J_\theta^*\widetilde\Pi_{v,\theta}=\hat\Pi_{v,\theta}$ and
$\widetilde\Pi_{v,\theta}^*J_\theta=\hat\Pi_{v,\theta}^*$ identify the
lifted block matrix with \eqref{R_eff}. Thus
\[
 \bJ_\theta(\mathcal A_{\theta,h}^{\mathrm{eff}}-z)^{-1}\bJ_\theta^*
 =\widetilde{\mathcal R}_{\theta,h}(z).
\]
Conjugation by $\mathfrak F$ gives the identity on $\mathcal H$.
\end{proof}

\subsection{Passage to the graph fibre operators}

We now collapse the cylindrical edge components to their transverse magnetic
ground modes.  We distinguish each geometric edge from its two endpoint incidences.  The reduction first produces a graph
with the exact finite-$h$ endpoint phases.  We then identify this graph with
the standard lattice Floquet fibre and, separately, compare it with the graph
obtained by retaining the ambient finite-thickness edge tangent.

Let $\mathscr E_{\mathrm{cell}}$ be the finite set of oriented geometric
edges in the periodic quotient and put
\[
   \mathscr I:=\{(e,\sigma):e\in\mathscr E_{\mathrm{cell}},\ 
   \sigma\in\{-,+\}\}.
\]
The incidence $(e,-)$ corresponds to $x=0$, while $(e,+)$
corresponds to $x=l_e$.  As in Section~1, the possible dependence of
$l_e$ and of the unit tangent $t_e$ on $h$ is suppressed.  We set
\[
   \theta_e^{\mathrm{geo}}:=\theta\cdot t_e,
   \qquad
   \theta_{e,h}^{\mathrm{lat}}:=\frac{\theta\cdot n_e}{l_e},
\]
and introduce the two one-dimensional magnetic derivatives
\[
   D_{\theta,e}^{\mathrm{geo}}
   :=\frac{d}{dx}+i\theta_e^{\mathrm{geo}},
   \qquad
   D_{\theta,e}^{\mathrm{lat}}
   :=\frac{d}{dx}+i\theta_{e,h}^{\mathrm{lat}}.
\]
The superscripts distinguish the ambient-geometric phase from the lattice
Floquet phase.  The total phase of the latter along the edge is
$\theta_{e,h}^{\mathrm{lat}}l_e=\theta\cdot n_e$.

For $u_e\in H^1(0,l_e)$, write
\[
   \gamma_eu_e:=
   \begin{pmatrix}u_e(0)\\u_e(l_e)\end{pmatrix}.
\]
For $u_e\in H^2(0,l_e)$ and
$r\in\{\mathrm{geo},\mathrm{lat}\}$, put
\[
   \partial_{\nu,\theta}^{r}u_e:=
   \begin{pmatrix}
      -D_{\theta,e}^{r}u_e(0)\\[0.2em]
       D_{\theta,e}^{r}u_e(l_e)
   \end{pmatrix}.
\]
Then, for $u_e,v_e\in H^2(0,l_e)$,
\[
 \bigl(-(D_{\theta,e}^{r})^2u_e,v_e\bigr)
 -\bigl(u_e,-(D_{\theta,e}^{r})^2v_e\bigr)
 =
 -\bigl(\partial_{\nu,\theta}^{r}u_e,\gamma_ev_e\bigr)_{\mathbb C^2}
 +
 \bigl(\gamma_eu_e,\partial_{\nu,\theta}^{r}v_e\bigr)_{\mathbb C^2}.
\]

For the incidence $\iota=(e,\sigma)$, let $x^h_{e,\sigma}$ be the
centre of the corresponding contact plate in the local torus chart of the
quotient vertex.  Define
\[
   \alpha^h_{e,\sigma}:=e^{-i\theta\cdot x^h_{e,\sigma}},
   \qquad
   a_e^h:=
   \begin{pmatrix}\alpha^h_{e,-}\\ \alpha^h_{e,+}\end{pmatrix},
   \qquad
   \mathbf1:=\begin{pmatrix}1\\1\end{pmatrix}.
\]
Since $|x^h_{e,\sigma}|\leq Ch^{2/3}$ and $\theta\in Q'$,
\begin{equation}\label{eq:incidence-phase-bound}
   \max_{e,\sigma}|\alpha^h_{e,\sigma}-1|
   +\max_e\Vert a_e^h-\mathbf1\Vert_{\mathbb C^2}
   \leq Ch^{2/3}.
\end{equation}
The two entries of $a_e^h$ need not coincide.

On the cylinder
$X^h_{\mathrm{cell},e}\simeq(0,l_e)\times h\omega_e$, let
\[
   \varphi^h_{\theta,e}(y)
   :=|h\omega_e|^{-1/2}e^{-i\theta_{\perp,e}\cdot y},
   \qquad
   \theta_{\perp,e}:=\theta-(\theta\cdot t_e)t_e.
\]
It is the normalised magnetic transverse ground mode.  Define
\begin{equation}\label{eq:ground-mode-isometry-new}
   (\Phi^0_{\theta,h,e}u_e)(x,y)
   :=u_e(x)\varphi^h_{\theta,e}(y),
   \qquad u_e\in L^2(0,l_e).
\end{equation}
Let $G^h_{\mathrm{cell}}$ be the metric graph with edge intervals
$(0,l_e)$ and set $\mathfrak H_{g,h}
   :=L^2(G^h_{\mathrm{cell}})\oplus\mathbb C.$
We introduce two identification maps. The first maps the graph fibre onto
the ground-mode sector of the reduced space:
\begin{equation}\label{eq:graph-reduced-identification}
   \Phi^g_{\theta,h}:=
   \left(\bigoplus_{e\in\mathscr E_{\mathrm{cell}}}
   \Phi^0_{\theta,h,e}\right)\oplus I_{\mathbb C}
   :\mathfrak H_{g,h}\longrightarrow\mathcal H_e\oplus\mathbb C.
\end{equation}
The second embeds the same graph fibre into the original physical fibre:
\begin{equation}\label{eq:graph-thin-identification-new}
   \Phi_{\theta,h}
   :=\mathfrak F^*\bJ_\theta\Phi^g_{\theta,h}
   :\mathfrak H_{g,h}\longrightarrow\mathcal H.
\end{equation}
Both maps are isometries.  The adjoint of the edge part is the magnetic
transverse average.

We next relate the scalar graph coordinate to the retained boundary channel.
Put
\[
   \eta_{\theta,h}:=h^{-1}e^{-i\theta\cdot x}
   \big|_{D_{\mathrm{cell}}}.
\]
Since $|\omega_e|=1$, one has
\[
   \eta_{\theta,h}\big|_{D^h_{e,\sigma}}
   =\alpha^h_{e,\sigma}\varphi^h_{\theta,e}.
\]
Its vertex lift is $h^{-1}e^{-i\theta\cdot x}$, and therefore
\[
   \Vert\hat\Pi_{v,\theta}\eta_{\theta,h}\Vert_{\mathcal H_v}^2
   =h^{-2}|X^h_{\mathrm{cell},v}|=|X_v^0|.
\]
Fix the phase of $J_\theta$ so that, with
$\kappa:=|X_v^0|^{1/2}$,
\begin{equation}\label{eq:graph-scalar-normalisation}
   \widetilde\Pi_{v,\theta}\eta_{\theta,h}=\kappa.
\end{equation}
Consequently, if the retained edge trace is
$c\eta_{\theta,h}$, then its endpoint ground-mode coefficients are
$a_e^hc$ and the scalar component of the reduced operator is
$\beta=\kappa c$.

\subsubsection{The phase-corrected, lattice and ambient-geometric graph fibres}

The exact phase-corrected form has domain
\begin{equation}\label{eq:phase-form-domain}
   \mathcal V^{\mathrm{ph}}_{\theta,h}:=
   \Bigl\{(u,\beta)\in
   \bigoplus_{e\in\mathscr E_{\mathrm{cell}}}H^1(0,l_e)
   \oplus\mathbb C:\ 
   \exists c\in\mathbb C\text{ such that}\ 
   \gamma_eu_e=a_e^h c\ \text{for all }e,
   \quad \beta=\kappa c\Bigr\},
\end{equation}
and
\begin{equation}\label{eq:phase-form}
   \mathfrak a^{\mathrm{ph}}_{\theta,h}
   [(u,\beta),(v,\gamma)]
   :=
   \sum_{e\in\mathscr E_{\mathrm{cell}}}
   \int_0^{l_e}
   D_{\theta,e}^{\mathrm{geo}}u_e\,
   \overline{D_{\theta,e}^{\mathrm{geo}}v_e}\,dx.
\end{equation}

The lattice and ambient-geometric fibres have the common form domain
\begin{equation}
\label{eq:common-graph-form-domain}
   \mathcal V^0_h:=
   \Bigl\{(u,\beta)
   \in\bigoplus_{e\in\mathscr E_{\mathrm{cell}}}H^1(0,l_e)
   \oplus\mathbb C:\ 
   \exists c\in\mathbb C\ \text{such that}\ 
   \gamma_eu_e=\mathbf1 c\ \text{for all }e,
   \quad \beta=\kappa c\Bigr\}.
\end{equation}
On this domain define
\begin{align}
   \mathfrak a^{\mathrm{lat}}_{\theta,h}
   [(u,\beta),(v,\gamma)]
   &:=
   \sum_{e\in\mathscr E_{\mathrm{cell}}}
   \int_0^{l_e}
   D_{\theta,e}^{\mathrm{lat}}u_e\,
   \overline{D_{\theta,e}^{\mathrm{lat}}v_e}\,dx,
   \label{eq:lattice-form}\\
   \mathfrak a^{\mathrm{geo}}_{\theta,h}
   [(u,\beta),(v,\gamma)]
   &:=
   \sum_{e\in\mathscr E_{\mathrm{cell}}}
   \int_0^{l_e}
   D_{\theta,e}^{\mathrm{geo}}u_e\,
   \overline{D_{\theta,e}^{\mathrm{geo}}v_e}\,dx.
   \label{eq:geometric-form}
\end{align}
We denote the associated operators by
$\mathcal A^{\mathrm{ph}}_{\theta,h}$,
$\mathcal A^{\mathrm{lat}}_{\theta,h}$, and
$\mathcal A^{\mathrm{geo}}_{\theta,h}$, respectively.

\begin{prop}[Self-adjoint graph fibres]\label{prop:selfadjoint-graph-fibres}
For every $h>0$ for which the standing geometric assumptions hold, and
every $\theta\in Q'$, the three forms just defined are
densely defined, closed, symmetric and non-negative in
$\mathfrak H_{g,h}$.  Hence the corresponding graph operators are
self-adjoint and non-negative. In each case, the operator domain consists
of the elements of the corresponding form domain for which
$u_e\in H^2(0,l_e)$ on every edge.

The phase-corrected operator acts as
\[
 \mathcal A^{\mathrm{ph}}_{\theta,h}
 \begin{pmatrix}u\\ \kappa c\end{pmatrix}
 =
 \begin{pmatrix}
    (-(D_{\theta,e}^{\mathrm{geo}})^2u_e)_e\\[0.4em]
    \displaystyle
    \kappa^{-1}
    \sum_{e\in\mathscr E_{\mathrm{cell}}}
    (a_e^h)^*\partial_{\nu,\theta}^{\mathrm{geo}}u_e
 \end{pmatrix},
 \qquad \gamma_eu_e=a_e^h c.
\]
For $r\in\{\mathrm{lat},\mathrm{geo}\}$,
\[
 \mathcal A^{r}_{\theta,h}
 \begin{pmatrix}u\\ \kappa c\end{pmatrix}
 =
 \begin{pmatrix}
    (-(D_{\theta,e}^{r})^2u_e)_e\\[0.4em]
    \displaystyle
    \kappa^{-1}
    \sum_{e\in\mathscr E_{\mathrm{cell}}}
    \mathbf1^*\partial_{\nu,\theta}^{r}u_e
 \end{pmatrix},
 \qquad \gamma_eu_e=\mathbf1c.
\]
\end{prop}

\begin{proof}
The trace maps $H^1(0,l_e)\to\mathbb C^2$ are continuous, so the
constraints in \eqref{eq:phase-form-domain} and
\eqref{eq:common-graph-form-domain} are closed.  The domains are dense in
$\mathfrak H_{g,h}$: prescribed endpoint values may be imposed in boundary
layers whose $L^2$-mass tends to zero.  Since $\theta$ ranges over a
bounded set and the edge lengths are uniformly bounded above and below, the
form norm
\[
 \bigl(\mathfrak a[u]+\Vert u\Vert_{L^2(G^h_{\mathrm{cell}})}^2
 +|\beta|^2\bigr)^{1/2}
\]
is uniformly equivalent to the product $H^1\oplus\mathbb C$-norm on the
corresponding domain.  The representation theorem gives the self-adjoint
operators.  Their actions follow from the one-dimensional Green identity.  In
the phase-corrected case the endpoint test vector is $a_e^hd$; in the two
common-value cases it is $\mathbf1d$.
\end{proof}

The incidence phases also appear in the endpoint Dirichlet-to-Neumann
description. If $M^{\mathrm{geo}}_{e,\theta}(z)$ maps the two endpoint
values of a homogeneous edge solution to its outgoing magnetic derivatives,
the contribution of edge $e$ to the scalar vertex equation is $(a_e^h)^*M^{\mathrm{geo}}_{e,\theta}(z)a_e^h.$
In general this differs from
$\mathbf1^*M^{\mathrm{geo}}_{e,\theta}(z)\mathbf1$; the two endpoint
phases must therefore be retained separately.

\subsubsection{The exact ground-mode graph and the resonant reduction}

Let $P^h_{0,\theta}$ be the orthogonal projection in $\mathcal H_e$ onto
the transverse magnetic ground modes on all edge cylinders.  The range of
$P^h_{0,\theta}$, together with the scalar vertex component, is precisely
the range of $\Phi^g_{\theta,h}$.

\begin{prop}[Ground-mode decomposition]\label{prop:ground-mode-decomposition}
There is a self-adjoint operator $\mathcal A^\perp_{\theta,h}$ in
$(I-P^h_{0,\theta})\mathcal H_e$ such that, under the orthogonal
decomposition of $\mathcal H_e\oplus\mathbb C$ into its ground-mode and
excited-mode sectors,
\begin{equation}\label{eq:reduced-thin-direct-sum}
 \mathcal A^{\mathrm{eff}}_{\theta,h}
 =
 \Phi^g_{\theta,h}\mathcal A^{\mathrm{ph}}_{\theta,h}
 (\Phi^g_{\theta,h})^*
 \oplus \mathcal A^\perp_{\theta,h}.
\end{equation}
Moreover, $\mathcal A^\perp_{\theta,h}\geq ch^{-2}I$
with $c>0$ independent of $h$ and $\theta$.  Consequently,
\begin{equation}\label{eq:reduced-to-phase-resolvent}
 \left\Vert
   (\mathcal A^{\mathrm{eff}}_{\theta,h}-z)^{-1}
   -
   \Phi^g_{\theta,h}
   (\mathcal A^{\mathrm{ph}}_{\theta,h}-z)^{-1}
   (\Phi^g_{\theta,h})^*
 \right\Vert
 \leq Ch^2,
 \qquad z\in K.
\end{equation}
\end{prop}

\begin{proof}
Remove the transverse magnetic phase on each cylinder and expand
\[
 u_e(x,y)=\sum_{n\geq0}u_{e,n}(x)\chi^h_{e,n}(y),
\]
where $\{\chi^h_{e,n}\}_{n\geq0}$ is the orthonormal Neumann eigenbasis of
$h\omega_e$,
\[
 -\Delta_y\chi^h_{e,n}=h^{-2}\mu_{e,n}\chi^h_{e,n},
 \qquad
 \mu_{e,0}=0,\qquad \inf_e\mu_{e,1}>0.
\]
The product geometry, the lateral Neumann condition and the magnetic
differential expression preserve the transverse index.

The domain condition of $\mathcal A^{\mathrm{eff}}_{\theta,h}$ requires the
endpoint trace to lie in the retained boundary channel.  By
\eqref{eq:graph-scalar-normalisation}, this means
\[
 \gamma_eu_{e,0}=a_e^hc,\qquad
 u_{e,n}(0)=u_{e,n}(l_e)=0\quad(n\geq1),\qquad
 \beta=\kappa c.
\]
In particular, the form domain in \eqref{eq:reduced-effective-form}
splits orthogonally into the ground-mode sector and the direct sum of the spaces $H^1_0(0,l_e)$ corresponding to
$n\geq1$, with finite total form energy.  On this decomposition the reduced form is
\[
 \mathfrak a^{\mathrm{ph}}_{\theta,h}[(u_0,\beta)]
 +
 \sum_{e\in\mathscr E_{\mathrm{cell}}}\sum_{n\geq1}
 \int_0^{l_e}
 \left(
   |D_{\theta,e}^{\mathrm{geo}}u_{e,n}|^2
   +h^{-2}\mu_{e,n}|u_{e,n}|^2
 \right)\,dx.
\]
The excited modes have zero endpoint trace, and their contribution to the
scalar flux component vanishes by transverse orthogonality.  Hence the first
representation theorem gives the operator direct sum
\eqref{eq:reduced-thin-direct-sum}.  The second summand is the direct sum of
the longitudinal magnetic Dirichlet operators shifted by
$h^{-2}\mu_{e,n}$; therefore
$\mathcal A^\perp_{\theta,h}\geq ch^{-2}I$, uniformly in $h$ and
$\theta$.  The resolvent estimate follows from the spectral theorem.
\end{proof}

Combining Proposition~\ref{prop:ground-mode-decomposition} with
Theorems~\ref{thm:R-approx} and~\ref{thm:eff-fibre} gives the first graph
approximation.

\begin{theorem}[Phase-corrected fibre approximation]\label{thm:phase-fibre}
There exist $C>0$ and $h_0>0$, independent of
$\theta\in Q'$ and $z\in K$, such that, for $0<h<h_0$,
\begin{equation}\label{eq:phase-fibre-estimate}
 \left\Vert
   (\mathcal A^\theta_{0,\mathbb I}-z)^{-1}
   -
   \Phi_{\theta,h}
   (\mathcal A^{\mathrm{ph}}_{\theta,h}-z)^{-1}
   \Phi_{\theta,h}^*
 \right\Vert_{\mathcal H\to\mathcal H}
 \leq Ch.
\end{equation}
\end{theorem}

\begin{proof}
By Theorem~\ref{thm:eff-fibre}, the lifted reduced resolvent is
$\mathcal R_{\theta,h}(z)$.  Proposition~\ref{prop:ground-mode-decomposition}
and the identity
$\Phi_{\theta,h}=\mathfrak F^*\bJ_\theta\Phi^g_{\theta,h}$
replace it by the phase-corrected graph resolvent with error $O(h^2)$.
Theorem~\ref{thm:R-approx} contributes $O(h)$.
\end{proof}

\subsubsection{The lattice graph and the geometric collapse}

Choose real phases $\gamma^h_{e,\sigma}:=-\theta\cdot x^h_{e,\sigma},$ $\alpha^h_{e,\sigma}=e^{i\gamma^h_{e,\sigma}},$ and interpolate them affinely on $[0,l_e]$:
\[
   \chi^h_e(x)
   :=\left(1-\frac{x}{l_e}\right)\gamma^h_{e,-}
     +\frac{x}{l_e}\gamma^h_{e,+}.
\]
Then
\begin{equation}\label{eq:chi-bounds}
   \Vert\chi_e^h\Vert_{L^\infty(0,l_e)}
   +\Vert(\chi_e^h)'\Vert_{L^\infty(0,l_e)}
   \leq Ch^{2/3}.
\end{equation}
Define the unitary operator $W_{\theta,h}$ in $\mathfrak H_{g,h}$ by
\[
   (W_{\theta,h}u)_e(x):=e^{i\chi_e^h(x)}u_e(x),
   \qquad W_{\theta,h}\beta:=\beta.
\]
It maps $\mathcal V_h^0$ onto
$\mathcal V^{\mathrm{ph}}_{\theta,h}$, and
\begin{equation}\label{eq:W-close-identity}
   \Vert W_{\theta,h}-I\Vert_{\mathfrak H_{g,h}\to\mathfrak H_{g,h}}
   \leq Ch^{2/3}.
\end{equation}
The centre-line identity
$n_e+x^h_{e,+}-x^h_{e,-}=l_et_e$ gives
\begin{equation}\label{eq:exact-lattice-gauge}
 \theta_e^{\mathrm{geo}}+(\chi_e^h)'
 =\theta\cdot t_e-
   \frac{\theta\cdot(x^h_{e,+}-x^h_{e,-})}{l_e}
 =\frac{\theta\cdot n_e}{l_e}
 =\theta_{e,h}^{\mathrm{lat}}.
\end{equation}
Hence
\begin{equation}\label{eq:exact-form-gauge}
 \mathfrak a^{\mathrm{ph}}_{\theta,h}
 [W_{\theta,h}(u,\beta),W_{\theta,h}(v,\gamma)]
 =
 \mathfrak a^{\mathrm{lat}}_{\theta,h}
 [(u,\beta),(v,\gamma)].
\end{equation}
In particular, one has
\begin{equation}\label{eq:phase-lattice-unitary}
 \mathcal A^{\mathrm{ph}}_{\theta,h}
 =W_{\theta,h}\mathcal A^{\mathrm{lat}}_{\theta,h}W_{\theta,h}^*.
\end{equation}
This is the standard Floquet gauge: the total magnetic phase on edge $e$
is $\theta\cdot n_e$.

Set
\begin{equation}\label{eq:lattice-identification}
   \Phi^{\mathrm{lat}}_{\theta,h}
   :=\Phi_{\theta,h}W_{\theta,h}.
\end{equation}
Theorem~\ref{thm:phase-fibre} and
\eqref{eq:phase-lattice-unitary} give the following fibre estimate.

\begin{theorem}[Lattice graph fibre approximation]\label{thm:lattice-fibre}
There exist $C>0$ and $h_0>0$, independent of
$\theta\in Q'$ and $z\in K$, such that, for $0<h<h_0$,
\begin{equation}\label{eq:lattice-fibre-estimate}
 \left\Vert
   (\mathcal A^\theta_{0,\mathbb I}-z)^{-1}
   -
   \Phi^{\mathrm{lat}}_{\theta,h}
   (\mathcal A^{\mathrm{lat}}_{\theta,h}-z)^{-1}
   (\Phi^{\mathrm{lat}}_{\theta,h})^*
 \right\Vert_{\mathcal H\to\mathcal H}
 \leq Ch.
\end{equation}
\end{theorem}

The ambient-geometric graph uses the same common endpoint values but retains
$\theta\cdot t_e$ in the magnetic derivative.  From
\eqref{eq:edge-centreline-geometry},
\begin{equation}\label{eq:lattice-geo-potential-difference}
 \left|
   \theta_{e,h}^{\mathrm{lat}}-\theta_e^{\mathrm{geo}}
 \right|
 =\frac{|\theta\cdot(x^h_{e,+}-x^h_{e,-})|}{l_e}
 \leq Ch^{2/3}.
\end{equation}

\begin{theorem}[Comparison with the ambient-geometric graph]
\label{thm:lattice-to-geometric}
There exist $C>0$ and $h_0>0$, independent of
$\theta\in Q'$ and $z\in K$, such that, for $0<h<h_0$,
\begin{align}
 \left\Vert
   (\mathcal A^{\mathrm{lat}}_{\theta,h}-z)^{-1}
   -(\mathcal A^{\mathrm{geo}}_{\theta,h}-z)^{-1}
 \right\Vert
 &\leq Ch^{2/3},                                      \label{eq:lattice-geo-resolvent}\\
 \left\Vert
   (\mathcal A^{\mathrm{ph}}_{\theta,h}-z)^{-1}
   -(\mathcal A^{\mathrm{geo}}_{\theta,h}-z)^{-1}
 \right\Vert
 &\leq Ch^{2/3}.                                      \label{eq:phase-geo-resolvent}
\end{align}
For $\theta=0$, all three graph operators coincide.
\end{theorem}

\begin{proof}
The lattice and ambient-geometric forms have the same domain.  By
\eqref{eq:lattice-geo-potential-difference},
\[
 \left|
  \mathfrak a^{\mathrm{lat}}_{\theta,h}[U,V]
  -\mathfrak a^{\mathrm{geo}}_{\theta,h}[U,V]
 \right|
 \leq Ch^{2/3}\Vert U\Vert_{\mathcal V_h^0}
                    \Vert V\Vert_{\mathcal V_h^0},
\]
where $\mathcal V_h^0$ carries either of the uniformly equivalent form
norms.  For $j\in\{\mathrm{lat},\mathrm{geo}\}$, the operator induced by
$\mathfrak a^j_{\theta,h}-z(\cdot,\cdot)$ has a uniformly bounded inverse from
$(\mathcal V_h^0)^*$ to $\mathcal V_h^0$.  Indeed, if
$\mathcal A^j_{\theta,h}$ is the associated non-negative operator, then
spectral calculus gives
\[
 \sup_{\lambda\geq0}
 \frac{1+\lambda}{|\lambda-z|}
 \leq C_K,
 \qquad z\in K,
\]
and the form norms are uniformly equivalent.  The resolvent identity in
$\mathcal B((\mathcal V_h^0)^*,\mathcal V_h^0)$ therefore proves
\eqref{eq:lattice-geo-resolvent}.

By \eqref{eq:phase-lattice-unitary}, one has
\[
 (\mathcal A^{\mathrm{ph}}_{\theta,h}-z)^{-1}
 =W_{\theta,h}
  (\mathcal A^{\mathrm{lat}}_{\theta,h}-z)^{-1}
  W_{\theta,h}^*.
\]
Hence \eqref{eq:phase-geo-resolvent} follows from
\eqref{eq:lattice-geo-resolvent}, \eqref{eq:W-close-identity}, and the
uniform graph resolvent bounds.  If $\theta=0$, the endpoint phases and both
magnetic potentials vanish.
\end{proof}

\begin{theorem}[Ambient-geometric fibre approximation]\label{thm:graph-fibre}
There exist $C>0$ and $h_0>0$, independent of
$\theta\in Q'$ and $z\in K$, such that, for $0<h<h_0$,
\begin{equation}\label{eq:geometric-fibre-estimate}
 \left\Vert
   (\mathcal A^\theta_{0,\mathbb I}-z)^{-1}
   -
   \Phi_{\theta,h}
   (\mathcal A^{\mathrm{geo}}_{\theta,h}-z)^{-1}
   \Phi_{\theta,h}^*
 \right\Vert_{\mathcal H\to\mathcal H}
 \leq Ch^{2/3}.
\end{equation}
For $\theta=0$, the right-hand side may be replaced by $Ch$.
\end{theorem}

\begin{proof}
Combine Theorems~\ref{thm:phase-fibre} and~\ref{thm:lattice-to-geometric}.  At $\theta=0$, the graph comparison error
vanishes.
\end{proof}

\begin{rmk}[The finite-$h$ and limiting graphs]
The edge lengths in the graph fibres are the exact finite-thickness
centre-line lengths $l_e=l_e^h$.  Relation
\eqref{eq:edge-centreline-geometry} implies
$l_e^h=|n_e|+O(h^{2/3})$ and
$t_e^h=n_e/|n_e|+O(h^{2/3})$ for the straight cell-to-cell connectors
considered here.  If the centre-line lengths are independent of $h$, the
lattice graph in Theorem~\ref{thm:lattice-fibre} is already a fixed graph.
Otherwise, comparison with the fixed graph having edge lengths $|n_e|$
requires a separate quasi-unitary identification of metric graphs with
varying edge lengths.  Such a comparison is not used in the present paper:
the principal $O(h)$ theorem is formulated for the exact finite-$h$
periodic graph dictated by the thin geometry.
\end{rmk}

\subsection{Spectral bands and gap stability}
\label{subsec:spectral-bands}

We now deduce quantitative fibre spectral approximation and gap stability
from the uniform norm-resolvent estimates. For $\theta\in Q'$, let
\[
 \Sigma_h(\theta):=\sigma(\mathcal A^\theta_{0,\mathbb I}),
 \qquad
 \Sigma_h^{\mathrm{lat}}(\theta)
 :=\sigma(\mathcal A^{\mathrm{lat}}_{\theta,h}),
 \qquad
 \Sigma_h^{\mathrm{geo}}(\theta)
 :=\sigma(\mathcal A^{\mathrm{geo}}_{\theta,h}).
\]
All three fibre operators have compact resolvent.  The lattice and geometric
forms have a common $\theta$-independent domain and depend polynomially on
$\theta$, so they form analytic families of type (B).  The physical fibre
forms are defined on the fixed space $H^1(X_{\mathrm{cell}}^h)$ and have the
same property.  We regard $Q'=[-\pi,\pi)^3$ as a fundamental representative
of the compact Brillouin torus for the physical and lattice families.  Their
fibres at opposite faces are related by the standard Floquet gauge.  The
ambient-geometric formula and the estimates above extend continuously from
$Q'$ to the closed cube $\overline Q'=[-\pi,\pi]^3$, but its values on
opposite faces need not be gauge-equivalent.  Consequently, the fibre
eigenvalues depend continuously on their respective parameter spaces, and the
three operator fields are measurable.  The physical and lattice families are
Floquet families of periodic operators.  The ambient-geometric family is used
only as a decomposable comparison family; it need not be the Floquet family
of a local periodic metric-graph differential operator.

\begin{cor}[Uniform fibre spectral approximation]\label{cor:fibre-spectrum}
For every $E>0$ there exists $C_E>0$ such that, for all sufficiently
small $h$ and all $\theta\in Q'$,
\begin{align}
 \Sigma_h(\theta)\cap[0,E]
 &\subset
 \bigl(\Sigma_h^{\mathrm{lat}}(\theta)\bigr)_{C_Eh},
 &
 \Sigma_h^{\mathrm{lat}}(\theta)\cap[0,E]
 &\subset
 \bigl(\Sigma_h(\theta)\bigr)_{C_Eh},                 \label{eq:lattice-spectral-inclusions}\\
 \Sigma_h(\theta)\cap[0,E]
 &\subset
 \bigl(\Sigma_h^{\mathrm{geo}}(\theta)\bigr)_{C_Eh^{2/3}},
 &
 \Sigma_h^{\mathrm{geo}}(\theta)\cap[0,E]
 &\subset
 \bigl(\Sigma_h(\theta)\bigr)_{C_Eh^{2/3}}.          \label{eq:geo-spectral-inclusions}
\end{align}
Here $(S)_\delta$ denotes the open $\delta$-neighbourhood of
$S\subset\mathbb R$.  At $\theta=0$, the exponent $2/3$ in the second
line may be replaced by $1$.
\end{cor}

\begin{proof}
Apply Theorems~\ref{thm:lattice-fibre} and~\ref{thm:graph-fibre} at one fixed
non-real point, say $z=i$; the preceding arguments apply with the conjugation-invariant set
$K=\{i,-i\}$.  The physical resolvent and each lifted graph resolvent are
compact normal operators.  The non-zero spectrum of a lifted graph resolvent
is exactly $\{(\lambda-i)^{-1}:\lambda\in\Sigma(\theta)\}$, while zero is
the only additional spectral point.  For normal operators the Hausdorff
distance between the spectra is bounded by the operator-norm distance.  If
$\lambda\in[0,E]$, then $|(\lambda-i)^{-1}|\geq(1+E^2)^{-1/2}$.  Thus,
for sufficiently small $h$, a spectral point paired with
$(\lambda-i)^{-1}$ cannot be the additional point zero.  The inverse map
$\mu\mapsto i+\mu^{-1}$ is uniformly Lipschitz on the resulting compact
set away from zero, and the stated inclusions follow.
\end{proof}

For $r\in\{\mathrm{lat},\mathrm{geo}\}$, define the physical and
lattice band spectra, and the ambient-geometric comparison spectral set, below
energy $E$ by
\[
 \mathfrak S_h(E)
 :=\overline{\bigcup_{\theta\in Q'}
       \bigl(\Sigma_h(\theta)\cap[0,E]\bigr)},
 \qquad
 \mathfrak S_h^{r}(E)
 :=\overline{\bigcup_{\theta\in Q'}
       \bigl(\Sigma_h^{r}(\theta)\cap[0,E]\bigr)}.
\]

\begin{cor}[Band and gap stability]\label{cor:band-gap}
\begin{samepage}
For every $E>0$ there exists $C_E>0$ such that, for all sufficiently
small $h$,
\begin{align*}
 \mathfrak S_h(E)
 &\subset(\mathfrak S_h^{\mathrm{lat}}(E+C_Eh))_{C_Eh},
 &
 \mathfrak S_h^{\mathrm{lat}}(E)
 &\subset(\mathfrak S_h(E+C_Eh))_{C_Eh},\\
 \mathfrak S_h(E)
 &\subset(\mathfrak S_h^{\mathrm{geo}}(E+C_Eh^{2/3}))_{C_Eh^{2/3}},
 &
 \mathfrak S_h^{\mathrm{geo}}(E)
 &\subset(\mathfrak S_h(E+C_Eh^{2/3}))_{C_Eh^{2/3}}.
\end{align*}
\end{samepage}
If $(a,b)\Subset(0,E)$ is an open gap of the lattice graph band spectrum
and $b-a>2C_Eh$, then
\[
   (a+C_Eh,b-C_Eh)
\]
is a gap of the physical periodic operator.  If $(a,b)$ is an open gap of
the ambient-geometric comparison spectral set and
$b-a>2C_Eh^{2/3}$, the analogous reduced interval is a physical gap.

More precisely, fix $r\in\{\mathrm{lat},\mathrm{geo}\}$.  For
$r=\mathrm{lat}$, let $Q_0$ be a compact subset of the Brillouin
torus.  For $r=\mathrm{geo}$, let $Q_0$ be a compact subset of the
closed parameter cube $\overline Q'$.  Suppose that there are
$h_0>0$ and $\delta>0$ such that, for every $0<h<h_0$, an effective
eigenvalue branch $\lambda_{j,h}^r(\theta)$ is simple on $Q_0$,
takes values in $[0,E]$, and satisfies
\[
 \inf_{0<h<h_0}\inf_{\theta\in Q_0}
 \operatorname{dist}\left(
   \lambda_{j,h}^r(\theta),
   \sigma(\mathcal A_{\theta,h}^r)
   \setminus\{\lambda_{j,h}^r(\theta)\}
 \right)\geq\delta.
\]
Then, for all sufficiently small $h$, there is a unique physical fibre
eigenvalue $\lambda_j^h(\theta)$ in the $\delta/2$-neighbourhood of
$\lambda_{j,h}^r(\theta)$, and
\[
 \sup_{\theta\in Q_0}
 |\lambda_j^h(\theta)-\lambda_{j,h}^r(\theta)|
 \leq
 \begin{cases}
   C_{E,\delta}h,&r=\mathrm{lat},\\
   C_{E,\delta}h^{2/3},&r=\mathrm{geo}.
 \end{cases}
\]
\end{cor}

\begin{proof}
For a spectral point in $[0,E]$, the comparison point supplied by
Corollary~\ref{cor:fibre-spectrum} has energy at most $E+C_Eh$ in the
lattice case and $E+C_Eh^{2/3}$ in the geometric case. Take the union
over $Q'$ and then the closure. After enlarging $C_E$, this gives
the stated inclusions with open neighbourhoods. The gap statements follow
directly from Corollary~\ref{cor:fibre-spectrum}, since a physical spectral
point in the reduced interval would have to lie within the stated error of
the effective spectral set outside $(a,b)$.

For the isolated branch, put $\alpha_{\mathrm{lat}}=1$,
$\alpha_{\mathrm{geo}}=2/3$, and write
$\Psi_{\theta,h}=\Phi^{\mathrm{lat}}_{\theta,h}$ if
$r=\mathrm{lat}$, and $\Psi_{\theta,h}=\Phi_{\theta,h}$ if
$r=\mathrm{geo}$. Consider the compact normal operators
\[
 T_{\theta,h}:=(\mathcal A^\theta_{0,\mathbb I}-i)^{-1},
 \qquad
 S_{\theta,h}:=\Psi_{\theta,h}
 (\mathcal A^r_{\theta,h}-i)^{-1}\Psi_{\theta,h}^*.
\]
Theorems~\ref{thm:lattice-fibre} and~\ref{thm:graph-fibre} give
$\Vert T_{\theta,h}-S_{\theta,h}\Vert\leq Ch^{\alpha_r}$.
The energy bound and the separation hypothesis imply that
$\mu_{\theta,h}:=(\lambda^r_{j,h}(\theta)-i)^{-1}$ is separated
from the rest of $\sigma(S_{\theta,h})$, including zero, by a constant
$\gamma>0$ depending only on $E$ and $\delta$.

On the circle $|\zeta-\mu_{\theta,h}|=\gamma/3$, the resolvent of
$S_{\theta,h}$ is uniformly bounded. For sufficiently small $h$,
a Neumann series gives the same bound for the resolvent of
$T_{\theta,h}$. The resolvent identity and integration around this
circle show that the corresponding Riesz projections differ in norm by
$O(h^{\alpha_r})$. The effective projection has rank one, so the
physical projection also has rank one for sufficiently small $h$.

Apply Corollary~\ref{cor:fibre-spectrum} with energy bound $E+\delta$.
It supplies a physical eigenvalue within $C_{E,\delta}h^{\alpha_r}$
of $\lambda^r_{j,h}(\theta)$. Conversely, every physical eigenvalue
in its $\delta/2$-neighbourhood must be within that error of the same
effective eigenvalue, since all other effective eigenvalues are at least
$\delta$ away. For sufficiently small $h$, the resolvent images of
all these physical eigenvalues lie inside the above circle. The rank-one
property therefore gives uniqueness, and the spectral inclusion gives the
stated error bound. All constants are uniform for $\theta\in Q_0$.
\end{proof}

\subsection{Reconstruction of the global effective operators}

We finally reassemble the fibre estimates.  Let
\[
 U:L^2(X^h)\longrightarrow
 \int_{Q'}^\oplus L^2(X^h_{\mathrm{cell}})\,d\theta
\]
be the Floquet--Gelfand transform on the thin structure.

We first specify the global graph space.  The periodic metric
graph $\mathscr G_h$ has vertices $m\in\mathbb Z^3$.  For every
$e\in\mathscr E_{\mathrm{cell}}$ and $m\in\mathbb Z^3$, it has an
oriented edge $e_m$ of length $l_e$, joining $m$ to $m+n_e$.  We set
\[
 \mathcal H_{g,h}
 :=
 \left(
 \bigoplus_{m\in\mathbb Z^3}
 \bigoplus_{e\in\mathscr E_{\mathrm{cell}}}L^2(0,l_e)
 \right)
 \oplus\ell^2(\mathbb Z^3).
\]
For compactly supported graph data $(u,\beta)$, define
\begin{align*}
 (U_gu)_e(x,\theta)
 :=(2\pi)^{-3/2}e^{-i(\theta\cdot n_e)x/l_e}
   \sum_{m\in\mathbb Z^3}e^{-i\theta\cdot m}u_{e,m}(x),\qquad
 (U_g\beta)(\theta)
 :=(2\pi)^{-3/2}
   \sum_{m\in\mathbb Z^3}e^{-i\theta\cdot m}\beta_m.
\end{align*}
This extends to a unitary map
\[
 U_g:\mathcal H_{g,h}\longrightarrow
 \int_{Q'}^\oplus\mathfrak H_{g,h}\,d\theta.
\]
The affine phase in the edge component converts the endpoint relation between
$m$ and $m+n_e$ into common fibre endpoint values and produces the
magnetic derivative
$d/dx+i(\theta\cdot n_e)/l_e$.  Thus
$\mathcal A^{\mathrm{lat}}_{\theta,h}$ is precisely the Floquet fibre of
the global periodic metric-graph operator.

The fields
$\theta\mapsto\Phi_{\theta,h}$,
$\theta\mapsto W_{\theta,h}$,
$\theta\mapsto\mathcal A^{\mathrm{lat}}_{\theta,h}$, and
$\theta\mapsto\mathcal A^{\mathrm{geo}}_{\theta,h}$ are measurable.  On
each edge, $\Phi_{\theta,h}$ and $W_{\theta,h}$ are given by
multiplication by explicit measurable phase factors.  On the vertex component our phase convention gives
\[
 J_\theta^*1
 =\kappa^{-1}h^{-1}e^{-i\theta\cdot x}
 \quad\hbox{in }X^h_{\mathrm{cell},v},
\]
so this field is measurable as well.  For the graph operators, measurability
follows from the common form domain and the explicit polynomial dependence on
$\theta$.

Define the two global identification operators
 \begin{equation*}
 \Theta_h^{\mathrm{lat}}
 :=U^*\left(\int_{Q'}^\oplus
       \Phi^{\mathrm{lat}}_{\theta,h}\,d\theta\right)U_g,\qquad\quad\Theta_h^{\mathrm{geo}}
 :=U^*\left(\int_{Q'}^\oplus
       \Phi_{\theta,h}\,d\theta\right)U_g,
\end{equation*}
and the global graph operators
\begin{equation*}
\mathcal A_h^{\mathrm{lat}}
 :=U_g^*\left(\int_{Q'}^\oplus
       \mathcal A^{\mathrm{lat}}_{\theta,h}\,d\theta\right)U_g,\qquad\quad
 \mathcal A_h^{\mathrm{geo}}
 :=U_g^*\left(\int_{Q'}^\oplus
       \mathcal A^{\mathrm{geo}}_{\theta,h}\,d\theta\right)U_g.
\end{equation*}
Here $\mathcal A_h^{\mathrm{lat}}$ is the periodic metric-graph
operator, whereas $\mathcal A_h^{\mathrm{geo}}$ is the decomposable
comparison operator described in Subsection~\ref{subsec:spectral-bands}.

\begin{theorem}[Global generalised norm-resolvent estimates]
\label{thm:global-generalised}
There exists $C>0$, independent of $h$ and $z\in K$, such that,
for all sufficiently small $h$,
\begin{align}
 \left\Vert
  (\mathcal A^h-z)^{-1}
  -\Theta_h^{\mathrm{lat}}
   (\mathcal A_h^{\mathrm{lat}}-z)^{-1}
   (\Theta_h^{\mathrm{lat}})^*
 \right\Vert_{L^2(X^h)\to L^2(X^h)}
 &\leq Ch,                                            \label{eq:global-lattice}\\
 \left\Vert
  (\mathcal A^h-z)^{-1}
  -\Theta_h^{\mathrm{geo}}
   (\mathcal A_h^{\mathrm{geo}}-z)^{-1}
   (\Theta_h^{\mathrm{geo}})^*
 \right\Vert_{L^2(X^h)\to L^2(X^h)}
 &\leq Ch^{2/3}.                                      \label{eq:global-geometric}
\end{align}
\end{theorem}

\begin{proof}
Take direct integrals in Theorems~\ref{thm:lattice-fibre} and~\ref{thm:graph-fibre}.  The norm of a decomposable operator is the essential
supremum of its fibre norms, and all constants are uniform in $\theta$.
\end{proof}

Estimate \eqref{eq:global-lattice} is the global periodic resonant
thin-to-graph theorem.  Estimate \eqref{eq:global-geometric} quantifies the
additional ambient-geometric collapse.  Since the physical and graph
operators act in different Hilbert spaces, both are generalised
norm-resolvent estimates, with the corresponding identification operator
$\Theta_h^{\mathrm{lat}}$ or $\Theta_h^{\mathrm{geo}}$.

\section{A simple-cubic spectral example}
\label{sec:simple-cubic}

Exner and Post \cite[Section~9.7 and Theorem~9.6]{exner2005convergence}
derived the resonant dispersion relation for periodic Cayley graphs and
proved the existence of infinitely many gaps. Their calculation includes the
simple-cubic lattice. We reproduce it in the present notation and determine
all band endpoints and the dependence of the gaps on the vertex mass.

We use the finite-$h$ lattice graph from
Theorem~\ref{thm:lattice-fibre}. For simple-cubic families satisfying
Assumption~\ref{ass:standing-geometry}, the uniform resolvent estimate gives
$O(h)$ control of gap persistence, refining the qualitative conclusion of
\cite[Theorem~9.2]{exner2005convergence}.

Assume that the quotient graph has one vertex and three oriented edges
$e_j$, $j=1,2,3$, with lattice shifts equal to the coordinate vectors of
$\mathbb Z^3$.  Suppose that their centre-line lengths have the common value
$L_h>0$.  Put
\[
Q'=[-\pi,\pi)^3,\qquad
 \theta=(\theta_1,\theta_2,\theta_3)\in Q',
 \qquad
 \kappa^2=|X_v^0|>0.
\]
The simple-cubic fibre
$\mathcal A_{\theta,h}^{\mathrm{cub}}$ is the corresponding special case of
$\mathcal A_{\theta,h}^{\mathrm{lat}}$, acting in
\[
 \mathfrak H_{g,h}^{\mathrm{cub}}
 :=
 \bigoplus_{j=1}^3L^2(0,L_h)\oplus\mathbb C.
\]
With $D_{\theta,j}:=d/dx+i\theta_j/L_h$, the eigenvalue equation for a
vector $(u,\kappa c)$ is
\begin{align}
 -D_{\theta,j}^2u_j&=\lambda u_j,
 &&0<x<L_h,\quad j=1,2,3,                     \label{eq:cubic-edge}\\
 u_j(0)=u_j(L_h)&=c,
 &&j=1,2,3,                                   \label{eq:cubic-continuity}\\
 \sum_{j=1}^3\bigl[-D_{\theta,j}u_j(0)
              +D_{\theta,j}u_j(L_h)\bigr]
 &=\lambda\kappa^2c.                          \label{eq:cubic-balance}
\end{align}

\subsection{Dispersion relation}

Set $\lambda=k^2$, $q=kL_h$, and $\eta_h:=\kappa^2/L_h.$ For $\sin q\neq0$, gauge away the constant magnetic potential by
$v_j(x)=e^{i\theta_jx/L_h}u_j(x)$.  Then
\[
 -v_j''=k^2v_j,\qquad
 v_j(0)=c,\qquad
 v_j(L_h)=e^{i\theta_j}c,
\]
and hence
\[
 v_j(x)
 =
 c\,
 \frac{\sin(k(L_h-x))+e^{i\theta_j}\sin(kx)}
      {\sin(kL_h)}.
\]
A direct differentiation gives
\[
 -v_j'(0)+e^{-i\theta_j}v_j'(L_h)
 =
 \frac{2kc}{\sin(kL_h)}
 \bigl(\cos(kL_h)-\cos\theta_j\bigr).
\]

\begin{prop}[Simple-cubic dispersion relation]
\label{prop:cubic-dispersion}
Let $k>0$ and $\sin(kL_h)\neq0$.  Then $k^2$ is an eigenvalue of
$\mathcal A_{\theta,h}^{\mathrm{cub}}$ if and only if
\begin{equation}
 \sum_{j=1}^3\cos\theta_j
 =
 F_{\eta_h}(q),
 \qquad
 F_{\eta}(q):=
 3\cos q-\frac{\eta q}{2}\sin q.
 \label{eq:cubic-dispersion}
\end{equation}
Every eigenfunction corresponding to such an eigenvalue has a non-zero
vertex coefficient.
\end{prop}

\begin{proof}
Substitution of the displayed endpoint derivatives into
\eqref{eq:cubic-balance}, followed by division by $2k c/\sin q$, gives
\eqref{eq:cubic-dispersion}.  If $c=0$ and $\sin q\neq0$, the endpoint
conditions force the edge solution to vanish.
\end{proof}

At $q=0$, the value $\lambda=0$ occurs for $\theta=0$, with
constant edge functions and a non-zero vertex coefficient; this is the left
endpoint of the first dispersive band.

At the Dirichlet levels $q=n\pi$, $n\geq1$, a non-zero vertex
coefficient would require $e^{i\theta_j}=(-1)^n$ for all $j$.  The
left-hand side of the vertex balance would then vanish, whereas its right-hand
side would equal $k^2\kappa^2c$; hence necessarily $c=0$.  The fibre
therefore has the flat-band eigenfunctions
\[
 u_j(x)
 =
 A_j e^{-i\theta_jx/L_h}
 \sin\frac{n\pi x}{L_h},
 \qquad c=0,
\]
where
\begin{equation}
 \sum_{j=1}^3
 A_j\bigl((-1)^ne^{-i\theta_j}-1\bigr)=0.
 \label{eq:cubic-flat-constraint}
\end{equation}
The multiplicity is two for generic $\theta$, and it is three when all
coefficients in \eqref{eq:cubic-flat-constraint} vanish.

\subsection{Exact bands and gaps}

As $\theta$ ranges over $Q'$, the quantity
$\sum_{j=1}^3\cos\theta_j$ fills the interval $[-3,3]$.  Consequently,
the dispersive spectrum is characterised by $|F_{\eta_h}(q)|\leq3.$
For $n\geq0$, let $t_{n,h}\in(0,\pi)$ be the unique solution of
\begin{equation}
 6\cot\frac{t_{n,h}}2
 =
 \eta_h(n\pi+t_{n,h}).
 \label{eq:cubic-tn}
\end{equation}

\begin{theorem}[Exact simple-cubic band-gap decomposition]
\label{thm:cubic-bands}
Let
\[
 \mathcal A_h^{\mathrm{cub}}
 :=
 \int_{Q'}^\oplus
 \mathcal A_{\theta,h}^{\mathrm{cub}}\,d\theta.
\]
Then
\begin{equation}
 \sigma(\mathcal A_h^{\mathrm{cub}})
 =
 \bigcup_{n=0}^{\infty}
 \left[
   \left(\frac{n\pi}{L_h}\right)^2,
   \left(\frac{n\pi+t_{n,h}}{L_h}\right)^2
 \right].
 \label{eq:cubic-spectrum}
\end{equation}
In particular,
\begin{equation}
 \mathcal G_{n,h}
 :=
 \left(
   \left(\frac{n\pi+t_{n,h}}{L_h}\right)^2,
   \left(\frac{(n+1)\pi}{L_h}\right)^2
 \right)
 \label{eq:cubic-gap}
\end{equation}
is a non-empty open gap for every $n\geq0$.  Its right endpoint is the
next Dirichlet flat-band level.
\end{theorem}

\begin{proof}
Write $q=n\pi+t$, $0\leq t\leq\pi$, and set
\[
 H_{n,h}(t)
 :=
 3\cos t-\frac{\eta_h(n\pi+t)}2\sin t.
\]
Then $F_{\eta_h}(n\pi+t)=(-1)^nH_{n,h}(t)$.  For $0<t<\pi$, one has $H_{n,h}(t)>-3$ if and only if $6\cot(t/2)>\eta_h(n\pi+t).$ The left-hand side is strictly decreasing from $+\infty$ to zero, whereas
the right-hand side is strictly increasing.  Hence
\eqref{eq:cubic-tn} has one solution, and
$H_{n,h}(t)<-3$ for $t_{n,h}<t<\pi$.

It remains to show that $H_{n,h}$ decreases from $3$ to $-3$ on
$[0,t_{n,h}]$.  This holds directly on $(0,\pi/2]$.  If
$\pi/2<t\leq t_{n,h}$, then
$\eta_h(n\pi+t)\leq6\cot(t/2)$, and therefore
\[
 H_{n,h}'(t)
 \leq
 -3\sin t-\frac{\eta_h}{2}\sin t
 +3\cot\frac t2\,|\cos t|<0,
\]
because $\cot(t/2)|\cos t|<\sin t$ for $\pi/2<t<\pi$.  Thus $|F_{\eta_h}(q)|\leq3$ on the half-open interval $[n\pi,(n+1)\pi)$ exactly for
$n\pi\leq q\leq n\pi+t_{n,h}$.  The Dirichlet levels supplied by
\eqref{eq:cubic-flat-constraint} complete the endpoints in
\eqref{eq:cubic-spectrum}.
\end{proof}

In the formal massless limit $\eta_h\downarrow0$, one has
$t_{n,h}\uparrow\pi$, and all gaps close.  Moreover, implicit
differentiation of \eqref{eq:cubic-tn} gives
\[
 \frac{\partial t_{n,h}}{\partial\eta_h}
 =
 -\frac{n\pi+t_{n,h}}
 {3\csc^2(t_{n,h}/2)+\eta_h}<0.
\]
Hence, at fixed edge length, every gap widens monotonically with the
resonant vertex mass.

\FloatBarrier
\begin{figure}[htbp]
 \centering
 \includegraphics[width=0.75\linewidth]{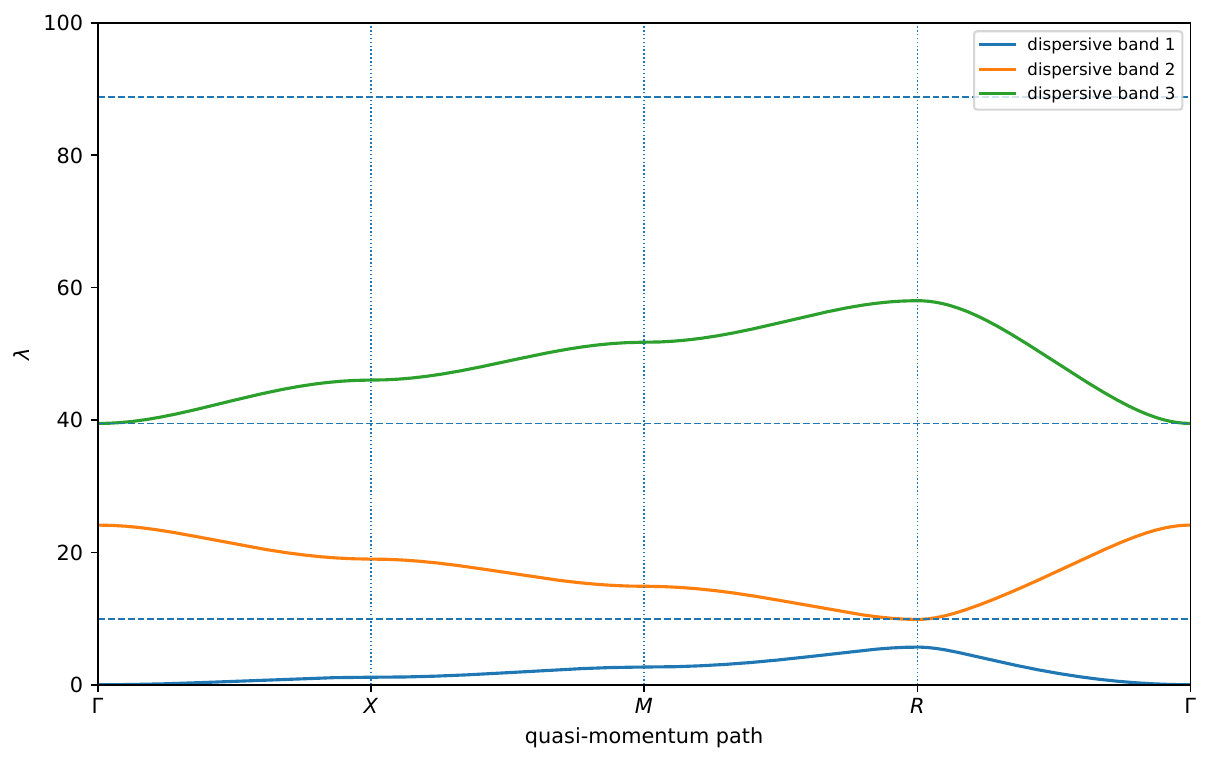}
 \caption{The first three dispersive bands of the simple-cubic lattice graph
 along the path $\Gamma-X-M-R-\Gamma$, where
 $\Gamma=(0,0,0)$, $X=(\pi,0,0)$,
 $M=(\pi,\pi,0)$, and $R=(\pi,\pi,\pi)$, with boundary
points understood modulo $2\pi$ on the Brillouin torus, for
 $L_h=1$ and $\kappa^2=1$.  The dashed horizontal lines are the
 first three Dirichlet flat bands.  The curves are obtained from the exact
 scalar equation \eqref{eq:cubic-dispersion}; no graph discretisation is
 involved.}
 \label{fig:simple-cubic-bands}
\end{figure}
\FloatBarrier

\subsection{Persistence in the thin structure}

Put
\[
 a_{n,h}
 :=
 \left(\frac{n\pi+t_{n,h}}{L_h}\right)^2,
 \qquad
 b_{n,h}
 :=
 \left(\frac{(n+1)\pi}{L_h}\right)^2.
\]

\begin{cor}[Persistence of the simple-cubic gaps]
\label{cor:cubic-gap-persistence}
Fix $E>0$, and let $C_E$ be the constant in
Corollary~\ref{cor:band-gap}.  For all sufficiently small $h$, if
$\overline{\mathcal G_{n,h}}\subset(0,E)$ and
$b_{n,h}-a_{n,h}>2C_Eh$, then
\begin{equation}
 (a_{n,h}+C_Eh,\ b_{n,h}-C_Eh)
 \subset\rho(\mathcal A^h).
 \label{eq:cubic-physical-gap}
\end{equation}
For every fixed $N$, choose $E$ above the limiting right endpoint of the
$N$-th gap.  Then the first $N$ simple-cubic gaps persist in the
physical thin structure for all sufficiently small $h$.
\end{cor}

\begin{proof}
The first assertion is Corollary~\ref{cor:band-gap} applied to the explicit
lattice-graph gap $\mathcal G_{n,h}$.  For the coordinate lattice shifts,
\eqref{eq:edge-centreline-geometry} gives
$L_h=1+O(h^{2/3})$.  Hence
$\eta_h\to\kappa^2$, and the unique solution of
\eqref{eq:cubic-tn} converges to a number $t_n\in(0,\pi)$.  The width
$b_{n,h}-a_{n,h}$ therefore converges, for every fixed $n$, to a strictly
positive limit.  The inequality in the first assertion consequently holds
for each fixed finite collection of gaps once $h$ is sufficiently small.
\end{proof}

\bibliographystyle{plainnat}
\bibliography{ref}
\end{document}